\documentclass[11pt]{article}
\usepackage[english]{babel}
\usepackage{amsmath,amssymb,diagrams}

\newcommand{\inj}{{\mathrm{in}}}
\newcommand{\Pow}{{\mathcal{P}}}
\newcommand{\Set}{{\mathbf{Set}}}
\newcommand{\FinSet}{{\mathbf{FinSet}}}
\newcommand{\Cat}{{\mathbf{Cat}}}
\newcommand{\cat}{{\mathbf{cat}}}
\newcommand{\N}{{\mathbb{N}}}
\newcommand{\A}{{\mathbf{A}}}
\newcommand{\B}{{\mathbf{B}}}

\newcommand{\C}{{\mathbf{C}}}

\newcommand{\Cl}{{\mathcal{C}}}
\newcommand{\D}{{\mathbf{D}}}
\newcommand{\E}{{\mathbf{E}}}
\newcommand{\F}{{\mathbf{F}}}
\newcommand{\Gra}{{\mathbb{G}}}
\newcommand{\Gbb}{{\mathbb{G}}}
\newcommand{\Gcal}{{\mathcal{G}}}
\newcommand{\PP}{{\mathbf{P}}}
\newcommand{\Se}{{\mathbf{S}}}
\newcommand{\X}{{\mathbf{X}}}
\newcommand{\XX}{{\mathfrak{X}}}
\newcommand{\Y}{{\mathbf{Y}}}
\newcommand{\Z}{{\mathbf{Z}}}
\newcommand{\Ze}{{\mathbb{Z}}}
\newcommand{\DDelta}{{\boldsymbol \Delta}}
\newcommand{\Fam}{{\mathrm{Fam}}}
\newcommand{\FFam}{\mathit{Fam}}
\newcommand{\id}{{\mathit{id}}}
\newcommand{\Id}{{\mathit{Id}}}
\newcommand{\gl}{{\mathsf{gl}}}
\newcommand{\Gl}{{\mathsf{Gl}}}
\newcommand{\Fib}{{\mathbf{Fib}}}
\newcommand{\Spl}{{\mathbf{Sp}}}
\newcommand{\Split}{{\mathit{Sp}}}
\newcommand{\Cart}{{\mathrm{Cart}}}
\newcommand{\cart}{{\mathrm{cart}}}
\newcommand{\cocart}{{\mathrm{cocart}}}
\newcommand{\op}{{\mathrm{op}}}
\newcommand{\Elts}{{\mathbf{Elts}}}
\newcommand{\Ob}{{\mathrm{Ob}}}
\newcommand{\Hom}{{\mathbf{Hom}}}
\newcommand{\homo}{{\mathrm{hom}}}
\newcommand{\Sub}{{\mathrm{Sub}}}
\newcommand{\sub}{{\mathrm{S}}}
\newcommand{\Yon}{{\mathrm{Y}}}
\newcommand{\two}{{\mathbf{2}}}
\newcommand{\one}{{\mathbf{1}}}
\newcommand{\HH}{{\mathcal{H}}}
\newcommand{\MM}{{\mathcal{M}}}
\newcommand{\mono}{\rightarrowtail}
\newcommand{\epi}{\twoheadrightarrow}
\newcommand{\supp}{{\mathsf{supp}}}

\newenvironment{proof}{\paragraph{Proof.}}{\hfill $\square$}
\newtheorem{Thm}{Theorem}[section]
\newtheorem{Lem}[Thm]{Lemma}
\newtheorem{Cor}[Thm]{Corollary}
\newtheorem{Exa}[Thm]{Example}
\newtheorem{Obs}[Thm]{Observation}
\newtheorem{Def}{Definition}[section]
\newcommand{\MYenddef}{\hfill \mbox{\ $\lozenge$}}

\newarrow{Dashto}    {dash}{dash}{dash}{dash}>
\newarrow{Equal}     =====
\newarrow{Onto}      ----{>>}
\newarrow{Embed}     >--->
\newarrow{Dots}      .....

\begin{document}

\title{{\huge\textsc{Fibered Categories}} \\ \`a la Jean B\'enabou}
\author{Thomas Streicher}
\date{April 1999 -- April 2022}
\maketitle

\bigskip\bigskip
The notion of \emph{fibered category} was introduced by A.~Grothendieck for
purely geometric reasons. The ``logical'' aspect of fibered categories and, in 
particular, their relevance for \emph{category theory over an arbitrary base
category with pullbacks}  has been investigated and worked out in detail 
by Jean B\'enabou. The aim of these notes is to explain B\'enabou's approach 
to fibered categories which is mostly unpublished but intrinsic to most fields 
of category theory, in particular to topos theory and categorical logic. 

There is no claim for originality by the author of these notes. On the 
contrary I want to express my gratitude to Jean B\'enabou for his lectures
and many personal tutorials where he explained to me various aspects of his 
work on fibered categories.
I also want to thank J.-R.~Roisin for making me available his handwritten 
notes \cite{Ben} of \emph{Des Cat\'egories Fibr\'ees}, a course by Jean 
B\'enabou given at the University of Louvain-la-Neuve back in 1980. 

The current notes are based essentially on \cite{Ben} and quite a few other
insights of J.~B\'enabou that I learnt from him personally. The last four 
sections are based on results of J.-L.~Moens's Th\'ese \cite{Moe} from 1982
which itself was strongly influenced by \cite{Ben}.

\tableofcontents
\newpage

\section{Motivation and Examples}

If $\C$ is a category then a functor \[ F : \C^\op \to \Set \]
also called a ``presheaf over $\C$'' is most naturally considered as a ``set
varying over $\C$''. Of course, one may consider also contravariant functors 
on $\C$ taking their values not in $\Set$ but in some big category of 
structures like {\bf Grp}, {\bf Ab}, {\bf Rng}, {\bf Sp} etc. Typically, 
a presheaf $G : \C^\op \to {\mathbf{Grp}}$ of groups appears as a group 
object in $\widehat{\C} = \Set^{\C^\op}$ which is a topos if the category
$\C$ is small.

More generally, one may consider ``presheaves of categories''
\[ \HH : \C^\op \to {\mathbf{Cat}} \]
which notion will soon be axiomatized and generalised to
our central notion of \emph{fibered category}. But before we consider some 
examples that (hopefully) will provide some intuition and motivation.

\begin{Exa}\label{ex1}
Let $\C$ be the category of monoids and monoid homomorphisms. With every 
monoid $M$ one may associate the category \[ \HH(M) = \Set^{M^\op} \]
of right actions of $M$ on some set and with every monoid homomorphism 
$h : N \to M$ one may associate the functor
\[ \HH(h) = h^* = \Set^{h^\op} : \Set^{M^\op} \to \Set^{N^\op} \]
where $h^*(X,\alpha) : X{\times}N \to X : (x,b) \mapsto \alpha(x,h(b))$. 
\MYenddef
\end{Exa}

\begin{Exa}\label{ex2}
Of course, Example~\ref{ex1} can be generalised by taking for $\C$ some 
subcategory of the category of (small) categories and instead of $\Set$ some
other big category ${\mathcal K}$ (e.g.\ ${\mathcal{K}} = {\mathbf{Ab}}$ 
and $\C = {\mathbf{Cat}}$). \MYenddef
\end{Exa}

\begin{Exa}\label{ex3}
Let $\E$ be an elementary topos (see e.g.\ \cite{Joh}). Then 
\[ \E(-,\Omega) : \E^\op \to {\mathbf{Ha}} \]
is a contravariant functor from $\E$ to the category ${\mathbf{Ha}}$ of
Heyting algebras and their morphisms. \MYenddef
\end{Exa}

\begin{Exa}\label{ex4}
Let $\C$ be the category ${\mathbf{CRng}}$ of commutative rings with $1$. 
Then we may consider the functor
$$\HH : {\mathbf{CRng}}^\op \to {\mathbf{Cat}}$$
where $\HH(R)$ is the category of $R$--modules and for a homomorphism
$h : R^\prime \to R$ the functor $\HH(h)$ performs ``restriction of
scalars'', i.e.\ $\HH(h)(M)$ is the $R^\prime$--module with the same
addition as $M$ and scalar multiplication given by 
$r \cdot x = h(r) \cdot_M x$. \MYenddef
\end{Exa}

\begin{Exa}\label{ex5}
Consider the following instance of Example~\ref{ex2}. Let $\C = \Set$ 
(where sets are considered as small discrete categories) and 
${\mathcal{K}} = \X$ be some (typically not small) category. Then we have
\[ \Fam(\X) : \Set^\op \to {\mathbf{Cat}} \]
where $\Fam(\X)(I) = \X^I$ and 
\[ \Fam(\X)(u) = \X^u : \X^I \to \X^J \]
for $u : J \to I$ in $\Set$.\\
This example is \emph{paradigmatic} for \emph{Fibered Category Theory \`a la
B\'enabou} as it allows categories over $\Set$ to be considered as fibrations 
over $\Set$. Replacing $\Set$ by more general categories $\B$ as e.g.\ toposes
or even just categories with pullbacks one may develop a fair amount of 
\emph{category theory over base $\B$} !
\end{Exa}

\begin{Exa}\label{ex6}
For a category $\B$ with pullbacks we may consider 
$\HH : \B^\op \to {\mathbf{Cat}}$ sending $I \in \B$ to
$\HH(I) = \B/I$ and $u : J \to I$ in $\B$ to the pullback functor
$\HH(u) = u^* : \B/I \to \B/J$ which is right adjoint to 
$\Sigma_u \equiv u \circ (-)$ (postcomposition with $u$).

Notice that this is an example only \emph{cum grano salis} as 
$u^* : \B/I \to \B/J$ involves some choice of pullbacks and, accordingly,
in general we do not have 
$\HH(uv) = \HH(v) \circ \HH(u)$ but only 
$\HH(uv) \cong \HH(v) \circ \HH(u)$ where the 
components of the natural isomorphism are given by the respective mediating 
arrows. Such ``functors'' preserving composition (and identity) only up
to isomorphism are usually called \emph{pseudo--functors}. \MYenddef
\end{Exa}

We definitely do \emph{not} want to exclude the situation of Example~\ref{ex6} 
as it allows one to consider the base category $\B$ as ``fibered over itself''. 
Therefore, one might feel forced to accept pseudo--functors and the ensuing 
bureaucratic handling of ``canonical isomorphisms''. However, as we will show 
immediately one may replace pseudo--functors 
$\HH : \B^\op \to \Cat$ by fibrations 
$P:\X \to \B$ where this bureaucracy will turn out as luckily hidden from us. 

To motivate the definition of a fibration let us consider a functor
$\HH : \B^\op \to {\mathbf{Cat}}$ from which we will 
construct the ``fibration'' $P = \int \HH : \X\to\B$. The objects 
of $\X$ are pairs $(I,X)$ where $I \in \B$ and $X \in \HH(I)$. A 
morphism in $\X$ from $(J,Y)$ to $(I,X)$ is a pair $(u,\alpha)$ where 
$u : J \to I$ in $\B$ and $\alpha : Y \to \HH(u)(X)$ in 
$\HH(J)$. Composition in $\X$ is defined as follows: for maps
$(v,\beta) : (K,Z) \to (J,Y)$ and $(u,\alpha) : (J,Y) \to (I,X)$ 
in $\int \HH$ their composition $(u,\alpha) \circ (v,\beta)$ is given by 
$(u \circ v, \HH(u)(\alpha) \circ \beta)$. It is readily checked 
that this composition is associative and identities are given by 
$\id_{(I,X)} = (\id_I,\id_{X})$. Let $P = \int\HH : \X \to \B$ be
the functor sending an object $(I,X)$ in $\X$ to $I$ in $\B$ and a morphism
$(u,\alpha)$ in $\X$ to $u$ in $\B$.

Similarly, the pseudo--functor from Example~\ref{ex6} may be replaced by the
functor $P_\B \equiv \partial_1 \equiv {\mathsf{cod}} : \B^\two \to \B$ where
$\two$ is the partial order $0 \to 1$, i.e.\ the ordinal $2$. Obviously,
$P_\B$ sends a commuting square 
\begin{diagram}[small]
B           &  \rTo^{f}  &   A         \\
\dTo^{b}    &            &   \dTo_{a}  \\
J           &  \rTo_{u}  &   I         \\
\end{diagram}
to $u$. Just as we have written $\partial_1$ for the ``codomain'' functor 
${\mathsf{cod}}$ we will write $\partial_0$ for the ``domain'' functor 
${\mathsf{dom}} : \B^\two \to \B$. As $P_\B$ allows one to consider $\B$ as
fibered over itself and this is fundamental for developing category theory
over $\B$ we call $P_\B$ the \emph{fundamental fibration of $\B$}.

Let $P : \X \to \B$ be a functor as described above. A morphism $\varphi$ in
$\X$ is called \emph{vertical} iff $P(\varphi) = \id$. We write $P(I)$ or
$\X_I$ for the subcategory of $\X$ which appears as ``inverse image of $I$
under $P$'', i.e.\ which consists of objects $X$ with $P(X) = I$ and morphisms
$\varphi$ with $P(\varphi) = \id_I$.
If $P = \int\HH$ then $(u,\alpha)$ will be called \emph{cartesian}
iff $\alpha$ is an isomorphism and if $P = P_\B$ then a morphism in $\B^\two$
will be called \emph{cartesian} iff the corresponding square is a pullback
in $\B$.

\newpage

\section{Basic Definitions}

From the examples in the previous section we destill the following definition
of fibered category.

\begin{Def}\label{cartdef}
Let $P : \X \to \B$ be a functor. A morphism $\varphi : Y \to X$ in $\X$
over $u := P(\varphi)$ is called \emph{cartesian} iff for all $v : K \to J$
in $\B$ and $\theta : Z \to X$ with $P(\theta) = u \circ v$ there exists a
unique morphism $\psi : Z \to Y$ with $P(\psi) = v$ and
$\theta = \varphi \circ \psi$. 
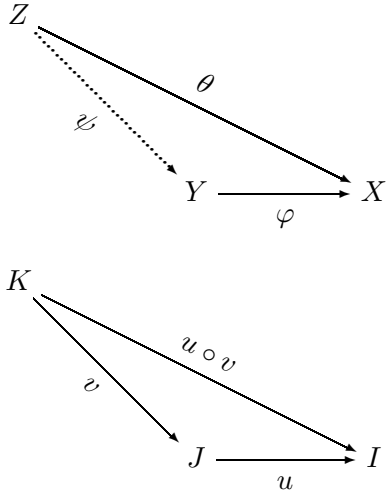
\begin{diagram}
Z &         &   &                &  \\
  & \rdDotsto_{\psi} \rdTo(4,2)^{\theta}  &   &       &  \\
  &         & Y & \rTo_{\varphi} & X \\
K &         &   &                &  \\
  & \rdTo_{v} \rdTo(4,2)^{u \circ v}  &   &       &  \\
  &         & J & \rTo_{u} & I \\
\end{diagram} 
A morphism $\alpha : Y \to X$ is called \emph{vertical} iff $P(\alpha)$ is
an identity morphism in $\B$.
For $I \in \B$ we write $\X_I$ or $P(I)$ for the subcategory of $\X$ 
consisting of those morphism $\alpha$ with $P(\alpha) = \id_I$. It is called
the \emph{fiber of $P$ over $I$}.
\MYenddef
\end{Def}

It is straightforward to check that cartesian arrows are closed under 
composition and that $\alpha$ is an isomorphism in $\X$ iff $\alpha$ is a
cartesian morphism over an isomorphism.

\begin{Def}\label{fibdef1}
A functor $P : \X \to \B$ is called a (Grothendieck) \emph{fibration} or
\emph{category fibered over $\B$} iff for all $u : J \to I$ in $\B$ and
$X \in P(I)$ there exists a cartesian arrow $\varphi : Y \to X$ over $u$
called a \emph{cartesian lifting of $X$ along $u$}.
\MYenddef
\end{Def}

Obviously, the functors $\int\HH$ and $P_\B$ of the previous section
are examples of fibrations and the \emph{ad hoc} notions of ``cartesian'' as
given there coincide with the official ones of Definition~\ref{fibdef1}.

Notice that cartesian liftings of $X \in P(I)$ along $u : J \to I$ are unique
up to vertical isomorphism: suppose that $\varphi : Y \to X$ and 
$\psi : Z \to X$ are cartesian over $u$ then there exist vertical arrows
$\alpha : Z \to Y$ and $\beta : Y \to Z$ with $\varphi \circ \alpha = \psi$
and $\psi \circ \beta = \varphi$, respectively, from which it follows by
cartesianness of $\varphi$ and $\psi$ that $\beta \circ \alpha = \id_Z$ and
$\alpha \circ \beta = \id_Y$ as 
$\psi \circ \beta \circ \alpha = \varphi \circ \alpha = \varphi = 
\varphi \circ \id_Y$ and
$\varphi \circ \beta \circ \alpha = \psi \circ \alpha = \varphi = 
\varphi \circ \id_Y$.

\begin{Def}\label{fibdef2}
Let $P : \X \to \B$ and $Q : \Y \to \B$ be fibrations over $\B$.\\
A \emph{cartesian} or \emph{fibered} functor from $P$ to $Q$ is an ordinary
functor $F : \X \to \Y$ such that 
\begin{enumerate}
\item[\rm (1)] $Q \circ F = P$ and
\item[\rm (2)] $F(\varphi)$ is cartesian w.r.t.\ $Q$ whenever $\varphi$ is
           cartesian w.r.t.\ $P$.
\end{enumerate}

If $F$ and $G$ are cartesian functors from $P$ to $Q$ then a \emph{cartesian
natural transformation from $F$ to $G$} is an ordinary natural transformation
$\tau : F \Rightarrow G$ with $\tau_X$ vertical for every $X \in \X$.

The ensuing 2-category will be called $\Fib(\B)$. \MYenddef
\end{Def}

Of course, if $\B$ is the terminal category then $\Fib(\B)$ is isomorphic
to the 2-category $\Cat$.

\bigskip\bigskip
\noindent
{\bf Remark.} What we have called ``cartesian'' in Definition~\ref{cartdef}
is usually called \emph{hypercartesian} whereas ``cartesian'' morphisms are
defined as follows: a morphism $\varphi : Y \to X$ is called \emph{cartesian}
iff for all $\psi : Z \to X$ with $P(\varphi) = P(\psi)$ there is a unique
vertical arrow $\alpha : Z \to Y$ with $\varphi \circ \alpha = \psi$. 
Employing this more liberal notion of ``cartesian'' one has to strengthen 
the definition of fibered category by adding the requirement that cartesian 
arrows are closed under composition. It is a simple exercise to show that
this addendum ensures that every cartesian arrow (in the liberal sense) is
actually hypercartesian (i.e.\ cartesian in the more restrictive sense of
our definition) and, accordingly, both definitions of fibered category are
equivalent.

As the current notes consider only fibrations for which ``cartesian'' and 
``hypercartesian'' are equivalent anyway we have adopted the somewhat
non--canonical Definition~\ref{cartdef} as in our context it will not lead
to any confusion.

Notice, however, that in more recent (unpublished) work by J.~B\'enabou 
on \emph{generalised fibrations} the distinction between cartesian arrows
(in the liberal sense) and hypercartesian arrows turns out as crucial.

\bigskip\bigskip
Obviously, a fibration $P : \X\to\B$ is a fibration of groupoids iff
all vertical arrows are isos iff all morphism of $\X$ are cartesian and
thus $P$ is a \emph{discrete} fibration, i.e.\ a fibration of discrete
categories, iff all vertical arrows are identities.  

\begin{Lem}\label{fibcartlm}
Suppose $P : \X \to \B$ and $Q : \Y \to \B$ are fibrations and $F : Q \to P$
is a cartesian functor over $\B$. If $P$ is discrete, i.e.\ all vertical
arrows are identities, then $F$ is a fibration itself.
\end{Lem}
\begin{proof}
Suppose $Y \in \Y$ and $f : X \to FY$ is a morphism in $\X$. Since $Q$ is a
fibration there exists a $Q$-cartesian arrow $\varphi : Z \to Y$ in $\Y$ above
$P(f)$. Since $F$ is cartesian $F(\varphi) : FZ \to FY$ is $P$-cartesian. We
have $P(F(\varphi)) = Q(\varphi) = P(f)$ and thus both $F(\varphi)$ and $f$
are morphism to $FY$ over $P(f)$. Since $P$ is a discrete fibration it follows
that $F(\varphi) = f$. It remains to show that $\varphi$ is $F$-cartesian.
For this purpose suppose $g : U \to X$ and $\psi : V \to Y$ with
$F(\psi) = F(\varphi) g$. Then $Q(\psi) = Q(\varphi)P(g)$ and thus, since
$Q$ is a fibration, there exists a unique $\theta : V \to Z$ with
$\varphi\theta = \psi$ and $Q(\theta) = P(g)$. Thus $F(\varphi) g
= F(\psi) = F(\varphi)F(\theta)$ from which it follows that $F(\theta) = g$
since $P$ is a discrete fibration. Suppose $\widetilde{\theta} : V \to Z$
with $\varphi \widetilde{\theta}= \psi$ and $F(\widetilde{\theta}) = g$. 
Then $Q(\widetilde{\theta}) = P(F(\widetilde{\theta})) = P(g)$.
Thus $\theta = \widetilde{\theta}$ as desird.      
\end{proof}

\bigskip
In general, i.e.\ if $P$ is not assumed to be discrete, a cartesian
functor $F : Q \to P$ will not be a fibration. For example if $\B$ is
nontrivial, $Q = \Id_\B$ and $F$ is right adjoint to $P$, i.e.\ $F$ picks a
terminal object in each fiber, then $F$ is not a fibration unless all fibers
are equivalent to $\one$. In particular, if $\B$ is the ordinal $\two$ and
$P$ is the fundamental fibration $P_\B$ of $\B$ then the functor
$1 : \Id_\B \to P_\B$ (sending $I$ to $\id_I$) is not a fibration. Thus, it is
not sufficient to require that $P$ is faithful, i.e.\ that $P$ is a fibration
of posetal categories, for $F$ being a fibration, too.

\newpage

\section{Split Fibrations and Fibered Yoneda Lemma}

If $P : \X\to\B$ is a fibration then using axiom of choice for classes we may
select for every $u : J \to I$ in $\B$ and $X \in P(I)$ a cartesian arrow
$\Cart(u,X) : u^*X \to X$ over $u$. Such a choice of cartesian liftings is 
called a \emph{cleavage} for $P$ and it induces for every map $u : J \to I$ 
in $\B$ a so-called \emph{reindexing functor} $u^* : P(I) \to P(J)$ 
in the following way
\begin{diagram}
u^*X & \rTo^{\Cart(u,X)} & X \\
\dDashto^{u^*\alpha} & & \dTo_{\alpha} \\
u^*Y & \rTo_{\Cart(u,Y)} & Y \\
\end{diagram}
where $u^*\alpha$ is the unique vertical arrow making the diagram commute.
Alas, in general for composable maps $u : J \to I$ and $v : K \to J$ in $\B$
it does not hold that
$$v^* \circ u^* = (u \circ v)^*$$ 
although the functors are canonically isomorphic via $c_{u,v}$ 
as shown in the following diagram
\begin{diagram}
v^*u^*X\; & & & & \\
 & \rdTo^{\Cart(v,u^*X)}_{\mbox{cart.}} & & & \\
\dDashto^{(c_{u,v})_X}_{\cong} & & u^*X & & \\
& & & \rdTo^{\Cart(u,X)}_{\mbox{cart.}} & \\
(uv)^*X & & \rTo_{\Cart(uv,X)}^{\mbox{cart.}} & & X \\
\end{diagram}
where $(c_{u,v})_X$ is the unique vertical arrow making the diagram commute.

Typically, for $P_{\B} = \partial_1 : \B^{\two} \to \B$, the
{\bf fundamental fibration} for a category $\B$ with pullbacks, we do not
know how to choose pullbacks in a functorial way, i.e.\ that
$\Cart(\id,X) = \id_X$ and
$\Cart(u{\circ}v,X) = \Cart(u,X) \circ \Cart(v,u^*X)$.
Of course, the first condition is easy to achieve but the problem is
the second condition since in general one does not know how to choose
pullbacks in such a way that they are closed under composition.

But, nevertheless, often such a functorial choice of cartesian liftings is 
possible in particular situations.

\begin{Def}
A cleavage $\Cart$ of a fibration $P : \X \to \B$ is called \emph{split} 
or a \emph{splitting of $P$} iff the following two conditions are satisfied
\begin{itemize}
\item[\rm (1)]   $\Cart(\id,X) = \id_X$
\item[\rm (2)]   $\Cart(uv,X) = \Cart(u,X) \circ \Cart(v,u^*X)$.  
\end{itemize}
A \emph{split fibration} is a fibration \emph{endowed} with a split cleavage.

A \emph{split cartesian functor} between split fibrations is a cartesian
functor $F$ between split fibrations which, moreover, preserves chosen 
cartesian liftings, i.e.\ satisfies 
\[ F(\Cart(u,X)) = \Cart(u,F(X)) \] 
for all $u : J \to I$ in the base and all $X$ over $I$. We write $\Spl(\B)$
for the ensuing category of split fibrations over $\B$ and split cartesian
functors between them.
\MYenddef 
\end{Def}

\bigskip
\noindent
{\bf Warning.}\\
(1) There are fibrations which are not splitable. Consider for example the 
groups $\B = (\Ze_2,+_2)$ and $\X = (\Ze,+)$ (considered as categories) and 
the fibration $P : \X \to \B : a \mapsto P(a) := a \, {\mathrm{mod}} \, 2$. 
A splitting of $P$ would give rise to a functor $F : \B \to \X$
with $P \circ F = {\mathrm{Id}}_\B$ but that cannot exist as there is no group
homomorphism $h : (\Z_2,+_2) \to (\Z,+)$ with $h(1)$ an odd number of $\Z$.\\
(2) Notice that different splittings of the same fibration may give rise to
the same presheaf of categories. Consider for example 
$\HH : \two^\op\to {\mathbf{Ab}}$ with
$\HH(1) = {\mathcal{O}}$, the zero group, and $\HH(0)$ some
non--trivial abelian group $A$. Then every $g \in A$ induces a splitting 
$\Cart_g$ of $P \equiv \int {\mathcal H}$ by putting
\begin{enumerate}
\item[]  \qquad $\Cart_g(u,\star) = (u,g)$ 
         \qquad for $u : 0 \to 1$ in $\two$
\end{enumerate}
but all these $\Cart_g$ induce the same functor $\two^\op \to \Cat$,
namely $\HH$ ! 

In the light of (2) it might appear as more appropriate to define 
split fibrations over $\B$ as functors from $\B^\op$ to $\Cat$. The latter
may be considered as categories internal to $\widehat{\B} = \Set^{\B^\op}$
and organise into the (2-)category $\cat(\B)$ of categories and functors 
internal to $\widehat{\B}$. However, as $\Spl(\B)$ and $\cat(\B)$ are strongly 
equivalent as 2-categories we will not distiguish them any further in the rest
of these notes.

\bigskip
Next we will presented the \emph{Fibered Yoneda Lemma} making precise the 
relation between fibered categories and split fibrations (over the same base).

\subsection*{Fibered Yoneda Lemma}

Though, as we have seen, not every fibration $P \in \Fib(\B)$ is isomorphic 
to a splitable fibration there is always a distinguished \emph{equivalent} 
split fibration as ensured by the so-called \emph{Fibered Yoneda Lemma}.
Before giving the full formulation of the Fibered Yoneda Lemma we motivate
the construction of a canonical split fibration $\Split(P)$ equivalent to a
given fibration $P \in \Fib(\B)$. 

For an object $I \in \B$ let $\underline{I} = P_I = \partial_0 : \B/I \to \B$
be the discrete fibration corresponding to the representable presheaf 
$Y_\B(I) = \B(-,I)$ and for $u : J \to I$ in $\B$ let $\underline{u} = 
P_u = \Sigma_u$ be the cartesian functor from $\underline{J}$ to 
$\underline{I}$ as given by postcomposition with $u$ and corresponding to 
the presheaf morphism $Y_\B(u) = \B(-,u) : Y_\B(J) \to Y_\B(I)$. Then 
cartesian functors from $\underline{I}$ to $P : \X \to \B$ in $\Fib(\B)$ 
correspond to choices of cartesian liftings for an object $X \in P(I)$. There 
is an obvious functor $E_{P,I} : \Fib(\B)(\underline{I},P) \to P(I)$ sending 
$F$ to $F(\id_I)$ and $\tau : F\to G$ to $\tau_{\id_I} : F(\id_I)\to G(\id_I)$.
It is a straightforward exercise to show that $E_{P,I}$ is full and faithful
and using the axiom of choice for classes we also get that $E_{P,I}$ is 
surjective on objects, i.e.\ that 
$E_{P,I} : \Fib(\B)(\underline{I},P) \to P(I)$ is an equivalence of categories.
Now we can define $\Split(P) : \B^\op \to \Cat$ as
\[ \Split(P)(I) = \Fib(\B)(\underline{I},P) \]
for objects $I$ in $\B$ and 
\[ \Split(P)(u) = \Fib(\B)(\underline{u},P) : \Split(P)(I) \to \Split(P)(J) \]
for morphisms $u : J \to I$ in $\B$. Let us write $U(\Split(P))$ for 
$\int \Split(P)$, the fibration obtained from $\Split(P)$ via the 
Grothendieck construction. Then the $E_{P,I}$ as described above arise as
the components of a cartesian functor $E_P : U(\Split(P)) \to P$ sending 
objects $(I,X)$ in $ U(\Split(P)) = \int \Split(P)$ to $E_{P,I}(X)$ and 
morphism $(u,\alpha) : G \to F$ in $U(\Split(P)) = \int \Split(P)$ 
over $u : J \to I$ to the morphism
$F(u{:}u{\to}\id_I) \circ \alpha_{\id_J} : G(\id_J) \to F(\id_I)$ in $\X$.
As all fibers of $E_P$ are equivalences it follows\footnote{We leave it as
an exercise to show that under assumption of axiom of choice for classes a 
cartesian functor is an equivalence in $\Fib(\B)$ iff all its fibers are 
equivalences of categories.} that $E_P$ is an equivalence in the $2$-category 
$\Fib(\B)$. 

Actually, the construction of $\Split(P)$ from $P$ is just the object part
of a $2$-functor $\Split : \Fib(\B) \to \Spl(\B)$ right adjoint to the 
forgetful $2$-functor from $\Spl(\B)$ to $\Fib(\B)$ as described in the 
following theorem 
(which, however, will not be used any further in the rest of these notes).

\begin{Thm}\label{fYl}\emph{(Fibered Yoneda Lemma)}\\
For every category $\B$ the forgetful $2$-functor $U : \Spl(\B) \to \Fib(\B)$
has a right $2$-adjoint $\Split : \Fib(\B) \to \Spl(\B)$, i.e.\ there is
an equivalence of categories
\[ \Fib(\B)(U(S),P) \simeq \Spl(\B)(S,\Split(P)) \]
naturally in $S \in \Spl(\B)$ and $P \in \Fib(\B)$, whose counit 
$E_P : U(\Split(P)) \to P$ at $P$ is an equivalence in $\Fib(\B)$ 
for all $P \in \Fib(\B)$.

However, in general the unit $H_S : S \to \Split(U(S))$ at $S \in \Spl(\B)$
is not an equivalence in $\Spl(\B)$ although $U(H_S)$ is always an equivalence 
in $\Fib(\B)$.
\end{Thm}
\begin{proof}
The functor $U : \Spl(\B) \to \Fib(\B)$ just forgets cleavages. The object 
part of its right adjoint $\Split$ is as described above, namely
\[ \Split(P)(I) = \Fib(\B)(\underline{I},P) \qquad\quad 
   \Split(P)(u) = \Fib(\B)(\underline{u},P)\]
for $P \in \Fib(\B)$. For cartesian functors $F : P \to Q$ in $\Fib(\B)$ we
define $\Split(F) : \Split(P) \to \Split(Q)$ as
\[  \Split(F)_I = \Fib(\B)(\underline{I},F) \]
for objects $I$ in $\B$. Under assumption of axiom of choice for classes
the counit for $U \dashv \Split$ at $P$ is given by the equivalence
$E_P : U(\Split(P)) \to P$ as described above. 
The unit $H_S : S \to \Split(U(S))$ for $U \dashv \Split$ at $S \in \Spl(\B)$
sends $X \in P(I)$ to the cartesian functor from $\underline{I}$ to $P$ which
chooses cartesian liftings as prescribed by the underlying cleavage of $S$ and
arrows $\alpha : X \to Y$ in $P(I)$ to the cartesian natural transformation
$H_S(\alpha) : H_S(X) \to H_S(Y)$ with $H_S(\alpha)_{\id_I} = \alpha$. We
leave it as a tedious, but straightforward exercise to show that these data
give rise to an equivalence
\[ \Fib(\B)(U(S),P) \simeq \Spl(\B)(S,\Split(P)) \]
naturally in $S$ and $P$.

As all components of $H_S$ are equivalences of categories it follows that
$U(H_S)$ is an equivalence in $\Fib(\B)$. However, it cannot be the case
that all $H_S$ are equivalences as otherwise a split cartesian functor $F$
were an equivalence in $\Spl(\B)$ already if $U(F)$ is an equivalence in 
$\Fib(\B)$ and this is impossible as not every epi in $\widehat{\B}$ is a
split epi.
\end{proof}

\bigskip
As $E_P : U(\Split(P)) \to P$ is always an equivalence it follows 
that for fibrations $P$ and $Q$
\[ \Split_{P,Q} : \Fib(\B)(P,Q) \to \Spl(\B)(\Split(P),\Split(Q)) \]
is an equivalence of categories. 

However, in general $\Split_{P,Q}$ is not an isomorphism of categories.
An arbitrary split cartesian functor $G : \Split(P) \to \Split(Q)$ corresponds
via the $2$-adjunction $U \dashv \Split$ to a cartesian functor 
$E_Q \circ U(G) : U(\Split(P)) \to Q$ which, however, need not factor as 
$E_Q \circ U(G) = F \circ E_P$ for some cartesian $F : P \to Q$.\footnote{
For example, if $Q = U({{\mathit{Sp}}}(P))$ and 
$E_Q \circ U(G) = {{\mathit{Id}}}_{U({{\mathit{Sp}}}(P))}$ and $E_P$ is not
one-to-one on objects which happens to be the case whenever cartesian 
liftings are not unique in $P$.} One may characterise the split cartesian
functors of the form $\Split(F)$ for some cartesian $F : P \to Q$ as those
split cartesian functors $G : \Split(P) \to \Split(Q)$ satisfying
$\Split(E_Q) \circ \Split(U(G)) =  G \circ \Split(E_P)$. One easily sees 
that this condition is necessary and if it holds then an $F$ with 
$G = \Split(F)$ can be obtained as $E_Q \circ U(G) \circ E'_P$ 
for some $E'_P$ with $E_P \circ E'_P = {{\mathit{Id}}}_P$ because we have
$\Split(F) = \Split(E_Q \circ U(G) \circ E'_P) =
 \Split(E_Q) \circ \Split(U(G)) \circ \Split(E'_P) =
 G \circ \Split(E_P) \circ \Split(E'_P) =  G \circ \Split(E_P \circ E'_P) = G$.
 
Although $\Split$ is not full and faithful the adjunction $U \dashv \Split$ 
nevertheless is of the type ``full reflective subcategory'' albeit in the 
appropriate $2$-categorical sense. This suggests that $\Fib(\B)$ is obtained 
from $\Spl(\B)$ by ``freely quasi-inverting weak equivalences in $\Fib(\B)$''
which can be made precise as follows. 

A split cartesian functor $F$ is called a \emph{weak equivalence}
iff all its fibers are equivalences of categories, i.e.\ iff $U(F)$ is an 
equivalence in $\Fib(\B)$. Let us write $\Sigma$ for the class of weak 
equivalences in $\Spl(\B)$. For a $2$-category $\XX$ and a $2$-functor 
$\Phi : \Spl(\B) \to \XX$ we say that $\Phi$ \emph{quasi-inverts} a morphism
$F$ in $\Spl(\B)$ iff $\Phi(F)$ is an equivalence in $\XX$. Obviously, the
$2$-functor $U : \Spl(\B) \to \Fib(\B)$ quasi-inverts all weak equivalences.
That $U$ freely inverts the maps in $\Sigma$ can be seen as follows. Suppose
that a $2$-functor $\Phi : \Spl(\B) \to \XX$ quasi-inverts all weak 
equivalences. Then there exists a $2$-functor $\Psi : \Fib(\B) \to \XX$ 
unique up to equivalence with the property that $\Psi \circ U \simeq  \Phi$. 
As by assumption $\Phi$ quasi-inverts weak equivalences we have 
$\Phi \circ \Split \circ U \simeq \Phi$ because all $H_S$ are weak 
equivalences. On the other hand if $\Psi \circ U \simeq \Phi$ then we have
$\Psi \simeq \Psi \circ U \circ \Split \simeq \Phi \circ \Split$ (because all
$E_P$ are equivalences) showing that $\Psi$ is unique up to equivalence.

\subsection*{A Left Adjoint Splitting}

The forgetful functor $U : \Spl(\B) \to \Fib(\B)$ admits also a left adjoint
$L : \Fib(\B) \to \Spl(\B)$ which like the right adjoint splitting
discussed previously was devised by J.~Giraud in the late 1960s. 

This left adjoint splitting $L(P)$ of a fibration $P : \X \to \B$ is 
constructed as follows. First choose a cleavage $\Cart_P$ of $P$ which is 
\emph{normalized} in the sense that $\Cart_P(\id_I,X) = \id_X$ for all $X$ 
over $I$. From this cleavage one may construct a presheaf 
$S(P) : \B^\op \to \Cat$ of categories giving rise to the desired split fibration
$L(P)$ over $\B$. For $I \in \B$ the objects of $S(P)(I)$ are pairs $(a,X)$ where
$X$ is an object of $\X$ and $a : I \to P(X)$. Morphisms from $(b,Y)$ to $(a,X)$
are vertical morphism $\alpha : b^*Y \to a^*X$ and composition in $S(P)(I)$ is
inherited from $\X$, i.e.\ $P(I)$. For $u : J \to I$ in $\B$ the functor $S(P)(u)
: S(P)(I) \to S(P)(J)$ is constructed as follows. For $(a,X)$ in $S(P)(I)$ let
$\Cart_{L(P)}(u,(a,X)) : (au)^*X \to u^*X$ be the unique cartesian arrow $\varphi$ 
over $u$ with $\Cart_P(a,X) \circ \varphi = \Cart_P(au,X)$. Let $\alpha : b^*Y \to a^*X$ 
be a morphism from $(b,Y)$ to $(a,X)$ in $S(P)(I)$. Then we define $S(P)(u)(\alpha)$
as the unique vertical morphism making the diagram
\begin{diagram}[small]
& & Y \\
 & \ruTo^{\Cart_P(bu,Y)} & \uTo_{\Cart_P(b,Y)} \\
(bu)^*Y & \rTo_{\;\;\;\Cart_{L(P)}(u,(b,Y))} & b^*Y \\
\dTo^{S(P)(u)(\alpha)} &  & \dTo_\alpha \\
(au)^*X & \rTo^{\;\;\;\Cart_{L(P)}(u,(a,X))} & a^*X \\
& \rdTo_{\Cart_P(au,X)} & \dTo_{\Cart_P(a,X)} \\
& & X
\end{diagram}
commute. One readily checks that $S(P)$ is indeed a functor from $\B^\op$ to $\Cat$
since  
$\Cart_{L(P)}(uv,(a,X) = \Cart_{L(P)}(u,(a,X)) \circ \Cart_{L(P)}(v,(au,X))$ 
and $\Cart_{L(P)}(\id_I,(a,X)) = \id_{a^*X}$ as one can see easily. 
Objects of the total category of $L(P)$ are objects of $S(P)(I)$ for some $I \in \B$
and morphisms from $(b,Y)$ to $(a,X)$ are just morphisms $b^*Y \to a^*X$ 
whose composition is inherited from $\X$. 
The functor $L(P)$ sends $(a,X)$ to the domain of $a$ and $f : b^*Y \to a^*X$ to $P(f)$.
The splitting of $L(P)$ is given by $\Cart_{L(P)}$ as defined above for specifying 
the morphism part of $S(P)$.
The unit $H_P : P \to U(L(P))$ of the (2-categorical) adjunction $L \dashv U$ 
sends $X$ to $(\id_{P(X)},X)$ and $f : Y \to X$ to $f : H_P(Y) \to H_P(X)$.

Notice that the above construction of $L(P)$ is based on a choice of a cleavage
for $P$. But this may be avoided by defining morphisms from $(b,Y)$ to $(a,X)$
over $u : J \to I$ as equivalence classes of spans $(\psi,f)$ in $\X$ where $\psi$ 
is a cartesian morphism to $Y$ over $b$ and $f$ is a morphism to $A$ over $au$ 
where $(\psi,f)$ and $(\psi^\prime,f^\prime)$ get identified iff there is a vertical 
isomorphism $\iota$ with $\psi \circ \iota = \psi^\prime$ and 
$f \circ \iota = f^\prime$. For a given cleavage $\Cart_P$ of $P$ the equivalence
class of $(\psi,f)$ contains a unique pair whose first component is $\Cart_P(b,Y)$.
 
\newpage

\section{Closure Properties of Fibrations}

In this section we will give some examples of fibrations and constructions of 
new fibrations from already given ones. Keeping in mind that we think of 
fibrations over $\B$ as generalisations of fibrations of the form $\Fam(\C)$ 
over $\Set$ it will appear that most of these constructions are 
generalisations of well-known constructions in ${\bf Cat}$.

\subsection*{Fundamental Fibrations}

For a category $\B$ the codomain functor
    $$P_\B \equiv \partial_1 : \B^{\two} \to \B$$
is a fibration if and only if $\B$ has pullbacks. In this case $P_\B$ is 
called the {\bf fundamental fibration} of $\B$.

\subsection*{Externalisations of Internal Categories}

Let $C$ be a category internal to $\B$ as given by domain and codomain
maps $d_0,d_1 : C_1 \to C_0$, the identity map $i : C_0 \to C_1$ and a 
composition map $m : C_1 \times_{C_0} C_1 \to C_1$. Then one may construct the
fibration $P_C : \underline{C} \to \B$ called \emph{externalisation of $C$}.
The objects of $\underline{C}$ over $I$ are pairs $(I,a : I \to C_0)$ and a
morphism in $\underline{C}$ from $(J,b)$ to $(I,a)$ over $u : J \to I$ is 
given by a morphism $f : J \to C_1$ with $d_0 \circ f = b$  and 
$d_1 \circ f =  a \circ u$. Composition in $C$ is defined using $m$ analogous
to $\Fam(\C)$. The fibration $P_C$ itself is defined as
$$P_C(I,a) = I \qquad\qquad  P_C(u,f) = u$$
and the cartesian lifting of $(I,a)$ along $u : J \to I$ is given by 
$i \circ a \circ u$.

In particular, every object $I \in \B$ can be considered as a \emph{discrete}
internal category of $\B$. Its externalisation is given by
$P_I = \partial_0 : \B/I \to \B$ for which (by a convenient abuse of notation)
we often also write $\underline{I}$
.

\subsection*{Change of Base and ``Glueing''}

If $P \in \Fib(\B)$ and $F : \C \to \B$ is an ordinary functor then
$F^*P \in \Fib(\C)$ where
\begin{diagram}[small]
\Y  \SEpbk         &  \rTo^{K}  &   \X         \\
\dTo^{F^*P}  &            &   \dTo_{P}  \\
\C           &  \rTo_{F}  &   \B         
\end{diagram}
is a pullback in ${\bf Cat}$. One says that fibration $F^*P$ is obtained
from $P$ \emph{by change of base along $F$}.
Notice that $(u,\varphi)$ in $\Y$ is cartesian w.r.t.\ $F^*P$ iff $\varphi$ 
is cartesian w.r.t.\ $P$. Accordingly, $K$ preserves cartesianness of arrows 
as $K(u,\varphi) = \varphi$.

When instantiating $P$ by the fundamental fibration $P_\B$ we get the 
following important particular case of change of base 
\begin{diagram}[small]
\B{\downarrow}F  \SEpbk  &  \rTo^{\partial_1^*F}  &    \B^{\two}  \\
\dTo^{P_F}           &           &   \dTo_{P_\B}  \\
\C           &  \rTo_{F\;}  &   \B         
\end{diagram}
where we write $P_F$ for $F^*P_\B$. This is often referred to as
\emph{(Artin) glueing} in which case
one often writes $\gl(F)$ for $P_F$ and $\Gl(F)$ for $\B{\downarrow}F$.
Typically, in applications the functor $F$ will be the inverse image part of 
a geometric morphism $F \dashv U : \E \to \Se$ between toposes. But already
if $F$ is a pullback preserving functor between toposes 
$\Gl(F) = \E{\downarrow}F$ is again a topos and the functor
$P_F = \gl(F) : \E{\downarrow}F \to \Se$ is \emph{logical}, i.e.\ preserves
all topos structure. The glueing construction will get very important later
on  when we discuss the \emph{Fibrational Theory of Geometric Morphisms}
\`a la J.-L.~Moens.  

We write $\Fib$ for the (non--full) subcategory of $\Cat^{\two}$ 
whose objects are fibrations and whose morphisms are commuting squares
\begin{diagram}[small]
\Y        &  \rTo^{K}  &   \X         \\
\dTo^{Q}  &            &   \dTo_{P}  \\
\C           &  \rTo_{F}  &   \B         
\end{diagram}
with $K$ cartesian over $F$, i.e.\ $K(\varphi)$ is cartesian over $F(u)$
whenever $\varphi$ is cartesian over $u$. Obviously, $\Fib$ is fibered
over $\Cat$ via the restriction of $\partial_1 : \Cat^{\two} \to \Cat$ to 
$\Fib$ for which we write $\Fib/\Cat : \Fib \to \Cat$. A morphism of $\Fib$
is cartesian iff it is a pullback square in $\Cat$.

We write $\Fib(\B)/\B$ for the fibration obtained from $\Fib/\Cat$ 
by change of base along the functor $\Sigma : \B \to \Cat$
sending $I$ to $\B/I$ and $u : J \to I$ to 
$\Sigma_u : \B/J \to \B/I : v \mapsto u \circ v$
\begin{diagram}[small]
\Fib {\downarrow} \Sigma  \SEpbk    &  \rTo  &  \Fib        \\
\dTo^{\Fib(\B)/\B}  &            &   \dTo_{\Fib/\Cat}  \\
\B           &  \rTo_{\Sigma\quad}  &   {\bf Cat}         
\end{diagram}
We leave it as an exercise to show that $P : \X \to \B/I$ is a fibration
iff $P_I \circ P$ is a fibration over $\B$ and
$P \in \Fib(\B)(P_I{\circ}P,P_I)$.
Accordingly, fibrations over $\B/I$ may be considered as $I$-indexed
families of fibrations over $\B$ in analogy with ordinary functors to a
discrete category $I$ which may be considered as $I$-indexed families of
categories.

\subsection*{Composition and Product of Fibrations}

First notice that fibrations are closed under composition. Even more we have 
the following 

\begin{Thm}\label{fibfib}
Let $P : \X \to \B$ be a fibration and $F : \Y \to \X$ be an arbitrary functor.
Then $F$ itself is a fibration over $\X$ iff
\begin{enumerate}
\item[\rm (1)] $Q \equiv P{\circ}F$ is a fibration and $F$ is a cartesian
               functor from $Q$ to $P$ over $\B$ and
\item[\rm (2)] all $F_I : \Y_I \to \X_I$ are fibrations and cartesian arrows
               w.r.t.\ these fibrations are stable under reindexing, i.e.\
               for every commuting diagram
\begin{diagram}[small]
Y_1 &  \rTo^{\varphi_1}  &   X_1   \\
\dTo^{\theta}  &   &   \dTo_{\psi}  \\
Y_2 &  \rTo_{\varphi_2}  &   X_2       
\end{diagram}
in $\Y$ with $\varphi_1$ and $\varphi_2$ cartesian w.r.t.\ $Q$ over the
same arrow $u : J \to I$ in $\B$ and $Q(\psi) =\id_I$ and $Q(\theta) =\id_J$ 
it holds that $\theta$ is cartesian w.r.t.\ $F_J$ 
whenever $\psi$ is cartesian w.r.t.\ $F_I$.
\end{enumerate} 
\end{Thm}
\begin{proof} Exercise left to the reader. 
\end{proof}

\medskip
The second condition means that the commuting diagram
\begin{diagram}[small]
\Y_I &  \rTo^{u^*}  &   \Y_J   \\
\dTo^{F_I}  &            &   \dTo_{F_J}  \\
\X_I &  \rTo_{u^*}  &   \X_J       
\end{diagram}
is a morphism in $\Fib$. (Notice that due to condition (1) 
of Theorem~\ref{fibfib} one can choose the reindexing functor 
$u^* : \Y_I \to \Y_J$ in such a way that the diagram actually commutes. For
arbitrary cartesian functors this need not be possible although for all
choices of the $u^*$ the diagram always commutes up to isomorphism.)

The relevance of Theorem~\ref{fibfib} is that it characterises 
``fibered fibrations'' as those fibered functors which are themselves ordinary
fibrations. This handy characterisation cannot even be formulated in the 
framework of indexed categories and, therefore, is considered as a typical
example of the superiority of the fibrational point of view.

For fibrations $P$ and $Q$ over $\B$ their product $P{\times_{\B}}Q$ in 
$\Fib(\B)$ is given by $P \circ P^*Q = Q \circ Q^*P$ as in
\begin{diagram}[small]
\PP \SEpbk &  \rTo^{Q^*P}  &   \Y   \\
\dTo^{P^*Q}  &            &   \dTo_{Q}  \\
\X &  \rTo_{P}  &   \B       
\end{diagram}
and it follows from Theorem~\ref{fibfib} that $P{\times_{\B}}Q$ is a 
fibration and that the projections $P^*Q$ and $Q^*P$ are cartesian functors.

\subsection*{Fibrations of Diagrams}

Let $\D$ be a category and $P : \X \to \B$ a fibration. Then the
\emph{fibration $P^{(\D)}$ of diagrams of shape $\D$} is given by
\begin{diagram}[small]
\X^{(\D)} \SEpbk &  \rTo  &   \X^{\D}   \\
\dTo^{P^{(\D)}}  &        &   \dTo_{P^\D}  \\
\B &  \rTo_{\Delta_\D}  &   \B^{\D}       
\end{diagram}
where the ``diagonal functor'' $\Delta_\D$ sends $I \in \B$ to the constant
functor with value $I$ and a morphism $u$ in $\B$ to the natural 
transformation all whose components are $u$. 

Somewhat surprisingly, as shown by A.~Kurz in spring 2019, 
the functor $P^\D$ is also a fibration, however, over $\B^\D$.

\subsection*{Exponentiation of Fibrations}

For fibrations $P$ and $Q$ over $\B$ we want to construct a fibration 
$[P{\to}Q]$ such that there is an equivalence
\[  \Fib(\B)(R,[P{\to}Q]) \simeq \Fib(\B)(R{\times_\B}P, Q)  \]
naturally in $R \in \Fib(\B)$.

Analogous to the construction of exponentials in 
$\widehat{\B} = \Set^{\B^\op}$ the fibered Yoneda lemma (Theorem~\ref{fYl})
suggest us to put 
$$[P{\to}Q](I) = \Fib(\B)(\underline{I}{\times_\B}P, Q)  \qquad
  [P{\to}Q](u) = \Fib(\B)(\underline{u}{\times_\B}P, Q)$$
where $\underline{u}$ is given by 
\begin{diagram}[small]
\B/J & & \rTo^{\Sigma_u} & & \B/I \\
     & \rdTo_{P_J} & &\ldTo_{P_I} &  \\
     & &      \B   &       &
\end{diagram}
for $u : J \to I$ in $\B$. We leave it as a tedious, but straightforward
exercise to verify that 
\[  \Fib(\B)(R,[P{\to}Q]) \simeq \Fib(\B)(R{\times_\B}P, Q)  \]
holds naturally in $R \in \Fib(\B)$.

Notice that we have 
\[  \Fib(\B)(P_I{\times_\B}P,Q) \simeq \Fib(\B/I)(P_{/I},Q_{/I}) \]
naturally in $I \in \B$ where $P_{/I} = {P_I}^*P$ and $Q_{/I} = {P_I}^*Q$ 
are obtained by change of base along $P_I$. Usually $P_{/I}$ is referred to 
as ``localisation of $P$ to $I$''. The desired equivalence follows from the 
fact that change of base along $P_I$ is right adjoint to postcomposition 
with $P_I$ and the precise correspondence between 
$F \in \Fib(\B)(P_I{\times_\B}P, Q)$ and 
$G \in \Fib(\B/I)(P_{/I},Q_{/I})$ is indicated by the following diagram
\begin{diagram}
\cdot &            & & & \\
  & \rdTo~{G} \rdTo(2,4)_{P_{/I}} \rdTo(4,2)^{F} & & & \\
  &            &   \cdot \SEpbk & \rTo &  \Y \\
  &            &  \dTo_{Q_{/I}}     &      &  \dTo_{Q} \\
  &            &    \B/I      & \rTo_{P_I} & \B  \\
\end{diagram}

\newpage

\section{The Opposite of a Fibration}

If $P : \X \to \B$ is a fibration thought of ``as of the form $\Fam(\C)$''
then one may want to construct the fibration $P^\op$ thought of ``of the 
form $\Fam(\C^\op)$''. It might be tempting at first sight to apply $(-)^\op$ 
to the functor $P$ giving rise to the functor $\X^\op \to \B^\op$ which, 
however, has the wrong base even if it were a fibration (which in general
will not be the case). If $P = \int \HH$ for some $\HH : \B^\op \to \Cat$ 
then one may consider
$\HH^\op = (-)^\op \circ \HH : \B^\op \to \Cat$, 
i.e.\ the assignment
\[   I \mapsto \HH(I)^\op \qquad\quad
     u : J \to I \; \mapsto \; \HH(u)^\op : \HH(I)^\op  \to  \HH(I)^\op
\]
where $(-)^\op$ is applied to the fibers of $\HH$ and to the reindexing 
functors. Now we express $P^\op = \int \HH^\op$ in terms of $P = \int \HH$ 
directly.

The fibration $P^\op : \Y \to \B$ is constructed from the fibration
$P : \X \to \B$ in the following way. The objects of $\Y$ and $\X$ are the 
same but for $X \in P(I)$, $Y \in P(J)$ and $u : J \to I$ the collection of
morphisms in $\Y$ from $Y$ to $X$ over $u$ is constructed as follows. It 
consists of all spans $(\alpha,\varphi)$ with $\alpha : Z \to Y$ vertical 
and $\varphi : Z \to X$ is cartesian over $u$ \emph{modulo} the equivalence 
relation $\sim_{Y,u,X}$ (also denoted simply as $\sim$) where 
$(\alpha,\varphi) \sim_{Y,u,X} (\alpha',\varphi')$ iff 
\begin{diagram}[small]
& & Z & & \\
& \ldTo^{\alpha} & & \rdTo^{\varphi}_\cart  & \\
Y  & & \uTo^{\iota}_{\cong} & & X\\
& \luTo_{\alpha'} & & \ruTo_{\varphi'}^\cart& \\
& & Z' & & \\
\end{diagram} 
for some (necessarily unique) vertical isomorphism $\iota : Z' \to Z$. 
Composition of arrows in $\Y$ is defined as follows: 
if $[(\alpha,\varphi)]_\sim : Y \to X$
over $u : J \to I$ and $[(\beta,\psi)]_\sim : Z \to Y$ over $v : K \to J$ then
$[(\alpha,\varphi)]_\sim \circ [(\beta,\psi)]_\sim := 
 [(\beta \circ \widetilde{\alpha}, \varphi \circ \widetilde{\psi})]_\sim$ where
\begin{diagram}[small]
& & & & (uv)^*X & & & &\\
&&&\ldTo^{\widetilde{\alpha}}&&\rdTo^{\widetilde{\psi}}_\cart & & &\\
& & v^*Y  & & \mbox{p.b.}& & u^*X & &\\
& \ldTo^{\beta} & & \rdTo_{\psi}^\cart & & \ldTo_{\alpha}& & 
\rdTo^{\varphi}_\cart &\\
Z & & & & Y & & & & X\\
\end{diagram} 
with $\widetilde{\alpha}$ vertical.

Actually, this definition does not depend on the choice of $\widetilde{\psi}$
as morphisms in $\Y$ are equivalence classes modulo $\sim$ which forgets about
all distinctions made by choice of cleavages. On objects $P^\op$ behaves like
$P$ and $P^\op([(\alpha,\varphi)]_\sim)$ is defined as $P(\varphi)$. 
The $P^\op$-cartesian arrows are the equivalence classes 
$[(\alpha,\varphi)]_\sim$ where $\alpha$ is a vertical isomorphism. 

Though most constructions appear more elegant from the fibrational point of 
view the construction of $P^\op$ from $P$ may appear as somewhat less 
immediate though (hopefully!) not too unelegant. Notice, however, that
for small fibrations, i.e.\ externalisations of internal categories, the
construction can be performed as in the case of presheaves of categories
as we have $P_{C^\op} \simeq P_C^\op$ for internal categories $C$.

Anyway, we have generalised now enough constructions from ordinary category
theory to the fibrational level so that we can perform (analogues of) the 
various constructions of (covariant and contravariant) functor categories on 
the level of fibrations. In particular, for a category $C$ internal to a 
category $\B$ with pullbacks we may construct the fibration 
$[P_C^\op{\to}P_\B]$ which may be considered as the fibration of (families of) 
$\B$-valued presheaves over the internal category $C$. Moreover, for categories
$C$ and $D$ internal to $\B$ the fibration of (families of) distributors from
$C$ to $D$ is given by $[P_D^\op{\times}P_C{\to}P_\B]$.\footnote{For an 
equivalent, but non-fibrational treatment of internal presheaves and 
distributors see \cite{Joh}.}

\newpage

\section{Internal Sums}

Suppose that $\C$ is a category. We will identify a purely fibrational 
property of the fibration $\Fam(\C) \to \Set$ equivalent to the requirement 
that the category $\C$ has small sums. This will provide a basis for 
generalising the property of ``having small sums'' to fibrations over 
arbitrary base categories with pullbacks.

Suppose that category $\C$ has small sums. Consider a family of objects
$A = (A_i)_{i \in I}$ and a map $u : I \to J$ in $\Set$. Then one may 
construct the family
$B := (\coprod_{i \in u^{-1}(j)} A_i)_{j \in J}$ together with the morphism
$(u,\varphi) : (I,A) \to (J,B)$ in $\Fam(\C)$ where 
$\varphi_i = {\mathrm{in}}_i : A_i \to B_{u(i)} = 
 \coprod_{k \in u^{-1}(u(i))} A_k$, i.e.\ the restriction of $\varphi$ to
$u^{-1}(j)$ is the cocone for the sum of the family $(A_i)_{i \in u^{-1}(j)}$.

One readily observes that $(u,\varphi) : A \to B$ satisfies the following
universal property: whenever $v : J \to K$ and $(v \circ u,\psi) : A \to C$
then there exists a unique $(v,\theta) : B \to C$ such that
$(v,\theta) \circ (u,\varphi) = (v \circ u,\psi)$, i.e.\ 
$\theta_{u(i)} \circ {\mathrm{in}}_i = \psi_i$ for all $i \in I$.
Arrows $(u,\varphi)$ satisfying this universal property are called 
\emph{cocartesian} and are determined uniquely up to vertical isomorphism.

Moreover, the cocartesian arrows of $\Fam(\C)$ satisfy the following 
so-called\footnote{Chevalley had this condition long before Beck who later
independently found it again.} \emph{Beck--Chevalley Condition} (BCC) 
which says that for every pullback
\begin{diagram}[small]
K \SEpbk & \rTo^{\widetilde{u}} & L \\
\dTo^{\widetilde{h}}  & \mbox{(1)} & \dTo_{h}  \\
I & \rTo_{u} & J \\
\end{diagram}
in $\Set$ and cocartesian arrow $\varphi : A \to B$ over $u$ it holds that
for every commuting square
\begin{diagram}[small]
C & \rTo^{\widetilde{\varphi}} & D \\
\dTo^{\widetilde{\psi}}  &  & \dTo_{\psi}  \\
A & \rTo_{\varphi} & B \\
\end{diagram}
over the pullback square (1) in $\B$ with $\psi$ and $\widetilde{\psi}$ 
cartesian the arrow $\widetilde{\varphi}$ is cocartesian, too.

Now it is a simple exercise to formulate the obvious generalisation to 
fibrations over an arbitrary base category with pullbacks.

\begin{Def}\label{intsumdef}
Let $\B$ be a category with pullbacks and $P : \X \to \B$ a fibration 
over $\B$.
An arrow $\varphi : X \to Y$ over $u : I \to J$ is called \emph{cocartesian}
iff for every $v : J \to K$ in $\B$ and $\psi : X \to Z$ over $v \circ u$ 
there is a unique arrow  $\theta : Y \to Z$ over $v$ with 
$\theta \circ \varphi = \psi$.\\
The fibration $P$ \emph{has internal sums} iff the following two conditions 
are satisfied.
\begin{enumerate}
\item[\rm (1)] For every $X \in P(I)$ and $u : I \to J$ in $\B$ there exists
a cocartesian arrow $\varphi : X \to Y$ over $u$.
\item[\rm (2)] The \emph{Beck--Chevalley Condition (BCC)} holds, i.e.\
for every commuting square in $\X$
\begin{diagram}[small]
C & \rTo^{\widetilde{\varphi}} & D \\
\dTo^{\widetilde{\psi}}  &  & \dTo_{\psi}  \\
A & \rTo_{\varphi} & B \\
\end{diagram}
over a pullback in the base it holds that $\widetilde{\varphi}$ is cocartesian
whenever $\varphi$ is cocartesian and $\psi$ and $\widetilde{\psi}$ are 
cartesian. \MYenddef
\end{enumerate}
\end{Def}

\bigskip \noindent
{\bf Remark.}\\
(1) One easily sees that for a fibration $P : \X \to \B$ an arrow $\varphi :
X \to Y$ is cocartesian iff for all $\psi : X \to Z$ over $P(\varphi)$ there
exists a unique vertical arrow $\alpha : Y \to Z$ with 
$\alpha \circ \varphi = \psi$.\\
(2) It is easy to see that BCC of Definition~\ref{intsumdef} is equivalent 
to the  requirement that for every commuting square in $\X$
\begin{diagram}[small]
C & \rTo^{\widetilde{\varphi}} & D \\
\dTo^{\widetilde{\psi}}  &  & \dTo_{\psi}  \\
A & \rTo_{\varphi} & B \\
\end{diagram}
over a pullback in the base it holds that $\psi$ is cartesian whenever
$\widetilde{\psi}$ is cartesian and $\varphi$ and $\widetilde{\varphi}$ 
are cocartesian.
 
\bigskip
Next we give a less phenomenological explanation of the concept of internal 
sums where, in particular, the Beck--Chevalley Condition arises in a less
\emph{ad hoc} way. For this purpose we first generalise the $\Fam$ 
construction from ordinary categories to fibrations.

\begin{Def}
Let $\B$ be a category with pullbacks and $P : \X \to \B$ be a fibration.
Then the \emph{family fibration $\Fam(P)$ for $P$} is defined as 
$P_\B \circ \FFam(P)$ where
\begin{diagram}[small]
P{\downarrow}\B \SEpbk & \rTo & \X \\
\dTo^{\FFam(P)} & & \dTo_{P} \\
\B^{\two} & \rTo_{\partial_0}& \B \\
\end{diagram}
The cartesian functor $\FFam(P) : \Fam(P) \to P_\B$ is called the
\emph{fibered family fibration of $P$}.

The cartesian functor $\eta_P : P \to \Fam(P)$ is defined as in the diagram
\begin{diagram}[small]
\X & & & & \\
   & \rdDashto(2,2)_{\eta_P} \rdEqual(4,2) & &  \\
\dTo^{P}  & & P{\downarrow}\B \SEpbk & \rTo & \X \\
   & & \dTo^{\FFam(P)} & & \dTo_{P} \\
\B & \rTo_{\quad\;\;\;\Delta_\B} & \B^{\two} & \rTo_{\partial_0} & \B \\
& \rdEqual &  \dTo_{\partial_1 = P_\B} & & \\
& & \B & & \\
\end{diagram}
where $\Delta_\B$ sends $u : I  \to J$ to 
\begin{diagram}[small]
I & \rTo^{u} & J \\ 
\dEqual & & \dEqual \\
I & \rTo_{u} & J \\
\end{diagram}
in $\B^{\two}$. More explicitly, $\eta_P$ sends $\varphi : X \to Y$ over 
$u : I \to J$ to
\begin{diagram}[small]
X & \rTo^{\varphi} & Y \\
I & \rTo^{u} & J \\
\dEqual & & \dEqual \\
I & \rTo_{u} & J \\
\end{diagram}
in $P{\downarrow}\B$. Obviously, the functor $\eta_P$ preserves 
cartesianness of arrows, i.e.\ $\eta_P$ is cartesian. \MYenddef
\end{Def}

\noindent
{\bf Remark.}\\
(1) If $\FFam(P)(\varphi)$ is cocartesian w.r.t.\ $P_\B$ then $\varphi$ is
cartesian w.r.t.\ $\FFam(P)$ iff $\varphi$ is cocartesian w.r.t.\ $\FFam(P)$.
Moreover, for every morphism 
\begin{diagram}[small]
A & \rTo^v & B \\
\dTo^a & & \dTo_b \\
I & \rTo_u & J
\end{diagram}
in $\B^{\two}$ we have 
\begin{diagram}[small]
A & \rTo^v & B & \rEqual & B & & A & \rEqual & A & \rTo^v & B \\
\dEqual & 1_v & \dEqual & \varphi_b & \dTo_b & = & 
\dEqual & \varphi_a & \dTo_a & & \dTo_b \\
A & \rTo_v & B & \rTo_b & I & & A & \rTo_a & I & \rTo_u & J
\end{diagram}
where $\varphi_a$ and $\varphi_b$ are cocartesian w.r.t.\ $P_\B$.
Using these two observations one can show that for fibrations $P$ and $Q$ 
over $\B$ a cartesian functor $F : \FFam(P) \to \FFam(Q)$ is determined
uniquely up to isomorphism by its restriction along the inclusion
$\Delta_\B : \B \to \B^{\two}$ from which it follows that $F$ is isomorphic
to $\FFam(\Delta_\B^*F)$. Thus, up to isomorphism all cartesian functors 
from $\FFam(P)$ to $\FFam(Q)$ are of the form $\FFam(F)$ for some cartesian 
functor $F : P \to Q$.

\noindent
(2) Notice, however, that not every cartesian functor $\Fam(P) \to \Fam(Q)$ 
over $\B$ is isomorphic to one of the form $\Fam(F)$ for some cartesian functor
$F : P \to Q$. An example for this failure is the cartesian functor
$\mu_P : \Fam^2(P) \to \Fam(P)$ sending $((X,v),u)$ to $(X,uv)$ for nontrivial
$\B$.\footnote{One can show that $\eta$ and $\mu$ are natural transformations
giving rise to a monad $(\Fam,\eta,\mu)$ on $\Fib(\B)$.}

\noindent
(3) If $\X$ is a category we write $\Fam(\X)$ for the category of families
in $\X$ and $\FFam(\X) : \Fam(\X) \to \Set$ for the family fibration.

The analogon of (1) in ordinary category theory is that for categories
$\X$ and $\Y$ a cartesian functor $F : \FFam(\X) \to \FFam(\Y)$ is isomorphic
to $\FFam(F_1)$ (the fiber of $F$ at $1 \in \Set$). 

The analogon of (2) in ordinary category theory is that not every ordinary 
functor $F : \Fam(\X) \to \Fam(\Y)$ is isomorphic to one of the form $\Fam(G)$
for some $G : \X \to \Y$.

\bigskip
Next we characterise the property of having internal sums in terms of
the family monad $\Fam$.

\begin{Thm}
Let $\B$ be a category with pullbacks and $P : \X \to \B$ be a fibration. 
Then $P$ has internal sums iff $\eta_P : P \to \Fam(P)$ has a fibered 
left adjoint $\coprod_P :  \Fam(P) \to P$ , i.e.\ $\coprod_P \dashv \eta_P$
where $\coprod_P$ is cartesian and unit and counit of the adjunction 
are cartesian natural transformations.
\end{Thm}
\begin{proof}
The universal property of the unit of the adjunction $\coprod_P \dashv \eta_P$
at $(u,X)$ is explicitated in the following diagram
\begin{diagram}[small]
& & & & Z \\
X & \rTo_{\eta_{(u,X)}} & Y & \ruTo(4,1)^{\psi} \ruDashto(2,1)_{\theta} & \\ 
\dDots & & \dDots & &  \dDots\\
I & \rTo^{u} & J & \rTo^{v}& K\\
\dTo^{u} & & \dEqual & &  \dEqual \\
J & \rEqual & J & \rTo_{v} & K \\
\end{diagram}
whose left column is the unit at $(u,X)$. From this it follows that 
$\eta_{(u,X)} : X \to Y$ is cocartesian over $u$.

Cartesianness of $\coprod_P$ says that the cartesian arrow $f$ as given by
\begin{diagram}[small]
X  & \rTo^{\psi}_\cart & Y \\
K \SEpbk & \rTo^{q} & L \\
\dTo^{p}  &  & \dTo_{v}  \\
I & \rTo_{u} & J \\
\end{diagram}
in $P{\downarrow}\B$ is sent by $\coprod_P$ to the cartesian arrow
$\coprod_P f$ over $u$ satisfying
\begin{diagram}[small]
X & \rTo^{\psi}_\cart & Y \\
\dTo^{\eta_{(p,X)}}  &  & \dTo_{\eta_{(v,Y)}}  \\
A & \rTo_{\coprod_P f}^\cart & B \\
\end{diagram}
where $\eta_{(p,X)}$ and $\eta_{(v,Y)}$ are the cocartesian units of the 
adjunction above $p$ and $v$, respectively. 
Thus, according to the second remark after Definition~\ref{intsumdef}
cartesianness of $\coprod_P$ is just the Beck--Chevalley Condition for
internal sums.

On the other hand if $P$ has internal sums then the functor $\coprod_P$ left
adjoint to $P$ is given by sending a morphism $f$ in $P{\downarrow}\B$ 
as given by 

\begin{diagram}[small]
X  & \rTo^{\psi} & Y \\
K  & \rTo^{q} & L \\
\dTo^{p}  &  & \dTo_{v}  \\
I & \rTo_{u} & J \\ 
\end{diagram}
to the morphism $\coprod_P f$ over $u$ satisfying
\begin{diagram}[small]
X & \rTo^{\psi} & Y \\
\dTo^{\varphi_1}  &  &   \dTo_{\varphi_2}\\
A & \rTo_{\coprod_P f} & B \\
\end{diagram}
where $\varphi_1$ and $\varphi_2$ are cocartesian over $p$ and $v$, 
respectively. It is easy to check that $\coprod_P$ is actually left adjoint 
to $\eta_P$ using for the units of the adjunction the cocartesian liftings
guaranteed for $P$. Cartesianness of $\coprod_P$ is easily seen to be
equivalent to the Beck--Chevalley condition.
\end{proof}

\newpage

\section{Internal Products}

Of course, by duality a fibration $P : \X \to \B$ has internal products iff
the dual fibration $P^\op$ has internal sums. After some explicitation (left
to the reader) one can see that the property of having internal products can 
be characterised more elementarily as follows.

\begin{Thm}
Let $\B$ be a category with pullbacks. Then a fibration $P : \X \to \B$ has 
internal products iff the following two conditions are satisfied.
\begin{itemize}
\item[\rm (i)] For every $u : I \to J$ in $\B$ and $X \in P(I)$ there is a span
$\varphi : u^*E \to E$, $\varepsilon : u^*E \to X$ with $\varphi$ cartesian
over $u$ and $\varepsilon$ vertical such that for every span 
$\theta : u^*Z \to Z$, $\alpha : u^*Z \to X$ with $\theta$ cartesian over $u$
and $\alpha$ vertical there is a unique vertical arrow $\beta : Z \to E$ 
such that $\alpha = \varepsilon \circ u^*\beta$ where $u^*\beta$ is the
vertical arrow with $\varphi \circ u^*\beta = \beta \circ \theta$ as 
illustrated in the diagram
\begin{diagram}[small]
 & & u^*Z & \rTo^{\theta}_\cart &  Z \\
 & \ldTo^{\alpha} & \dTo_{u^*\beta }& & \dTo_{\beta}\\ 
X & \lTo_{\varepsilon} & u^*E & \rTo_{\varphi}^\cart & E\\
\end{diagram} 
Notice that the span $(\varphi,\varepsilon)$ is determined uniquely up to
vertical isomorphism by this universal property and is called an
\emph{evaluation span for $X$ along $u$}.
\item[\rm (ii)] Whenever
\begin{diagram}[small]
L \SEpbk & \rTo^{\widetilde{v}} & I \\
\dTo^{\widetilde{u}}  & \mathrm{(1)} & \dTo_{u}  \\
K & \rTo_{v} & J \\
\end{diagram}
is a pullback in $\B$ and $\varphi :  u^*E \to E$, $\varepsilon : u^*E \to X$
is an evaluation span for $X$ along $u$ then for every diagram
\begin{diagram}[small]
\widetilde{v}^*X & \rTo^{\psi}_\cart & X \\
\uTo^{\widetilde{\varepsilon}} & & \uTo_{\varepsilon}\\
\widetilde{u}^*\widetilde{E}&\rTo^{\widetilde{\theta}}_\cart &u^*E\\
\dTo^{\widetilde{\varphi}} & &  \dTo_{\varphi}\\
\widetilde{E} & \rTo_{\theta}^\cart & E \\
\end{diagram}
where the lower square is above pullback \emph{(1)} in $\B$ and 
$\widetilde{\varepsilon}$ is vertical it holds that 
$(\widetilde{\varphi},\widetilde{\varepsilon})$ is an evaluation span 
for $\widetilde{v}^*X$ along $\widetilde{u}$.
\end{itemize}
\end{Thm}
\begin{proof}
Tedious, but straightforward explicitation of the requirement that
$P^\op$ has internal sums.
\end{proof}

\bigskip
Condition (ii) is called Beck--Chevally Condition (BCC) for internal products
and essentially says that evaluation spans are stable under reindexing. 

\bigskip\noindent
{\bf Examples.}\\
(1) ${\mathrm{Mon}}({\cal E})$ fibered over topos ${\cal E}$ has both internal
sums and internal products.\\
(2) For every category $\B$ with pullbacks the fundamental fibration 
$P_\B = \partial_1 : \B^{\two}\to\B$ has internal sums which look as follows
\begin{diagram}[small]
A & \rEqual & A \\
\dTo^{a} & & \dTo_{\coprod_u a} \\
I & \rTo_u & J \\
\end{diagram}
The fundamental fibration $P_\B$ has internal products iff for every 
$u : I \to J$ in $\B$ the pullback functor $u^* : \B/J \to \B/I$ has a
right adjoint $\prod_u$. For $\B = \Set$ this right adjoint gives 
\emph{dependent products} (as known from Martin-L\"of Type Theory).

\subsection*{Models of Martin--L\"of Type Theory}

A category $\B$ with finite limits such that its fundamental fibration
$P_\B$ has internal products---usually called a \emph{locally cartesian closed
category}---allows one to interpret $\Sigma$, $\Pi$ and Identity Types of  
Martin--L\"of Type Theory. 
Dependent sum $\Sigma$ and dependent product $\Pi$ are interpreted as 
internal sums and internal products.The fiberewise diagonal $\delta_a$ 

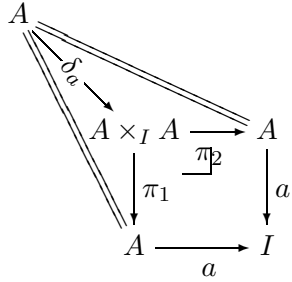
\begin{diagram}[small]
A  &            & & & \\
  & \rdTo~{\delta_a} \rdEqual(2,4) \rdEqual(4,2) & & & \\
  &            &   A\times_I A\SEpbk & \rTo_{\pi_2} &   A \\
  &            &  \dTo_{\pi_1}     &      &  \dTo_{a} \\
  &            &  A        & \rTo_{a} &   I \\
\end{diagram} 
is used for interpreting identity types: the sequent 
$i{:}I,x,y{:}A \vdash {\mathrm{Id}}_A(x,y)$ is interpreted as $\delta_a$
when $i{:}I \vdash A$ is interpreted as $a$.

One may interpret W-types in $\B$ iff for $b : B \to A$ and $a : A \to I$
there is a ``least'' $w : W \to I$ such that $W \cong \coprod_a \prod_b b^*w$
mimicking on a categorical level the requirement that $W$ is the ``least''
solution of the recursive type equation $W \cong \Sigma x{:}A.W^{B(x)}$.

\newpage

\section{Fibrations of Finite Limit Categories\\ and Complete Fibrations}

Let $\B$ be a category with pullbacks remaining fixed for this section.

\begin{Lem}\label{flf1}
For a fibration $P : \X \to \B$ we have that
\begin{enumerate}
\item[\rm (1)] a commuting square of cartesian arrows in $\X$ over 
               a pullback in $\B$ is always a pullback in $\X$
\item[\rm (2)] a commuting square
\begin{diagram}[small]
Y_1 & \rTo^{\varphi_1}_\cart & X_1 \\
\dTo^{\beta}  & & \dTo_{\alpha}  \\
Y_2 & \rTo_{\varphi_2}^\cart & X_2\\
\end{diagram}
in $\X$ is a pullback in $\X$ whenever the $\varphi_i$ are cartesian and 
$\alpha$ and $\beta$ vertical.
\end{enumerate}
\end{Lem}
\begin{proof}
Straightforward exercise.
\end{proof}

\begin{Def}\label{flfdef}
$P : \X \to \B$ is a \emph{fibration of categories with pullbacks} iff every
fiber $P(I)$ has pullbacks and these are stable under reindexing along 
arbitrary morphisms in the base. \MYenddef 
\end{Def}

\begin{Lem}\label{flf2}
If $P : \X \to \B$ is a fibration of categories with pullbacks then every
pullback in some fiber $P(I)$ is also a pullback in $\X$. 
\end{Lem}
\begin{proof}
Suppose
\begin{diagram}[small]
Z \SEpbk& \rTo^{\beta_2} & X_2 \\
\dTo^{\beta_1} & \mbox{($\dagger$)}& \dTo_{\alpha_2} \\
X_1 & \rTo_{\alpha_1}& Y \\
\end{diagram}
is a pullback in $P(I)$ and $\theta_1$, $\theta_2$ is a cone over $\alpha_1$,
$\alpha_2$ in $\X$, i.e.\ $\alpha_1 \circ \theta_1 = \alpha_2 \circ \theta_2$.
Obviously, $\theta_1$ and $\theta_2$ are above the same arrow $u$ in $\B$. 
For $i=1,2$ let $\varphi_i : u^*X_i \to X_i$ be a cartesian arrow over $u$ 
and $\gamma_i : V \to u^*X_i$ be a vertical arrow with 
$\varphi_i \circ \gamma_i = \theta_i$. As the image of $(\dagger)$ under $u^*$
is a pullback in its fiber there is a vertical arrow $\gamma$ with $\gamma_i =
u^*\beta_i \circ \gamma$ for $i=1,2$. The situation is illustrated in the
following diagram
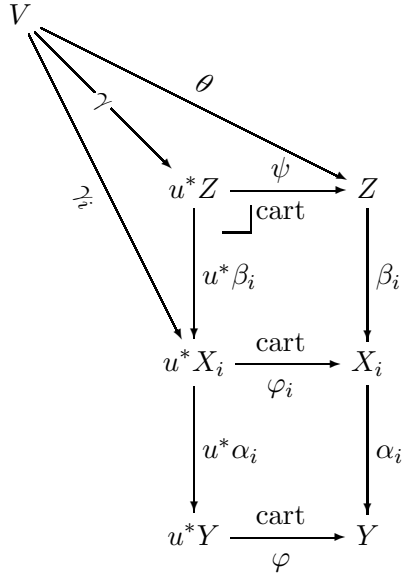
\begin{diagram}[small]
V  &            & & & \\
  & \rdTo~{\gamma} \rdTo(2,4)_{\gamma_i} \rdTo(4,2)^{\theta} & & & \\
  &            &   u^*Z \SEpbk & \rTo^{\psi}_\cart &   Z \\
  &            &  \dTo_{u^*\beta_i}     &      &  \dTo_{\beta_i} \\
  &            &  u^*X_i  & \rTo_{\varphi_i}^\cart &  X_i \\
  &            &  \dTo_{u^*\alpha_i}     &      &  \dTo_{\alpha_i} \\  
  &            &  u^*Y  & \rTo_{\varphi}^\cart &  Y \\
\end{diagram}
where $\varphi$ and $\psi$ are cartesian over $u$. From this diagram it is
obvious that $\theta := \psi \circ \gamma$ is a mediating arrow as desired.
If $\theta'$ were another such mediating arrow then for 
$\theta' = \psi \circ \gamma'$ with $\gamma'$ vertical it holds that 
$\gamma' = \gamma$ as both are mediating arrows to $u^*(\dagger)$ for the 
cone given by $\gamma_1$ and $\gamma_2$ and, therefore, it follows that 
$\theta = \theta'$. Thus $\theta$ is the unique mediating arrow.
\end{proof}

\bigskip
Now we can give a simple characterisation of fibrations of categories 
with pullbacks in terms of a preservation property.

\begin{Thm}\label{flf3}
$P : \X \to \B$ is a fibration of categories with pullbacks iff $\X$ has and
$P$ preserves pullbacks.
\end{Thm}
\begin{proof}
Suppose that $P : \X \to \B$ is a fibration of categories with pullbacks.
For $i=1,2$ let $f_i : Y_i \to X$ be arrows in $\X$ and 
$f_i = \varphi_i \circ \alpha_i$ be some vertical/cartesian factorisations.
Consider the diagram
\begin{diagram}[small]
U \SEpbk & \rTo^{\beta_2} & \SEpbk & \rTo^{\varphi_1''} & Y_2 \\
\dTo^{\beta_1} & \mbox{(4)} & \dTo^{\alpha_2'} &  \mbox{(3)}
& \dTo_{\alpha_2} \\
\SEpbk & \rTo^{\alpha_1'} & \SEpbk & \rTo^{\varphi_1'} & Z_2 \\
\dTo^{\varphi_2''} & \mbox{(2)}& \dTo^{\varphi_2'} & \mbox{(1)}
& \dTo_{\varphi_2} \\
Y_1 & \rTo_{\alpha_1} & Z_1 & \rTo_{\varphi_1} & X \\
\end{diagram}
where the $\varphi$'s are cartesian and the $\alpha$'s and $\beta$'s are 
vertical. Square (1) is a pullback in $\X$ over a pullback in $\B$ by
Lemma~\ref{flf1}(1). Squares (2) and (3) are pullbacks in $\X$ by 
Lemma~\ref{flf1}(2). Square (4) is a pullback in $\X$ by Lemma~\ref{flf2}.
Accordingly, the big square is a pullback in $\X$ over a pullback in $\B$.
Thus, $\X$ has and $P$ preserves pullbacks.

For the reverse direction assume that $\X$ has and $P$ preserves pullbacks.
Then every fiber of $P$ has pullbacks and they are preserved under reindexing
for the following reason. For every pullback
\begin{diagram}[small]
Z \SEpbk& \rTo^{\beta_2} & X_2 \\
\dTo^{\beta_1} & \mbox{($\dagger$)}& \dTo_{\alpha_2} \\
X_1 & \rTo_{\alpha_1}& Y \\
\end{diagram}
in $P(I)$ and $u: J \to I$ in $\B$ by Lemma~\ref{flf1} we have
\begin{diagram}[small]
u^*Z \SEpbk & \rTo^{\theta}_\cart & Z\\
\dTo^{u^*\beta_i}     &      &  \dTo_{\beta_i} \\
u^*X_i \SEpbk & \rTo_{\varphi_i}^\cart &  X_i \\
\dTo^{u^*\alpha_i}     &      &  \dTo_{\alpha_i} \\  
u^*Y  & \rTo_{\varphi}^\cart &  Y \\
\end{diagram}
and, therefore, the image of pullback $(\dagger)$ under $u^*$ is isomorphic 
to the pullback of $(\dagger)$ along $\varphi$ in $\X$. As pullback functors
preserve pullbacks it follows that the reindexing of $(\dagger)$ along $u$ is
a pullback, too.
\end{proof}

\begin{Def}\label{fibtermdef}
A fibration $P : \X \to \B$ is a \emph{fibration of categories with terminal
objects} iff every fiber $P(I)$ has a terminal object and these are stable
under reindexing.\MYenddef
\end{Def}

One easily sees that $P$ is a fibration of categories with terminal objects 
iff for every $I \in \B$ there is an object $1_I \in P(I)$ such that for every
$u : J \to I$ in $\B$ and $X \in P(J)$ there is a unique arrow $f : X \to 1_I$
in $\X$ over $u$. Such a $1_I$ is called an``$I$--indexed family of terminal
objects''. It is easy to see that this property is stable under reindexing.

\begin{Lem}\label{flf4}
Let $\B$ have a terminal object (besides having pullbacks). 
Then $P : \X \to \B$ is a fibration of categories with terminal objects iff 
$\X$ has a terminal object $1_\X$ with $P(1_\X)$ terminal in $\B$.
\end{Lem}
\begin{proof} 
Simple exercise. 
\end{proof}

\begin{Thm}\label{flf5}
For a category $\B$ with finite limits a fibration $P : \X \to \B$ is a 
fibration of categories with finite limits, i.e.\ all fibers of $P$ have 
finite limits preserved by reindexing along arbitrary arrows in the base, 
iff $\X$ has finite limits and $P$ preserves them.
\end{Thm}
\begin{proof}
Immediate from Theorem~\ref{flf3} and Lemma~\ref{flf4}.
\end{proof}

\bigskip
From ordinary category theory one knows that $\C$ has small limits iff $\C$ 
has finite limits and small products. Accordingly, one may define
``completeness'' of a fibration over a base category $\B$ with finite limits
by the requirements that
\begin{enumerate}
\item[(1)] $P$ is a fibration of categories with finite limits and
\item[(2)] $P$ has internal products (satisfying BCC).
\end{enumerate}
In \cite{Bor} vol.2, Ch.8 it has been shown that for a fibration $P$ complete 
in the sense above it holds for all $C \in \cat(\B)$ that the fibered 
``diagonal'' functor $\Delta_C : P \to [P_C{\to}P]$ has a fibered right 
adjoint $\prod_C$ sending diagrams of shape $C$ to their limiting cone 
(in the appropriate fibered sense).
Thus, requirement (2) above is necessary and sufficient for internal 
completeness under the assumption of requirement (1).

\newpage

\section{Elementary Fibrations and Representability}

A fibration $P:\X\to\B$ is called \emph{discrete} iff all its fibers are
discrete categories, i.e.\ iff $P$ reflects identity morphisms. However,
already in ordinary category theory discreteness of categories is not stable
under equivalence (though, of course, it is stable under isomorphism of
categories). Notice that a category $\C$ is equivalent to a discrete one
iff it is a \emph{posetal groupoid}, i.e.\ Hom--sets contain at most one 
element and all morphisms are isomorphisms. Such categories will be called
\emph{elementary}.

This looks even nicer from a fibrational point of view.

\begin{Thm}\label{elemfibchar}
Let $P : \X \to \B$ be a fibration. Then we have
\begin{itemize}
\item[\rm (1)] $P$ is a fibration of groupoids iff $P$ is \emph{conservative},
               i.e.\ $P$ reflects isomorphism. 
\item[\rm (2)] $P$ is a fibration of posetal categories iff $P$ is faithful.
\item[\rm (3)] $P$ is a fibration of elementary categories iff 
               $P$ is faithful and reflects isomorphisms. 
\end{itemize}
Fibrations $P : \X \to \B$ are called \emph{elementary} iff $P$ is faithful
and reflects isomorphisms.
\end{Thm}
\begin{proof}
Straightforward exercise.
\end{proof}
 
\bigskip
It is well known that a presheaf $A : \B^{\mathrm op} \to \Set$ is 
representable iff $\int A : \Elts(A) \to \B$ has a terminal object.
This motivates the following definition.

\begin{Def}
An elementary fibration $P : \X \to \B$ is \emph{representable} iff 
$\X$ has a terminal object, i.e.\ there is an object $R \in P(I)$ such that
for every $X \in \X$ there is a unique classifying morphism $u : P(X) \to I$
in $\B$ with $X \cong u^*R$, i.e.\ fibration $P$ is equivalent to 
$P_I = \partial_0 : \B/I \to \B$ for some $I \in \B$, i.e.\ $P$
is equivalent to some small discrete fibration over $\B$. \MYenddef
\end{Def}

\newpage

\section{Local Smallness}

\begin{Def}\label{locsmalldef}
Let $P : \X \to \B$ be a fibration. For objects $X,Y \in P(I)$ let 
$\Hom_I(X,Y)$ be the category defined as follows. Its objects are spans 
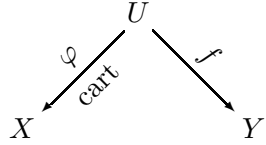
\begin{diagram}[small]
 & & U & & \\
 & \ldTo^{\varphi}_\cart & & \rdTo^{f} & \\
X & & & & Y \\
\end{diagram}
with $P(\varphi) = P(f)$ and $\varphi$ cartesian. A morphism from
$(\psi,g)$ to $(\varphi,f)$ is a morphism $\theta$ in $\X$ such that
$\varphi \circ \theta = \psi$ and $f \circ \theta = g$
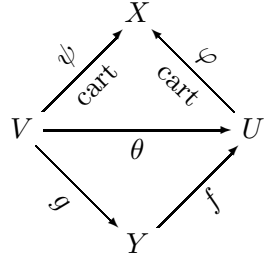
\begin{diagram}[small]
& & X & & \\
& \ruTo^{\psi}_\cart & & \luTo^{\varphi}_\cart & \\
V  & & \rTo_{\theta}& & U\\
& \rdTo_{g} & & \ruTo_{f}& \\
& & Y & & \\
\end{diagram} \\
Notice that $\theta$ is necessarily cartesian and fully determined by
$P(\theta)$. The category $\Hom_I(X,Y)$ is fibered over $\B/I$ by 
sending $(\varphi, f)$ to $P(\varphi)$ and $\theta$ to $P(\theta)$. 
The fibration $P$ is called \emph{locally small} iff for all $X,Y \in P(I)$
the elementary fibration $\Hom_I(X,Y)$ over $\B/I$ is representable, 
i.e.\ $\Hom_I(X,Y)$ has a terminal object.     \MYenddef
\end{Def}

The intuition behind this definition can be seen as follows. Let 
$(\varphi_0,f_0)$ be terminal in $\Hom_I(X,Y)$. Let 
$d := P(\varphi_0) : \homo_I(X,Y)\to I$. Let 
$f_0 = \psi_0 \circ \mu_{X,Y}$ with $\psi_0$ cartesian and $\mu_{X,Y}$ 
vertical. Then for every $u : J \to I$ and $\alpha : u^*X \to u^*Y$ there 
exists a unique $v : J \to \homo_I(X,Y)$ with $d \circ v = u$ such that 
$\alpha = v^*\mu_{X,Y}$ as illustrated in the following diagram.
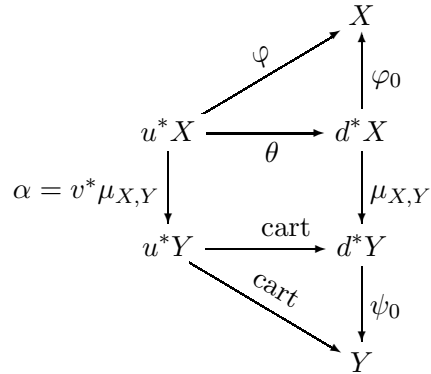
\begin{diagram}[small]
 & & X \\
 & \ruTo^{\varphi} & \uTo_{\varphi_0} \\
u^*X & \rTo_{\theta\quad} & d^*X \\
\dTo^{\alpha = v^*\mu_{X,Y}} & & \dTo_{\mu_{X,Y}} \\
u^*Y & \rTo^\cart  & d^*Y \\
 & \rdTo^\cart & \dTo_{\psi_0}\\
& & Y \\
\end{diagram} 

\begin{Thm}\label{lsbp}
Let $P : \X \to \B$ be a locally small fibration and $\B$ have binary products.
Then for all objects $X$, $Y$ in $\X$ there exist morphisms
$\varphi_0 : Z_0 \to X$ and $f_0 : Z_0 \to Y$ with $\varphi_0$ cartesian 
such that for morphisms $\varphi : Z \to X$ and $f : Z \to Y$ with 
$\varphi$ cartesian there exists a unique $\theta : Z \to Z_0$ making
the diagram

\goodbreak 

\begin{diagram}[small]
     &                   &  X  \\
     &  \ruTo^{\varphi}      &  \uTo_{\varphi_0} \\
Z    &  \rDashto_{\theta}   &  Z_0 \\
     &  \rdTo_{f}   &  \dTo_{f_0} \\
     &                   &  Y \\
\end{diagram}
commute.
\end{Thm}
\begin{proof}
Let $p : K \to I$, $q : K \to J$ be a product cone in $\B$. Then the desired
span $(\varphi_0 , f_0)$ is obtained by putting
$$\varphi_0 := \varphi_X \circ \widetilde{\varphi} \qquad
  f_0 := \varphi_Y \circ \widetilde{f}$$
where $(\widetilde{\varphi}, \widetilde{f})$ is terminal in 
$\Hom_K(p^*X,q^*Y)$ and $\varphi_X : p^*X \to X$ and $\varphi_Y : p^*Y \to Y$ 
are cartesian over $p$ and $q$, respectively. We leave it as a straightforward
exercise to verify that $(\varphi_0 , f_0)$ satisfies the desired universal 
property.
\end{proof}

\bigskip
Locally small categories are closed under a lot of constructions as e.g.\
finite products and functor categories. All these arguments go through for
locally small fibrations (see e.g.\ \cite{Bor} vol.~2, Ch.~8.6). There arises 
the question what it means that $\B$ fibered over itself is locally small. 
The answer given by the following Theorem is quite satisfactory.

\begin{Thm}\label{lslccc}
Let $\B$ be a category with pullbacks. Then the fundamental fibration
$P_\B = \partial_0 : \B^{\two} \to \B$ is locally small if and only if
for every $u : J \to I$ in $\B$ the pullback functor $u^* : \B/I \to \B/J$
has a right adjoint $\Pi_u$ or, equivalently, iff every slice of $\B$ is
cartesian closed. 
Such categories are usually called \emph{locally cartesian closed}.
\end{Thm}
\begin{proof}
Lengthy but straightforward exercise.
\end{proof}

\bigskip
Some further uses of local smallness are the following.

\begin{Obs}
Let $\B$ be a category with an initial object $0$ and $P : \X \to \B$ be
a locally small fibration. Then for $X,Y \in P(0)$ there is precisely one
vertical morphism from $X$ to $Y$.
\end{Obs}
\begin{proof}
Let $(\varphi_0,f_0)$ be terminal in $\Hom_0(X,Y)$. Then there is a
1--1--correspondence between vertical arrows $\alpha : X \to Y$ and 
sections $\theta$ of $\varphi_0$
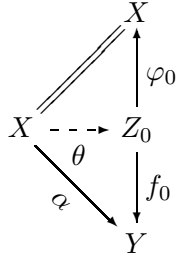
\begin{diagram}[small]
     &                   &  X  \\
     &  \ruEqual      &  \uTo_{\varphi_0} \\
X    &  \rDashto_{\theta}   &  Z_0 \\
     &  \rdTo_{\alpha}   &  \dTo_{f_0} \\
     &                   &  Y \\
\end{diagram}
As there is precisely one map from $0$ to $P(Z_0)$ there is precisely one 
section $\theta$ of $\varphi_0$. Accordingly, there is precisely one
vertical arrow $\alpha :X \to Y$.
\end{proof}

\begin{Obs}
Let $P : \X \to \B$ be a locally small fibration. Then every cartesian arrow
over an epimorphism in $\B$ is itself an epimorphism in $\X$.
\end{Obs}
\begin{proof}
Let $\varphi : Y \to X$ be cartesian with $P(\varphi)$ epic in $\B$. 
For $\varphi$ being epic in $\X$ it suffices to check that $\varphi$ is epic 
w.r.t.\ vertical arrows. Suppose that 
$\alpha_1 \circ \varphi = \alpha_2 \circ \varphi$ 
for vertical $\alpha_1,\alpha_2 : X \to Z$. Due to local smallness of $P$
there is a terminal object $(\varphi_0,f_0)$ in $\Hom_{P(X)}(X,Z)$. Thus, for 
$i{=}1,2$ there are unique cartesian arrows
$\psi_i$ with $\varphi_0 \circ \psi_i = \id_X$ and 
$f_0 \circ \psi_i = \alpha_i$. We have
\begin{enumerate}
\item[]  $\varphi_0 \circ \psi_1 \circ \varphi = \varphi = 
          \varphi_0 \circ \psi_2 \circ \varphi$ \qquad\qquad  and
\item[]  $f_0 \circ \psi_1 \circ \varphi = \alpha_1 \circ \varphi = 
           \alpha_2 \circ \varphi = f_0 \circ \psi_2 \circ \varphi$
\end{enumerate}
from which it follows that $\psi_1 \circ \varphi = \psi_2 \circ \varphi$. 
Thus, $P(\psi_1) \circ P(\varphi) = P(\psi_2) \circ P(\varphi)$ and, therefore,
as $P(\varphi)$ is epic by assumption it follows that $P(\psi_1) = P(\psi_2)$.
As $\varphi_0 \circ \psi_1 =\varphi_0 \circ \psi_2$ and $P(\psi_1) = P(\psi_2)$
it follows that $\psi_1 = \psi_2$ as $\varphi_0$ is cartesian. Thus, we finally
get $$\alpha_1 = f_0 \circ \psi_1 = f_0 \circ \psi_2 = \alpha_2$$
as desired.
\end{proof}

\bigskip\bigskip
Next we introduce the notion of generating family.

\begin{Def}\label{genfamdef}
Let $P : \X \to \B$ be a fibration. A \emph{generating family} for $P$ is
an object $G \in P(I)$ such that for every parallel pair of distinct vertical
arrows $\alpha_1, \alpha_2 : X \to Y$ there exist morphisms $\varphi : Z \to G$
and $\psi : Z \to X$ with $\varphi$ cartesian and 
$\alpha_1 \circ \psi \not= \alpha_2 \circ \psi$. 
\end{Def}

For locally small fibrations we have the following useful characterisation of
generating families provided the base has binary products.

\begin{Thm}\label{genfamthm}
Let $\B$ have binary products and $P : \X \to \B$ be a locally small fibration.
Then $G \in P(I)$ is a generating family for $P$ iff for every $X \in \X$
there are morphisms $\varphi_X : Z_X \to G$ and $\psi_X : Z_X \to X$ such
that $\varphi_X$ is cartesian and $\psi_X$ is \emph{collectively epic} in the
sense that vertical arrows $\alpha_1, \alpha_2 : X \to Y$ are equal 
iff $\alpha_1 \circ \psi_X = \alpha_2 \circ \psi_X$.
\end{Thm}
\begin{proof}
The implication from right to left is trivial. 

For the reverse direction suppose that $G \in P(I)$ is a generating family. 
Let $X \in P(J)$. As $\B$ is assumed to have binary products by 
Theorem~\ref{lsbp}
there exist $\varphi_0 : Z_0 \to G$ and $\psi_0 : Z_0 \to X$
with $\varphi_0$ cartesian such that for morphisms $\varphi : Z \to G$ and
$\psi : Z \to X$ with $\varphi$ cartesian there exists a unique 
$\theta : Z \to Z_0$ with
\begin{diagram}[small]
     &                   &  G  \\
     &  \ruTo^{\varphi}      &  \uTo_{\varphi_0} \\
Z    &  \rDashto_{\theta}   &  Z_0 \\
     &  \rdTo_{\psi}   &  \dTo_{\psi_0} \\
     &                   &  X \\
\end{diagram}
Now assume that $\alpha_1, \alpha_2 : X \to Y$ are distinct vertical arrows.
As $G$ is a generating family for $P$ there exist $\varphi : Z \to G$ and
$\psi : Z \to X$ with $\varphi$ cartesian and 
$\alpha_1 \circ \psi \not= \alpha_2 \circ \psi$. But there is a 
$\theta : Z \to Z_0$ with $\psi = \psi_0 \circ \theta$. Then we
have $\alpha_1 \circ \psi_0 \not= \alpha_2 \circ \psi_0$ (as otherwise
$\alpha_1 \circ \psi = \alpha_1 \circ \psi_0 \circ \theta =
\alpha_2 \circ \psi_0 \circ \theta = \alpha_2 \circ \psi$). Thus, we have shown
that $\psi_0$ is collectively epic and we may take $\varphi_0$ and $\psi_0$ as
$\varphi_X$ and $\psi_X$, respectively.
\end{proof}

\bigskip
Intuitively, this means that ``every object can be covered by a sum of 
$G_i$'s'' in case the fibration has internal sums.

\newpage

\section{Well-Poweredness}

For ordinary categories well-poweredness means that for every object the 
collection of its subobjects can be indexed by a set. Employing again the
notion of representability (of elementary fibrations) we can define a notion
of well--poweredness for (a wide class of) fibrations.

\begin{Def}\label{wellpowerdef}
Let $P : \X \to \B$ be a fibration where vertical monos are stable under 
reindexing. For $X \in P(I)$ let $\Sub_I(X)$ be the following category.
Its objects are pairs $(\varphi,m)$ where $\varphi : Y \to X$ is cartesian
and $m : S \to Y$ is a vertical mono. A morphism from $(\psi,n)$ to
$(\varphi,m)$ is a morphism $\theta$ such that $\varphi \circ \theta = \psi$
and $\theta \circ n = m \circ \widetilde{\theta}$
\begin{diagram}
T \SEpbk & \rTo^{\widetilde{\theta}}_\cart  & S \\
\dEmbed^{n} & & \dEmbed_{m} \\
Z & \rTo_{\theta}^\cart  & Y \\
 & \rdTo_{\psi} & \dTo_{\varphi}\\
 & & X \\
\end{diagram} 
for a (necessarily unique) cartesian arrow $\widetilde{\theta}$. The category
$\Sub_I(X)$ is fibered over $\B/I$ by sending objects $(\varphi,m)$ to 
$P(\varphi)$ and morphisms $\theta$ to $P(\theta)$.

The fibration $P$ is called \emph{well--powered} iff for every $I \in \B$ and
$X \in P(I)$ the elementary fibration $\Sub_I(X)$ over $\B/I$ is representable,
i.e.\ $\Sub_I(X)$ has a terminal object. \MYenddef
\end{Def}

If $(\varphi_X,m_X)$ is terminal in $\Sub_I(X)$ then, roughly speaking,
for every $u : J \to I$ and $m \in \Sub_{P(J)}(u^*X)$ there is a unique map 
$v : u \to P(\varphi_X)$ in $\B/I$ with $v^*(m_X) \cong m$.
We write $\sigma_X : \sub_I(X) \to I$ for $P(\varphi_X)$.

\medskip
Categories with finite limits whose fundamental fibration is well-powered 
have the following pleasant characterisation.

\begin{Thm}
A category $\B$ with finite limits is a topos if and only if its fundamental
fibration $P_\B = \partial_1 : \B^{\two} \to \B$ is well--powered.\\
Thus, in this particular case well--poweredness entails local smallness 
as every topos is locally cartesian closed.
\end{Thm}
\begin{proof}
Lengthy, but straightforward exercise.
\end{proof}

\bigskip
One may find it reassuring that for categories $\B$
with finite limits we have
\begin{itemize}
\item[]  $P_\B$ is locally small iff   $\B$ is locally cartesian closed
\item[]  $P_\B$ is wellpowered   iff   $\B$ is a topos
\end{itemize}
i.e.\ that important properties of $\B$ can be expressed by simple 
conceptual properties of the corresponding fundamental fibration.

\newpage

\section{Definability}

If $\C$ is a category and $(A_i)_{i \in I}$ is a family of objects in $\C$
then for every subcategory $\PP$ of $\C$ one may want to form the subset
\[ \{ i \in I \mid A_i \in \PP\}  \]
of $I$ consisting of all those indices $i \in I$ such that object $A_i$ 
belongs to the subcategory $\PP$. Though intuitively ``clear'' it is somewhat 
disputable from the point of view of \emph{formal axiomatic set theory} (e.g.\
ZFC or GBN) whether the set $\{ i \in I \mid A_i \in \PP\}$ actually exists.
The reason is that the usual \emph{separation axiom} guarantees the existence 
of (sub)sets of the form $\{ \, i \in I \mid P(i) \,\}$ only for predicates 
$P(i)$ that can be expressed\footnote{i.e.\ by a first order formula using 
just the binary relation symbols $=$ and $\in$} in the formal language of set 
theory. Now this may appear as a purely ``foundationalist'' argument to the 
working mathematician. However, we don't take any definite position w.r.t.\ 
this delicate foundational question but, instead, investigate the 
mathematically clean concept of \emph{definability} for fibrations.

\begin{Def}\label{stabclass}
Let $P : \X \to \B$ be a fibration. A class $\Cl \subseteq \Ob(\X)$ is called
\emph{$P$--stable} or simply \emph{stable} iff for $P$--cartesian arrows 
$\varphi : Y \to X$ it holds that $Y \in \Cl$ whenever $X \in \Cl$, i.e.\ iff 
the class $\Cl$ is stable under reindexing (w.r.t.\ $P$).
\MYenddef
\end{Def}

\begin{Def}\label{dfblclass}
Let $P : \X \to \B$ be a fibration. A stable class $\Cl \subseteq \Ob(\X)$ 
is called \emph{definable} iff for every $X \in P(I)$ there is a subobject 
$m_0 : I_0 \mono I$ such that
\begin{enumerate}
\item[\rm (1)]  $m_0^*X \in \Cl$ and
\item[\rm (2)]  $u : J \to I$ factors through $m_0$ 
                whenever $u^*X \in \Cl$. \MYenddef
\end{enumerate}
\end{Def}

Notice that  $u^*X \in \Cl$ makes sense as stable classes 
$\Cl \subseteq \Ob(\X)$ are necessarily closed under (vertical) isomorphisms.

\bigskip\noindent
{\bf Remark.}
If $\Cl \subseteq \Ob(\X)$ is stable then $\Cl$ is definable iff for every 
$X \in P(I)$ the elementary fibration $P_{\Cl,X} : \Cl_X \to \B/I$ is 
representable where $\Cl_X$ is the full subcategory of $\X/X$ on cartesian 
arrows and $P_{\Cl,X} = P_{/X}$ sends
\begin{diagram}[small]
Z & & \rTo^{\theta} & & Y \\
 & \rdTo_{\psi}^\cart & & \ldTo_{\varphi}^\cart & \\
 & & X & & \\
\end{diagram}
in $\Cl_X$ to
\begin{diagram}[small]
P(Z) & & \rTo^{P(\theta)} & & P(Y) \\
 & \rdTo_{P(\psi)} & & \ldTo_{P(\varphi)} & \\
 & & P(X) & & \\
\end{diagram}
in $\B/I$. Representability of the elementary fibration $P_{\Cl,X}$ then means 
that there is a cartesian arrow $\varphi_0 : X_0 \to X$ with $X_0 \in \Cl$ 
such that for every cartesian arrow $\psi : Z \to X$ with $Z \in \Cl$ there 
exists a unique arrow $\theta : Z \to X_0$ with $\varphi_0\circ\theta = \psi$.
This $\theta$ is necessarily cartesian and, therefore, already determined by 
$P(\theta)$. From uniqueness of $\theta$ it follows immediately that 
$P(\varphi_0)$ is monic. 

One also could describe the situation as follows. Every $X \in P(I)$ gives
rise to a subpresheaf $C_X$ of $\Yon_\B(P(X))$ consisting of the arrows 
$u : J \to I$ with $u^*X \in \Cl$. Then $\Cl$ is definable iff for every 
$X \in \X$ the presheaf $C_X$ is representable, i.e.\
\begin{diagram}[small]
\Cl_X & \rEmbed & \Yon_\B(P(X)) \\
\dTo_{\cong} & \ruEmbed_{\Yon_\B(m_X)} & \\
\Yon_\B(I_X) & & \\
\end{diagram}
where $m_X$ is monic as $\Yon_\B$ reflects monos.
\MYenddef

\medskip
Next we give an example demonstrating that \emph{definability is not vacuously
true}. Let $\C = \FinSet$ and $\X = \Fam(\C)$ fibered over $\Set$. Let 
$ \Cl \subseteq \Fam(\C)$ consist of those families $(A_i)_{i{\in}I}$ such 
that $\exists n \in \N.\, \forall i \in I. \,|A_i| \leq n$.
Obviously, the class $\Cl$ is stable but it is not definable as for the family
$$K_n = \{ i \in \N \mid i < n \} \qquad (n \in \N)$$
there is no greatest subset $P \subseteq \N$ with 
$\exists n \in \N.\, \forall i \in P. \,i < n$. Thus, the requirement of 
definability is non--trivial already when the base is $\Set$.

\medskip
For a fibration $P : \X \to \B$ one may consider the fibration
$P^{(\two)} : \X^{(\two)} \to \B$ of (vertical) arrows in $\X$. Thus, it is
clear what it means that a class $\MM \subseteq \Ob(\X^{(\two)})$ 
is ($P^{(\two)}$-)stable. Recall that $\Ob(\X^{(\two)})$ is the class of 
$P$--vertical arrows of $\X$. Then $\MM$ is stable iff for
all $\alpha : X \to Y$ in $\MM$ and cartesian arrows 
$\varphi : X' \to X$ and $\psi : Y' \to Y$ over the same arrow $u$ in $\B$ 
the unique vertical arrow $\alpha' : X' \to Y'$ with $\psi \circ \alpha' = 
\alpha \circ \varphi$ is in $\MM$, too. In other words 
$u^*\alpha \in \MM$ whenever $\alpha \in \MM$.

\begin{Def}\label{dfblclassmor}
Let $P : \X \to \B$ be a fibration and  $\MM$ a stable class of 
vertical arrows in $\X$. Then  $\MM$ is called \emph{definable} iff for
every $\alpha \in P(I)$ there is a subobject $m_0 : I_0 \mono I$ such that 
$m_0^*\alpha \in \MM$ and $u : J \to I$ factors through $m_0$ whenever 
$u^*\alpha \in \MM$. \MYenddef
\end{Def}

Next we discuss what is an appropriate notion of subfibration for a fibration
$P : \X \to \B$. Keeping in mind the analogy with $\Fam(\C)$ over $\Set$ we 
have to generalise the situation $\Fam(\Se) \subseteq \Fam(\C)$ where $\Se$ is
a subcategory of $\C$ which is \emph{replete} in the sense that 
${{\mathrm{Mor}}}(\Se)$ is stable under composition with isomorphisms in $\C$.
In this case the objects of $\Fam(\Se)$ are stable under reindexing and so 
are the vertical arrows of $\Fam(\Se)$. This motivates the following

\begin{Def}\label{sfdef}
Let $P : \X \to \B$. A \emph{subfibration} of $P$ is given by a subcategory
$\Z$ of $\X$ such that
\begin{enumerate}
\item[\rm (1)] cartesian arrows of $\X$ are in $\Z$ whenever their codomain
               is in $\Z$ (i.e.\ a cartesian arrow $\varphi : Y \to X$ is in 
               $\Z$ whenever $X \in \Z$) and
\item[\rm (2)] for every commuting square in $\X$
\begin{diagram}[small]
X' & \rTo^{\varphi}_\cart & X \\
\dTo^{f'} & & \dTo_{f} \\
Y' & \rTo_{\psi}^\cart & Y \\
\end{diagram}
 the morphism $f' \in \Z$ whenever $f \in \Z$ and 
$\varphi$ and $\psi$ are cartesian.\MYenddef
\end{enumerate}
\end{Def}

Notice that a subfibration $\Z$ of $P : \X \to \B$ is determined uniquely by 
${\mathcal V} \cap \Z$ where ${\mathcal V}$ is the class of vertical arrows
of $\X$ w.r.t.\ $P$. Thus, $\Z$ gives rise to \emph{replete} subcategories
$$S(I) = \Z \cap P(I) \qquad\qquad  (I \in \B)$$
which are stable under reindexing in the sense that for $u : J \to I$ in $\B$ 
\begin{itemize}\item[]
\begin{itemize}
\item[\rm ($S_{\mathrm{obj}}$)] \quad
$u^*X \in S(J)$ whenever $X \in S(I)$ \quad and
\item[\rm ($S_{\mathrm{mor}}$)] \quad
$u^*\alpha \in S(J)$ whenever $\alpha \in S(I)$.
\end{itemize}
\end{itemize}
On the other hand for every such such system $S = (S(I) \mid I \in \B)$ of
replete subcategories of the fibers of $P$ which is stable under reindexing
in the sense that the above conditions ($S_{\mathrm{obj}}$) and 
($S_{\mathrm{mor}}$) are satisfied we can define a subfibration $\Z$
of $P : \X \to \B$ as follows: $f : Y \to X$ in $\Z$ iff 
$X \in S(P(X))$ and $\alpha \in S(P(Y))$ where the diagram
\begin{diagram}[small]
Y & & \\
\dTo^{\alpha} & \rdTo^{f}& \\
Z & \rTo^\cart_{\varphi} & X \\
\end{diagram}
commutes and $\alpha$ is vertical and $\varphi$ is cartesian. Obviously, this
subcategory $\Z$ satisfies condition (1) of Definition~\ref{sfdef}. For 
condition (2) consider the diagram
\begin{diagram}[small]
X' & \rTo^{\varphi}_\cart & X \\
\dTo^{\alpha'} & & \dTo_{\alpha} \\
Z' & \rTo_\cart & Z \\
\dTo^{\theta'} & & \dTo_{\theta} \\
Y' & \rTo_{\psi}^\cart & Y \\
\end{diagram}
where $\alpha'$ and $\alpha$ are vertical, $\theta'$ and $\theta$ are 
cartesian and $f' = \theta' \circ \alpha'$ and $f = \theta \circ \alpha$ 
from which it is clear that $\alpha' \in S(P(X'))$ whenever 
$\alpha \in S(P(X))$ and, therefore, $f' \in \Z$ whenever $f \in \Z$.

\smallskip
Now we are ready to define the notion of definability for subfibrations.
\begin{Def}\label{subfibdfbl}
A subfibration $\Z$ of a fibration $P : \X \to \B$ is \emph{definable} iff
$\Cl = \Ob(\Z)$ and $\MM = {\mathcal{V}} \cap \Z$ are definable classes of 
objects and morphisms, respectively. \MYenddef
\end{Def}

Without proof we mention a couple of results illustrating the strength of 
definability. Proofs can be found in \cite{Bor} vol.2, Ch.8.
\begin{enumerate}
\item[(1)]
Locally small fibrations are closed under definable subfibrations. 
\item[(2)]
Let $P : \X \to \B$ be a locally small fibration over $\B$ with finite limits. 
Then the class of vertical isomorphisms of $\X$ is a definable subclass of 
objects of $\X^{(\two)}$ w.r.t.\ $P^{(\two)}$.
\item[(3)]
If, moreover, $\X$ has finite limits and $P$ preserves them then vertical 
monos (w.r.t.\ their fibers) form a definable subclass of objects of 
$\X^{(\two)}$ w.r.t.\ $P^{(\two)}$ and fiberwise terminal objects form a 
definable subclass of objects of $\X$ w.r.t.\ $P$.
\item[(4)]
Under the assumptions of (3) for every finite category $\D$ the fiberwise 
limiting cones of fiberwise $\D$--diagrams from a definable class.
\end{enumerate}

A pleasant consequence of (3) is that under the assumptions of (3) the class
of pairs of the form $(\alpha,\alpha)$ for some vertical arrow $\alpha$ form
a definable subclass of the objects of $\X^{(\Gra)}$ w.r.t.\ $P^{(\Gra)}$ 
where $\Gra$ is the category with two objects $0$ and $1$ and two nontrivial 
arrows from $0$ to $1$. In other words under the assumptions of (3) 
\emph{equality of morphisms is definable}.

On the negative side we have to remark that for most fibrations the class
$\{ (X,X) \mid X \in \Ob(X)\}$ is not definable as a subclass of $\X^{(2)}$ 
(where $2 = 1+1$ is the discrete category with 2 objects) simply because this 
class is not even stable (under reindexing). Actually, stability fails already
if some of the fibers contains different isomorphic objects! This observation
may be interpreted as confirming the old suspicion that equality of objects 
is somewhat ``fishy'' at least for non--split fibrations. 
Notice, however, that even for discrete split fibrations equality need not be 
definable which can be seen as follows. Consider a presheaf 
$A \in \widehat{\Gra}$ (where $\Gra$ is defined as in the previous paragraph)
which may most naturally be considered as a directed graph. Then for $A$
considered as a discrete split fibration equality of objects is definable if
and only if $A$ is subterminal, i.e.\ both $A(1)$ and $A(0)$ contain at most
one element. Thus, for interesting graphs equality of objects is not definable!

\smallskip
We conclude this section with the following positive result.

\begin{Thm}\label{descent}
Let $\B$ be a topos and $P : \X \to \B$ a fibration. If $\Cl$ is a definable 
class of objects of $\X$ (w.r.t.\ $P$) then for every cartesian arrow 
$\varphi : Y \to X$ over an epimorphism in $\B$ we have that $X \in \Cl$ iff 
$Y \in \Cl$ (often refered to as ``descent property'').
\end{Thm}
\begin{proof} 
The implication from left to right follows from stability of $\Cl$. 

For the reverse direction suppose that $\varphi : Y \to X$ is cartesian over
an epi $e$ in $\B$. Then by definability of $\Cl$ we have $e = m \circ f$ 
where $m$ is a mono in $\B$ with $m^*X \in \Cl$. But as $e$ is epic and $m$ 
is monic and we are in a topos it follows that $m$ is an isomorphism and, 
therefore, $X \cong m^*X \in \Cl$.
\end{proof}

\bigskip
Notice that this Theorem can be generalised to regular categories $\B$
where, however, one has to require that $P(\varphi)$ is a regular epi 
(as a monomorphism $m$ in
a regular category is an isomorphism if $m \circ f$ is a regular epimorphism
for some morphism $f$ in $\B$).

\newpage

\section{Preservation Properties of Change of Base}

We know already that for an arbitrary functor $F : \A \to \B$ we have that 
$F^*P \in {\bf Fib}(\A)$ whenever $P \in {\bf Fib}(\B)$. The (2-)functor 
$F^* : {\bf Fib}(\B) \to {\bf Fib}(\A)$ is known as \emph{change of base 
along $F$}. In this section we will characterize those functors $F$ which by 
change of base along $F$ preserve ``all good properties of fibrations''.

\begin{Lem}\label{cbsm}
Let $F : \A \to \B$ be a functor. Then $F^* : \Fib(\B) \to \Fib(\A)$
preserves smallness of fibrations if and only if $F$ has a right adjoint $U$.
\end{Lem}
\begin{proof}
Suppose that $F$ has a right adjoint $U$. If $C \in \cat(\B)$ then
$F^*P_C$ is isomorphic to $P_{U(C)}$ where $U(C)$ is the image of $C$ under
$U$ which preserves all existing limits as it is a right adjoint.

Suppose that $F^*$ preserves smallness of fibrations. Consider for $I \in \B$
the small fibration $P_I = \underline{I} = \partial_0 : \B/I \to \B$. 
Then $F^*P_I$ is isomorphic to $\partial_0 : F{\downarrow}I \to \B$ which
is small iff there exists $U(I) \in \A$ such that $F^*P_I \cong P_{U(I)}$,
i.e.\ $\B(F(-),I) \cong \A(-,U(I))$. Thus, if $F^*$ preserves smallness
of fibrations then for all $I \in \B$ we have $\B(F(-),I) \cong \A(-,U(I))$
for some $U(I) \in \A$, i.e.\ $F$ has a right adjoint $U$.
\end{proof}

\bigskip
As a consequence of Lemma~\ref{cbsm} we get that for $u : I \to J$ in $\B$
change of base along $\Sigma_u : \B/I \to \B/J$ preserves smallness of 
fibrations iff $\Sigma_u$ has a right adjoint $u^* : \B/J \to \B/I$, i.e.\ 
pullbacks in $\B$ along $u$ do exist. Analogously, change of base along 
$\Sigma_I : \B/I \to \B$ preserves smallness of fibrations iff $\Sigma_I$ has 
a right adjoint $I^*$, i.e.\ for all $K \in \B$ the cartesian product of $I$
and $K$ exists. One can show that change of base along $u^*$ and $I^*$
is right adjoint to change of base along $\Sigma_u$ and $\Sigma_I$, 
respectively. Thus, again by Lemma~\ref{cbsm} change of base along $u^*$ 
and $I^*$  preserves smallness of fibrations iff $u^*$ and $I^*$ have right
adjoints $\Pi_u$ and $\Pi_I$, respectively.

\bigskip
From now on we make the reasonable assumption that all base categories have 
pullbacks as otherwise their fundamental fibrations would not exist.

\begin{Lem}\label{cbis}
Let $\A$ and $\B$ be categories with pullbacks and $F : \A \to \B$ an 
arbitrary functor. Then the following conditions are equivalent
\begin{enumerate}
\item[\rm (1)] $F$ preserves pullbacks
\item[\rm (2)] $F^* : \Fib(\B) \to \Fib(\A)$ preserves the property
               of having internal sums
\item[\rm (3)] $\partial_1 :  \B{\downarrow}F \to \A$ has internal sums.
\end{enumerate}
\end{Lem}
\begin{proof}
The implications (1) $\Rightarrow$ (2) and (2) $\Rightarrow$ (3) are easy.
The implication  (3) $\Rightarrow$ (1) can be seen as follows. Suppose that
the bifibration $\partial_1 :  \B{\downarrow}F \to \A$ satisfies BCC then
for pullbacks
\begin{diagram}[small]
L \SEpbk & \rTo^{q} & K \\
\dTo^{p} &          & \dTo_{v} \\
J & \rTo_{u} & I \\
\end{diagram}
in $\A$ we have 
\begin{diagram}[small]
F(L) \SEpbk &  &  \rTo^{F(q)}  &  &  F(K)  &  &  \\
   & \rdDashto_{\alpha}^{\cong} & & & \dEqual & \rdEqual & \\ 
\dEqual & &  \SEpbk & \rTo & \HonV & & F(K) \\
& & \dTo & & \dEqual & & \\
F(L) & \hLine & \VonH & \rTo^{F(q)} & F(K) & & \dTo_{F(v)} \\
  & \rdTo_{F(p)} & & & & \rdTo^{F(v)} & \\
  &            & F(J)   &   &  \rTo_{F(u)} & & F(I) \\
\end{diagram}
As back and front face of the cube are cartesian arrows and the right face is 
a cocartesian arrow it follows from the postulated BCC for 
$\partial_1 :  \B {\downarrow} F \to \A$ that the left face is a cocartesian
arrow, too. Thus, the map $\alpha$ is an isomorphism from which it follows that
\begin{diagram}[small]
F(L) \SEpbk & \rTo^{F(q)} & F(K) \\
\dTo^{F(p)} &          & \dTo_{F(v)} \\
F(J) & \rTo_{F(u)} & F(I) \\
\end{diagram}
is a pullback as required.
\end{proof} 

\bigskip
Thus, by the previous two lemmas a functor $F : \A \to \B$ between categories
with pullbacks necessarily has to preserve pullbacks and have a right adjoint
$U$ whenever $F^*$ preserves ``all good properties of fibrations'' as being 
small and having internal sums certainly are such ``good properties''.

Actually, as pointed out by B\'enabou in his 1980 Louvain-la-Neuve lectures
\cite{Ben} these requirements for $F$ are also sufficient for $F^*$ preserving 
the following good properties of fibrations 
\begin{itemize}
\item (co)completeness
\item smallness
\item local smallness
\item definability
\item well--poweredness.
\end{itemize}
We will not prove all these claims but instead discuss \emph{en detail}
preservation of local smallness. Already in this case the proof is 
paradigmatic and the other cases can be proved analogously.

\begin{Lem}\label{lspres}
If $F : \A \to \B$ is a functor with right adjoint $U$ and $\A$ and $\B$ 
have pullbacks then change of base along $F$ preserves local smallness 
of fibrations.
\end{Lem}
\begin{proof}
Suppose $P \in \Fib(\B)$ is locally small. Let $X,Y \in P(FI)$ and
\begin{diagram}[small]
       &   &                  X  \\
      & \ruTo^{\varphi_0}_\cart & \\
d^*X   &    &  \\ 
& \rdTo_{\psi_0} & \\
   &    & Y
\end{diagram}
with $d = P(\varphi_0) = P(\psi_0) : \hom_{FI}(X,Y) \to FI$ be the terminal 
such span. Then consider the pullback (where we write $H$ for $\hom_{FI}(X,Y)$)
\begin{diagram}[small]
\widetilde{H} \SEpbk  &  \rTo^{h}  & UH \\
\dTo^{\widetilde{d}} & & \dTo_{Ud}  \\
I & \rTo_{\eta_I} & UFI \\
\end{diagram}
in $\A$ where $\eta_I$ is the unit of $F \dashv U$ at $I$. Then there is a 
natural bijection between $v : u \to \widetilde{d}$ in $\A/I$ and 
$\widehat{v} : F(u) \to d$ in $\B/FI$ by sending $v$ to 
$\widehat{v} = \varepsilon_H \circ F(h) \circ F(v)$ 
where $\varepsilon_H$ is the counit of $F \dashv U$ at $H$.

Let $\theta_1$ be a $P$-cartesian arrow over 
$\varepsilon_H \circ F(h) : F\widetilde{H} \to H$ to $d^*X$. Let 
$\varphi_1 := \varphi_0 \circ \theta_1$ and $\psi_1 := \psi_0 \circ \theta_1$
which are both mapped by $P$ to
\[ d \circ \varepsilon_H \circ F(h)  = 
   \varepsilon_{FI} \circ F(Ud) \circ F(h) =
   \varepsilon_{FI} \circ F\eta_I \circ F\widetilde{d} = F\widetilde{d}
\]
We show now that the span
$\bigl((\widetilde{d},\varphi_1),((\widetilde{d},\psi_1)\bigr)$ is a terminal 
object in the category $\Hom_I(X,Y)$ for $F^*P$.
For that purpose suppose that $u : J \to I$ in $\A$ and $\varphi : Z \to X$, 
$\psi : Z \to Y$ are arrows over $u$  w.r.t.\ $F^*P$ with $\varphi$ cartesian
w.r.t.\ $F^*P$. 

There exists a unique $P$-cartesian arrow $\theta_2$ with
$\varphi = \varphi_0 \circ \theta_2$ and $\psi = \psi_0 \circ \theta_2$. 
For $\widehat{v} := P(\theta_2)$ we have $d \circ \widehat{v} = F(u)$ 
as $P(\varphi) = F(u) = P(\psi)$. Then there exists a unique map 
$v : u \to \widetilde{d}$ with 
$\varepsilon_H \circ F(h) \circ F(v) = \widehat{v}$. 
Now let $\theta$ be the unique $P$-cartesian arrow over $F(v)$ 
with $\theta_2 = \theta_1 \circ \theta$ which exists 
as $P(\theta_1) = \varepsilon_H \circ F(h)$ and $P(\theta_2) = \widehat{v} =
\varepsilon_H \circ F(h) \circ F(v)$. Thus, we have $v : u \to \widetilde{d}$
and a cartesian arrow $\theta$ with 
$P(\theta) = F(v)$, $\varphi_1 \circ \theta = \varphi$
and $\psi_1 \circ \theta = \psi$ as desired.

For uniqueness of $(v,\theta)$ with this property suppose that
$v' : u \to \widetilde{d}$ and $\theta'$ is a cartesian arrow with
$P(\theta') = F(v')$, $\varphi_1 \circ \theta' = \varphi$ and 
$\psi_1 \circ \theta' = \psi$. 
From the universal property of $(\varphi_0,\psi_0)$ it follows that
$\theta_2 = \theta_1 \circ \theta'$. Thus, we have
\[ \widehat{v} = P(\theta_2) = P(\theta_1) \circ P(\theta') = 
   \varepsilon_H \circ F(h) \circ P(\theta') =
   \varepsilon_H \circ F(h) \circ F(v')
\]
from which it follows that $v = v'$ as by assumption we also have 
$\widetilde{d} \circ v' = u$. From $\theta_2 = \theta_1 \circ \theta'$ and 
$P(\theta') = F(v') = F(v)$ it follows that $\theta' = \theta$ because 
$\theta_1$ is cartesian and we have $\theta_2 = \theta_1 \circ \theta$ and
$P(\theta') = F(v)$ due to the construction of $\theta$.
\end{proof}

\bigskip
Analogously one shows that under the same premisses as in Lemma~\ref{lspres}
the functor $F^*$ preserves well--poweredness of fibrations and that, for
fibrations $P : \X \to \B$ and definable classes $\Cl \subseteq \X$,
the class 
$$F^*(\Cl) := \{(I,X) \mid X \in P(FI) \wedge X \in \Cl \}$$ 
is definable w.r.t.\ $F^*P$.

\bigskip\noindent
{\bf Warning.} If $F \dashv U : \E \to \Se$ is an unbounded geometric morphism
then $P_\E = \partial_1 : \E^\two \to \E$ has a generic family 
though $ P_F \cong F^*P_\E$ does not have a generic family
as otherwise by Theorem~\ref{bgmthm} (proved later on) the geometric morphism
$F \dashv U$ were bounded! Thus, the property of having a generating family is
not preserved by change of base along functors that preserve finite limits
and have a right adjoint. In this respect the property of having a small 
generating family is not as ``good'' as the other properties of fibrations 
mentioned before which are stable under change of base along functors 
that preserve pullbacks and have a right adjoint. \MYenddef

\bigskip
The moral of this section is that functors $F$ between categories with 
pullbacks preserve ``all good'' (actually ``most good'') properties of 
fibrations by change of base along $F$ if and only if $F$ preserves 
pullbacks and has a right adjoint. In particular, this holds for inverse image
parts of geometric morphisms, i.e.\ finite limit preserving functors having a 
right adjoint. But there are many more examples which are also important. 
Let $\B$ be a category with pullbacks and $u : I \to J$ a morphism in $\B$ then
$\Sigma_u : \B/I \to \B/J$ preserves pullbacks and has a right adjoint, namely
the pullback functor $u^* : \B/J \to \B/I$, but $\Sigma_u$ preserves 
terminal objects if and only if $u$ is an isomorphism. Notice that for 
$I \in \B$ the functor $\Sigma_I = \partial_0 : \B/I \to \B$ always preserves 
pullbacks but has a right adjoint $I^*$ if and only if for all $K \in \B$ the 
cartesian product of $I$ and $K$ exists. Thus, for a category $\B$ with 
pullbacks the functors $\Sigma_I : \B/I \to \B$ preserve ``all good 
properties'' of fibrations by change of base if and only if $\B$ has all 
binary products (but not necessarily a terminal object!). 

A typical such example is the full subcategory $\F$ of $\Set^{\N}$ on those 
$\N$-indexed families of sets which are empty for almost all indices. Notice,
however, that every slice of $\F$ actually is a (Grothendieck) topos. This
$\F$ is a typical example for B\'enabou's notion of \emph{partial topos}, 
i.e.\ a category with binary products where every slice is a topos.
The above example can be generalised easily. Let $\E$ be some topos and 
$F$ be a (downward closed) subset of $\Sub_\E(1_\E)$ then $\E_{/F}$, the full
subcategory of $\E$ on those objects $A$ whose terminal projection factors
through some subterminal in $F$, is a partial topos whose subterminal objects
form a full reflective subcategory of $\E_{/F}$ and have binary infima. 

\bigskip\noindent
{\bf Exercise.} Let $\B$ be an arbitrary category. Let ${{\mathrm{st}}}(\B)$
be the full subcategory of $\B$ on \emph{subterminal} objects, i.e.\ objects
$U$ such that for every $I \in \B$ there is at most one arrow $I \to U$ 
(possibly none!). We say that $\B$ \emph{has supports} iff 
${{\mathrm{st}}}(\B)$ is a (full) reflective subcategory of $\B$.

Show that for a category $\B$ having pullbacks and supports it holds that
$\B$ has binary products iff ${{\mathrm{st}}}(\B)$ has binary meets!

\newpage

\section{Adjoints to Change of Base}

We first show that for a functor $F : \A \to \B$ there is a left (2-)adjoint
$\coprod_F$ and a right (2-)adjoint $\prod_F$ to $F^* : \Fib(\B) \to \Fib(\A)$,
i.e.\ change of base along $F$.

The right (2-)adjoint $\prod_F$ is easier to describe as its behaviour
is prescribed by the fibered Yoneda Lemma as 
\[ \prod_F(P)(I) \simeq \Fib(\B)(\underline{I},\prod_F(P)) \simeq
  \Fib(\A)(F^*\underline{I},P)
\]
for $I\in\B$. Accordingly, one verifies easily that the right adjoint 
$\prod_F$ to $F^*$ is given by
$$\prod_F(P)(I) = \Fib(\A)(F^*\underline{I},P) \qquad\quad
  \prod_F(P)(u) = \Fib(\A)(F^*\underline{u},P)$$
for objects $I$ and morphisms $u$ in $\B$. Obviously, as expected if $\B$ is 
terminal then $\prod P=\prod_F P$ is the category of all cartesian sections 
of $P$.

Notice further that in case $F$ has a right adjoint $U$ then 
$F^*\underline{I} \cong \underline{UI}$ and, accordingly, we have 
$\prod_F \simeq U^*$.

We now turn to the description of $\coprod_F$. We consider first the simpler
case where $\B$ is terminal. Then one easily checks that for a fibration
$P : \X\to\A$ the sum $\coprod P=\coprod_F P$ is given by $\X[\Cart(P)^{-1}]$,
i.e.\ the category obtained from $\X$ be freely inverting all cartesian arrows.
This we can extend to the case of arbitrary functors $F$ as follows. 
For $I\in\B$ consider the pullback $P_{(I)}$ of $P$ along 
$\partial_1 : I/F \to \A$
\begin{diagram}[small]
\X_{(I)} \SEpbk &  \rTo & \X \\
\dTo^{P_{(I)}} & & \dTo_{P} \\
I/F & \rTo_{\partial_1} & \A
\end{diagram}
and for $u : J \to I$ in $\B$ let $G_u$ the mediating cartesian functor 
from $P_{(I)}$ to $P_{(J)}$ over $\underline{u}/F$ (precomposition by $u$) 
in the diagram
\begin{diagram}[small]
\X_{(I)} \SEpbk & \rTo^{G_u} & \X_{(J)} \SEpbk & \rTo & \X \\
\dTo^{P_{(I)}} & & \dTo^{P_{(J)}} & & \dTo_{P} \\
I/F & \rTo_{\underline{u}/F} & J/F & \rTo_{\partial_1} & \A
\end{diagram}
bearing in mind that $\partial_1 \circ \underline{u}/F = \partial_1$. 
Now the reindexing functor $\coprod_F(u) : \coprod_F(I) \to \coprod_F(J)$ 
is the unique functor $H_u$ with
\begin{diagram}
\X_{(I)}[\Cart(P_{(I)})^{-1}] & \rTo^{H_u} & \X_{(J)}[\Cart(P_{(J)})^{-1}] \\
\uTo^{Q_I} & & \uTo_{Q_J} \\
\X_{(I)} & \rTo_{G_u} & \X_{(J)}
\end{diagram}
which exists as $Q_J \circ G_u$ inverts the cartesian arrows of $X_{(I)}$.

Notice, however, that due to the non--local character\footnote{Here we mean 
that $\X_{(I)}[\Cart(P_{(I)})]$ and $\Cart(F^*\underline{I},P)$ do not depend
only on $P(I)$, the fiber of $P$ over $I$. This phenomenon already turns up
when considering reindexing of presheaves which in general for does not 
preserve exponentials for example.} of the construction of $\coprod_F$ 
and $\prod_F$ in general the Beck--Chevalley Condition does not hold 
for $\coprod$ and $\prod$. 

As for adjoint functors $F \dashv U$ we have $\prod_F \simeq U^*$ it follows
that $F^* \simeq \coprod_U$.

\bigskip
Now we will consider change of base along distributors. 
Recall that a distributor $\phi$ from $\A$ to $\B$ (notation 
$\phi:\A\nrightarrow\B$) is a functor from $\B^\op{\times}\A$ to $\Set$, or 
equivalently, a functor from $\A$ to $\widehat{\B} = \Set^{\B^\op}$. 
Of course, up to isomorphism distributors from $\A$ to $\B$ are in 
1--1--correspondence with cocontinuous functors from $\widehat{\A}$ to 
$\widehat{\B}$ (by left Kan extension along $\Yon_\A$). 
Composition of distributors is defined in terms of composition of the 
associated cocontinuous functors.\footnote{As the correspondence between 
distributors and cocontinuous functors is only up to isomorphism composition 
of distributors is defined also only up to isomorphism. That is the reason 
why distributors do form only a bicategory and not an ordinary category!} 
For a functor $F:\A\to\B$ one may define a distributor 
$\phi_F : \A\nrightarrow\B$ as $\phi_F(B,A) = \B(B,FA)$ 
and a distributor $\phi^F : \B\nrightarrow\A$ in the reverse direction as 
$\phi^F(A,B) = \B(FA,B)$. Notice that $\phi_F$ corresponds
to $\Yon_\B \circ F$ and $\phi^F$ is right adjoint to $\phi_F$.

For a distributor $\phi : \A\to\widehat{\B}$ change of base along $\phi$
is defined as follows 
(identifying presheaves over $\B$ with their corresponding discrete fibrations)
$$\phi^*(P)(I) = \Fib(\B)(\phi(I),P) \qquad\qquad
  \phi^*(P)(u) = \Fib(\B)(\phi(u),P)$$
for objects $I$ and morphisms $u$ in $\A$. From this definition one easily 
sees that for a functor $F : \A\to\B$ change of base along $\phi^F$ coincides
with $\prod_F$, i.e.\ we have
$$(\phi^F)^*P \cong \prod_F P$$
for all fibrations $P$ over $\A$. 

This observation allows us to reduce change of base along distributors 
to change of base along functors and their right adjoints. The reason is 
that every distributor $\phi:\A\nrightarrow\B$ can be factorised as a 
composition of the form $\phi^G\phi_F$.\footnote{Let $F$ and
$G$ be the inclusions of $\A$ and $\B$, respectively, into the \emph{display
category} $\D_\phi$ of $\phi$ which is obtained by adjoining to the disjoint 
union of $\A$ and $\B$ the elements of $\phi(B,A)$ as morphisms from $B$ to 
$A$ and defining $u{\circ}x{\circ}v$ as $\phi(v,u)(x)$ for $u:A\to A'$, 
$x{\in}\phi(B,A)$ and $v : B'\to B$.} Thus, we obtain
$$\phi^* = (\phi^G\phi_F)^* \simeq (\phi_F)^*(\phi^G)_* \simeq F^*\prod_G$$
as $(\phi_F)^* \simeq F^*$ and $(\phi^G)^* \simeq \prod_G$ and change of base 
along distributors is functorial in a contravariant way (i.e.\ 
$(\phi_2\phi_1)^* \simeq \phi_1^*\phi_2^*$). Thus $\phi^*$ has a left adjoint
$\coprod_\phi = G^*\coprod_F$.

One might ask whether for all distributors $\phi:\A\nrightarrow\B$ there also
exists a right adjoint $\prod_\phi$ to $\phi^*$. Of course, if $\prod_\phi$
exists then by the fibered Yoneda Lemma it must look as follows
$$(\prod_\phi P)(I) \simeq \Fib(\B)(\underline{I},\prod_\phi P) \simeq
  \Fib(\A)(\phi^*\underline{I},P)$$
from which it is obvious that it restricts to an adjunction between 
$\widehat{\A}$ and $\widehat{\B}$ as $\Fib(\A)(\phi^*\underline{I},P)$ is
discrete whenever $P$ is discrete. Thus, a necessary condition for the 
existence of $\prod_\phi$ is that the functor $\phi^* : \widehat{\B} \to
\widehat{\A}$ is cocontinuous. As $\phi^* : \widehat{\B} \to \widehat{\A}$ is
right adjoint to $\widehat{\phi} : \widehat{\A} \to \widehat{\B}$, the left
Kan extension of $\phi : \A\to\widehat{\B}$ along $\Yon_\B$, the distributor 
$\phi$ has a right adjoint if and only if $\phi^*$ is cocontinuous.
Thus, a necessary condition for the existence of $\prod_\phi$ is the
existence of a right adjoint distributor to $\phi$. This, however, is known to
be equivalent (see e.g.\ \cite{Bor} vol.1) to the requirement that $\phi(A)$ 
is a retract of a representable presheaf for all objects $A$ in $\A$. In case 
$\B$ is Cauchy complete, i.e.\ all idempotents in $\B$ split, this means that 
up to isomorphism $\phi$ is of the form $\phi_F$ for some functor $F:\A\to\B$ 
and then $\prod_F$ provides a right adjoint to $\phi^*$. As $\Fib(\B)$ is 
equivalent to $\Fib({{\mathrm{IdSp}}}(\B))$, where ${{\mathrm{IdSp}}}(\B)$ is 
obtained from $\B$ by splitting all idempotents, one can show that $\phi^*$ 
has a right adjoint $\prod_\phi$ whenever $\phi$ has a right adjoint distributor. Thus, for a distributor $\phi$ the change of base functor $\phi^*$ has a 
right adjoint $\prod_\phi$ if and only if $\phi$ has a right adjoint 
distributor, i.e.\ if and only if $\phi$ is essentially a functor. An example 
of a distributor $\phi$ where $\phi^*$ does not have a right adjoint can be 
obtained as follows. Let $\A$ be a terminal category and $\B$ a small category
whose splitting of idempotents does not have a terminal object. 
Let $\phi : \A\to\widehat{\B}$
select a terminal presheaf from $\widehat{\B}$. Then $\phi^*$ amounts to the 
global sections functor on $\widehat{\B}$ which, however, does not have a 
right adjoint as otherwise ${{\mathrm{IdSp}}}(\B)$ would have a 
terminal object.

\newpage
\section{Finite Limit Preserving Functors as Fibrations}

If $F : \B \to \C$ is a finite limit preserving functor between categories
with finite limits then the fibration
$$P_F \equiv F^*P_{\C} = \partial_1 : \C{\downarrow}F \to \B$$
satisfies the following conditions which later will turn out as sufficient
for reconstructing the functor $F : \B \to \C$ up to equivalence.

\begin{enumerate}
\item[(1)] $\C{\downarrow}F$ has finite limits and 
           $P_F$ preserves them. 
\item[(2)] $P_F$ has internal sums which are \emph{stable} in the sense 
that cocartesian arrows are stable under pullbacks along arbitrary morphisms.
\item[(3)] The internal sums of $P_F$ are \emph{disjoint} in the
sense that for every cocartesian arrow $\varphi : X \to Y$ the fiberwise 
diagonal $\delta_{\varphi}$ is cocartesian, too.
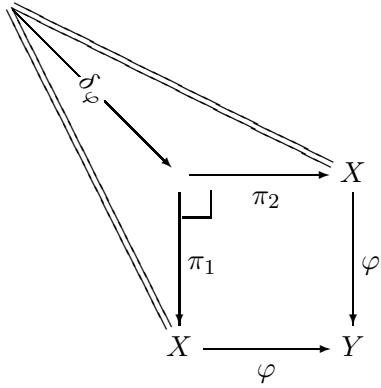
\begin{diagram}
  &            & & & \\
  & \rdTo~{\delta_{\varphi}} \rdEqual(2,4) \rdEqual(4,2) & & & \\
  &            &   \SEpbk & \rTo_{\pi_2} &   X\\
  &            &  \dTo_{\pi_1}     &      &  \dTo_{\varphi} \\
  &            &  X        & \rTo_{\varphi} &   Y \\
\end{diagram} 
\end{enumerate}

We refrain from giving the detailed verifications of properties (1)--(3).
Instead we recall some few basic facts needed intrinsically when verifying
the claims.

Notice that a morphism
\begin{diagram}[small]
A & \rTo^{f} & B \\
\dTo^{a} & & \dTo_{b}  \\
FI & \rTo_{Fu} & FJ
\end{diagram}
in $\C{\downarrow}F$ over $u : I \to J$ is cocartesian iff $f$ is an isomorphism.

Notice that pullbacks in $\C{\downarrow}F$ are given by
\begin{diagram}[small]
D  &  &  \rTo^{\pi_2}  &  &  C  &  &  \\
   & \rdTo_{\pi_1} & & & \vLine_{c} & \rdTo^{g} & \\ 
\dTo^{d} & & B & \rTo^{f} & \HonV & & A \\
& & \dTo^{b} & & \dTo & & \\
FL & \hLine & \VonH & \rTo^{Fq} & FK & & \dTo_{a} \\
  & \rdTo_{Fp} & & & & \rdTo^{Fv} & \\
  &            & FJ   &   &  \rTo_{Fu} & & FI \\
\end{diagram}
where the top square is a pullback and the bottom square is the image of a 
pullback under $F$. From this is is clear that $\partial_0,\partial_1 :
\C{\downarrow}F \to \B$ both preserve pullbacks. Condition (3) follows from
preservation of pullbacks by $\partial_0$ and the above characterisation of 
cocartesian arrows.

Now based on work by J.-L.~Moens from his Th\'ese \cite{Moe} we will
characterise those fibrations over a category $\B$ with finite limits 
which up to equivalence are of the form $P_F$ for some finite limit 
preserving functor $F$ from $\B$ to a category $\C$ with finite limits. 
It will turn out that the three conditions above are necessary and sufficient.
In particular, we will show that the functor $F$ can be recovered from 
$P_F$ in the following way. First observe that 
$\partial_0 : \C{\downarrow}F \to \C$ is isomorphic to the functor
$\DDelta : \C{\downarrow}F \to \C{\downarrow}F1 \cong \C$ given by
\begin{diagram}[small]
X & \rTo^{\varphi_X}_\cocart & \DDelta(X) \\
\dTo^{f} & & \dTo_{\DDelta(f)} \\
Y & \rTo_{\varphi_Y}^\cocart & \DDelta(Y)
\end{diagram}
with $\DDelta(f)$ vertical over the terminal object in $\B$.
Now the functor $F$ itself can be obtained up to isomorphism as 
$\Delta = \DDelta \circ 1$ (where $1$ is the cartesian functor choosing
fiberwise terminal objects).

Notice that this construction makes sense for arbitrary fibrations $P$ over
$\B$ with internal sums. Our goal now is to show that every fibration $P$ of 
categories with finite limits over $\B$ with stable disjoint (internal) sums 
is equivalent to $P_\Delta$ where $\Delta$ is defined as above
and preserves finite limits.

But for this purpose we need a sequence of auxiliary lemmas.

\begin{Lem}\label{auxMoensLem}
Let $\B$ be category with finite limits and $P : \X \to \B$ be a fibration
of categories with finite limits and stable disjoint internal sums. Then in
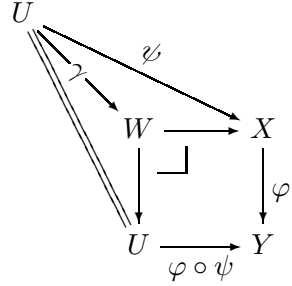
\begin{diagram}[small]
U &            & & & \\
  & \rdTo~{\gamma} \rdEqual(2,4) \rdTo(4,2)^{\psi} & & & \\
  &            &  W \SEpbk & \rTo &   X\\
  &            &  \dTo     &      &  \dTo_{\varphi} \\
  &            &  U        & \rTo_{\varphi \circ \psi} &   Y  
\end{diagram}
the arrow $\gamma$ is cocartesian whenever $\varphi$ is cocartesian.
\end{Lem}
\begin{proof}
Consider the diagram
\begin{diagram}[small]
U \SEpbk  &  \rTo^{\psi} & X & & \\
\dTo^{\gamma}   & & \dTo^{\delta_\varphi} &  \rdEqual& \\
W  \SEpbk &  \rTo  &  Z  \SEpbk & \rTo^{\pi_2} &   X\\
\dTo^{\widetilde{\varphi}} &  &  \dTo^{\pi_1} &    & \dTo_{\varphi} \\
U  &  \rTo_{\psi} &  X        & \rTo_{\varphi} &   Y  
\end{diagram}
with $\pi_i \delta_{\varphi} = \id_X$ and 
$\widetilde{\varphi} \circ \gamma = \id_U$. Thus, by stability of sums
$\gamma$ is cocartesian as it appears as pullback of the cocartesian 
arrow $\delta_{\varphi}$.
\end{proof}

\begin{Lem}\label{MoensLem}
Let $\B$ be category with finite limits and $P : \X \to \B$ be a fibration
of categories with finite limits and stable internal sums, i.e.\ $P$ is also
a cofibration whose cocartesian arrows are stable under pullbacks along 
arbitrary maps in $\X$.

Then the following conditions are equivalent
\begin{enumerate}
\item[\rm (1)] The internal sums of $P$ are disjoint.
\item[\rm (2)] If $\varphi$ and $\varphi \circ \psi$ are cocartesian 
           then $\psi$ is cocartesian, too.
\item[\rm (3)] 
If $\alpha$ is vertical and both $\varphi$ and $\varphi \circ \alpha$ are 
cocartesian then $\alpha$ is an isomorphism.
\item[\rm (4)] A commuting diagram
\begin{diagram}[small]
X & \rTo^\varphi_\cocart & U \\
\dTo^{\alpha} & & \dTo_{\beta} \\
Y & \rTo_\psi^\cocart & V
\end{diagram}
is a pullback in $\X$ whenever $\varphi$, $\psi$ are cocartesian and $\alpha$,
$\beta$ are vertical.
\end{enumerate} 
The equivalence of conditions (2)--(4) holds already under the weaker
assumption that cocartesian arrows are stable under pullbacks along 
vertical arrows.
\end{Lem}
\begin{proof} 
(1) $\Rightarrow$ (2) : Suppose that both $\varphi$ and $\varphi \circ \psi$ 
are cocartesian. Then for the diagram
\begin{diagram}
Z &            & & & \\
  & \rdTo~{\gamma} \rdEqual(2,4) \rdTo(4,2)^{\psi} & & & \\
  &            &  \cdot \SEpbk & \rTo_{\theta} &   Y \\
  &            &  \dTo     &      &  \dTo_{\varphi} \\
  &            &  Z        & \rTo_{\varphi \circ \psi} &   X  
\end{diagram}
we have that $\psi = \theta \circ \gamma$ is cocartesian as $\gamma$ is 
cocartesian by Lemma~\ref{auxMoensLem} and $\theta$ is cocartesian by 
stability of sums as it appears as pullback of the 
cocartesian arrow $\varphi \circ \psi$.\\
(2) $\Rightarrow$ (1) : As $\pi_i$ and $\pi_i \circ \delta_{\varphi} = \id$
are both cocartesian it follows from assumption (2) that $\delta_{\varphi}$
is cocartesian, too.\\
(2) $\Leftrightarrow$ (3) : Obviously, (3) is an instance of (2). For 
the reverse direction assume (3) and suppose that both $\varphi$ and 
$\varphi \circ \psi$ are cocartesian. Let $\psi = \alpha \circ \theta$
with $\theta$ cocartesian and $\alpha$ vertical. Then $\varphi\circ\alpha$ 
is cocartesian from which it follows by (3) that $\alpha$ is a vertical 
isomorphism and thus $\psi = \alpha \circ \theta$ is cocartesian.\\
(3) $\Leftrightarrow$ (4) : Obviously, (4) entails (3) instantiating $\beta$
by identity as isos are stable under pullbacks. For the reverse direction 
consider the diagram

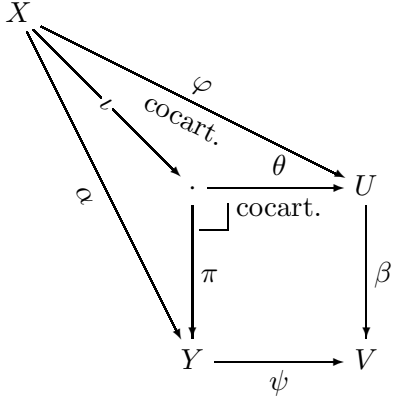
\begin{diagram}
X  &            & & & \\
&\rdTo~{\iota} \rdTo(2,4)_{\alpha} \rdTo(4,2)^{\varphi}_{\mathrm{cocart.}}&&&\\
  &            &  \cdot \SEpbk & \rTo^{\theta}_{\mathrm{cocart.}} &   U\\
  &            &  \dTo_{\pi}     &      &  \dTo_{\beta} \\
  &            &  Y        & \rTo_{\psi} &   V \\
\end{diagram}  
with $\pi$ vertical. The morphism $\theta$ is cocartesian since it arises
as pullback of the cocartesian arrow $\psi$ along the vertical arrow $\beta$.
Moreover, the map $\iota$ is vertical since $\alpha$ and $\pi$
are vertical. Thus, by assumption (3) it follows that $\iota$ is an 
isomorphism. Thus, the outer square is a pullback since it is isomorphic 
to a pullback square via $\iota$.
\end{proof}

\bigskip\noindent
{\bf Remark.} Alternatively, we could have proved Lemma~\ref{MoensLem}  
by showing (1) $\Rightarrow$ (4) $\Rightarrow$ (3) $\Rightarrow$ (2) 
$\Rightarrow$ (1) where the last three implications have already been 
established. The implication (1) $\Rightarrow$ (4) was proved in \cite{Moe} 
as follows. Consider the diagram
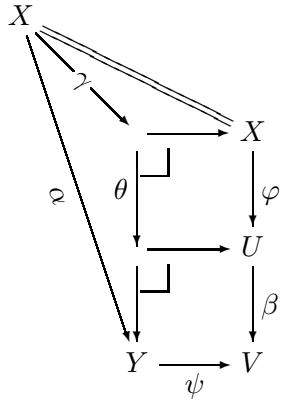
\begin{diagram}[small]
X &&&&  \\
  &  \rdTo(2,6)_{\alpha} \rdTo~{\gamma} \rdEqual(4,2)& & &  \\
  & &    \SEpbk   &  \rTo            &  X            \\
  & & \dTo^{\theta} &                  &  \dTo_{\varphi}  \\
  & &    \SEpbk   &  \rTo          &  U        \\
  & & \dTo         &                &  \dTo_{\beta} \\
  & & Y            & \rTo_{\psi} &  V        \\
\end{diagram} 
where $\theta$ is cocartesian by stability of sums since $\theta$ appears 
as pullback of the cocartesian arrow $\varphi$. From Lemma~\ref{auxMoensLem}
it follows that $\gamma$ is cocartesian as by assumption 
$\beta \circ \varphi = \psi \circ \alpha$ and $\psi$ is cocartesian. 
Thus, the map $\theta \circ \gamma$ is cocartesian over an isomorphism and, 
therefore, an isomorphism itself. \MYenddef

Notice that condition (3) of Lemma~\ref{MoensLem} is equivalent 
to the requirement that for every map $u : I \to J$ in $\B$ the coproduct 
functor $\coprod_u : \X_I \to \X_J$ reflects isomorphisms.

\medskip
As a consequence of Lemma~\ref{MoensLem} we get the following 
characterisation of disjoint stable sums in terms of 
extensivity.

\begin{Lem}\label{extensiveMoensLem}
Let $\B$ be category with finite limits and $P : \X \to \B$ be a fibration
of categories with finite limits and internal sums. Then the following 
conditions are equivalent
\begin{enumerate}
\item[\rm (1)] The internal sums of $P$ are stable and disjoint.
\item[\rm (2)] The internal sums of $P$ are extensive\footnote{Recall 
               that a category with pullbacks and sums is
               called \emph{extensive} iff for every family of squares
               \begin{diagram}[small]
               B_i & \rTo^{f_i} & B \\
               \dTo^{a_i} &  & \dTo_b \\
               A_i & \rTo_{\inj_i} & \coprod_{i{\in}I} A_i \\
               \end{diagram}
               all squares are pullbacks iff $f_i : B_i \to B$ is a 
               coproduct cone.},
i.e.\ for all commuting squares
\begin{diagram}[small]
X & \rTo^\varphi & U \\
\dTo^\alpha & & \dTo_\beta \\
Y & \rTo_\psi^\cocart & V \\
\end{diagram}
where $\psi$ is cocartesian and $\alpha$ and $\beta$ are vertical it holds 
that $\varphi$ is cocartesian iff the square is a pullback.
\item[\rm (3)] The internal sums of $P$ are 
               extensive in the sense of Lawvere\footnote{Recall that
               a category $\C$ is extensive in the sense of Lawvere
               iff for all sets $I$ the categories 
               $\C^I$ and $\C/\coprod_I 1$ are canonically isomorphic.}, 
               i.e.\ for all commuting squares
\begin{diagram}[small]
X & \rTo^\varphi & U \\
\dTo^\alpha & & \dTo_\beta \\
1_I & \rTo_{\varphi_I}^\cocart & \coprod_I 1_I \\
\end{diagram}
where $\varphi_I$ is cocartesian over $!_I : I \to 1$ in $\B$, 
$1_I$ is terminal in its fiber and $\alpha$ and $\beta$ are vertical 
it holds that $\varphi$ is cocartesian iff the square is a pullback.
\end{enumerate}
The equivalence of (2) and (3) holds already under the weaker assumption
that cocartesian arrows are stable under pullbacks along vertical arrows.
\end{Lem}
\begin{proof}
(1)$\Leftrightarrow$(2) :
The implication from right to left in (2) is just stability of internal sums.
The implication from left to right in (2) is just condition (4) of 
Lemma~\ref{MoensLem} which under assumption of stability of sums  
by Lemma~\ref{MoensLem} is equivalent to the disjointness of sums.

Obviously, condition (3) is an instance of condition (2). Thus it remains
to show that (3) entails (2). 

Consider the diagram
\begin{diagram}[small]
V & \rTo^{\psi_0} & V_0 \\
\dTo^\beta & ^{(*)} & \dTo_{\beta_0} \\
U \SEpbk & \rTo_{\varphi_0}^\cocart & U_0 \\
\dTo^{\gamma} & & \dTo_{\gamma_0} \\
1_I & \rTo_{\varphi_I}^\cocart & \coprod_I 1_I \\
\end{diagram}
where $\beta$, $\beta_0$, $\gamma$ and $\gamma_0$ are vertical, $\varphi_I$
is cocartesian over $!_I$ and $1_I$ is terminal in its fiber. The lower
square is a pullback due to assumption (3). If the upper square is a pullback 
then the outer rectangle is a pullback and thus $\psi_0$ is cocartesian by (3).
If $\psi_0$ is cocartesian then the outer rectangle is a pullback by (3) and 
thus the upper square is a pullback, too. 

Thus we have shown that 
\begin{enumerate}
\item[$(\dagger)$] a diagram of the form $(*)$ is a pullback 
                   iff $\psi_0$ is cocartesian.
\end{enumerate}
Now consider a commuting diagram
\begin{diagram}[small]
Y & \rTo^\psi & V \\
\dTo^\alpha & ^{(+)} & \dTo_\beta \\
X & \rTo_\varphi^\cocart & U
\end{diagram}
with $\alpha$ and $\beta$ vertical. 

We have to show that $\psi$ is cocartesian iff $(+)$ is a pullback.

Suppose $\psi$ is cocartesian. Then by $(\dagger)$ the outer rectangle and 
the right square in
\begin{diagram}[small]
Y & \rTo^\psi_\cocart & V & \rTo^{\psi_0}_\cocart & V_0 \\
\dTo^\alpha & & \dTo_\beta & & \dTo_{\beta_0} \\
X & \rTo_\varphi^\cocart & U & \rTo_{\varphi_0}^\cocart & U_0
\end{diagram}
are pullbacks from which it follows that the left square, i.e.\ $(+)$,
is a pullback, too, as desired. 

Suppose the square $(+)$ is a pullback. Then we have
\begin{diagram}[small]
Y \SEpbk & \rTo^\psi & V & \rTo^{\psi_0}_\cocart & V_0 \\
\dTo^\alpha & & \dTo_\beta & & \dTo_{\beta_0} \\
X & \rTo_\varphi^\cocart & U & \rTo_{\varphi_0}^\cocart & U_0
\end{diagram}
As by $(\dagger)$ the right square is a pullback it follows that 
the outer rectangle is a pullback, too, from which it follows by $(\dagger)$ 
that $\psi_0 \psi$ is cocartesian. Now consider the diagram
\begin{diagram}[small]
Y & \rTo_\theta^\cocart & Z & \rTo_{\theta_0}^\cocart & Z_0 \\
  & \rdTo_\psi  &  \dTo_\iota  &   & \dTo_{\iota_0} \\
  & & V & \rTo_{\psi_0} & V_0
\end{diagram}
where $\iota$ and $\iota_0$ are vertical. Then $\iota_0$ is an isomorphism
because $\theta_0 \theta$ and $\psi_0 \psi$ start from the same source and
are both cocartesian over the same arrow in $\B$. By $(\dagger)$ the right
square is a pullback from which it follows that $\iota$ is an isomorphism
(as isomorphisms are pullback stable) and thus $\psi$ is cocartesian
as desired.
\end{proof}

\medskip
Notice that condition (3) of Lemma~\ref{extensiveMoensLem} is equivalent to 
the requirements that for all $I\in\B$ the coproduct functor 
$\coprod_I : \X_I \to \X_1$ reflects isomorphisms and $\beta^*\varphi_I$ 
is cocartesian for all vertical maps $\beta : U \to \coprod_I 1_I$.

An immediate consequence of Lemma~\ref{extensiveMoensLem} is the following

\begin{Cor}\label{corMoens1}
Let $\B$ have finite limits and $P : \X \to \B$ be a fibration of categories 
with finite limits and stable disjoint internal sums. Then for every 
$u : I \to J$ in $\B$ and $X \in P(I)$ the functor 
$\coprod_u / X  : \X_I / X \to \X_J / \coprod_u X$ is an equivalence.
In particular, we get that $\X_I \cong \X_I/1_I$ is equivalent to 
$\X_J/ \coprod_u 1_I$ via $\coprod_u / 1_I$ and that $\X_I \cong \X_I/1_I$
is equivalent to $\X_1 / \Delta(I)$ via $\coprod_I / 1_I$
where $\Delta(I) = \coprod_I 1_I$.
\end{Cor}

\begin{Cor}\label{corMoens2}
Let $\B$ have finite limits and $P : \X \to \B$ be fibration of categories 
with finite limits and stable disjoint internal sums. Then for every 
$u : I \to J$ in $\B$ the functor $\coprod_u : \X_I \to \X_J$ 
preserves pullbacks.
\end{Cor}
\begin{proof}
Notice that $\coprod_u = \Sigma_{\X_J/\coprod_u 1_I} \circ \coprod_u/1_I$ 
where we identify $\X_I$ and $\X_I/1_I$ via their canonical isomorphism.
The functor $\coprod_u/1_I$ preserves pullbacks as it is an equivalence by
Corollary~\ref{corMoens1}. The functor 
$\Sigma_{\X_J/\coprod_u1_I} = \partial_0$ is known to preserve pullbacks 
anyway. Thus, the functor $\coprod_u$ preserves pullbacks 
as it arises as the composite of pullback preserving functors.
\end{proof}

\begin{Lem}\label{Moenslem1}
Let $P : \X \to \B$ be a fibration of categories with finite limits and
stable disjoint internal sums. Then the mediating arrow $\theta$ is 
cocartesian for any diagram in $\X$
\begin{diagram}[small]
U \SEpbk  &   & \rTo^{\phi_2}    & &                X_2  \\
      &  \rdDashto_{\theta}     &      &          & \dTo_{\varphi_2}\\
 \dTo^{\phi_1}   &      &  V  \SEpbk  & \rTo^{\beta_2}   &       Y_2    \\
    &                   & \dTo^{\beta_1}  &      &  \dTo_{\alpha_2}  \\
X_1 & \rTo_{\varphi_1}  & Y_1   & \rTo_{\alpha_1} &        Y          \\
\end{diagram}
whenever the $\varphi_i$ are cocartesian, the $\alpha_i$, $\beta_i$ 
are vertical and the outer and the inner square are pullbacks.
\end{Lem}
\begin{proof}
Consider the diagram
\begin{diagram}[small]
U \SEpbk  &   \rTo^{\psi_2}  & U_2 \SEpbk &   \rTo^{\gamma_2}     &    X_2  \\
 \dTo^{\psi_1}      &      & \dTo^{\theta_2}      &      & \dTo_{\varphi_2} \\
 U_1 \SEpbk  & \rTo^{\theta_1}  &  V  \SEpbk  & \rTo^{\beta_2}   &   Y_2    \\
\dTo^{\gamma_1}  &          & \dTo^{\beta_1}  &      &  \dTo_{\alpha_2}  \\
X_1 & \rTo_{\varphi_1}  & Y_1   & \rTo_{\alpha_1} &        Y          \\
\end{diagram}
where by stability of sums the $\psi_i$ and $\theta_i$ are cocartesian as they
arise as pullbacks of $\varphi_1$ or $\varphi_2$, respectively. 
As the big outer square is a pullback we may assume that 
$\phi_i = \gamma_i \circ \psi_i$ (by appropriate choice of the $\psi_i$). 

Thus, $\theta = \theta_1 \circ \psi_1 = \theta_2 \circ \psi_2$
cocartesian as it arises as composition of cocartesian arrows.
\end{proof}

\begin{Lem}\label{Moenslem2}
Let $\B$ have finite limits and $P : \X \to \B$ be fibration of categories 
with finite limits and stable disjoint internal sums. 

Then the functor $\DDelta : \X \to \X_1$ given by
\begin{diagram}
X & \rTo^{\varphi_X}_\cocart & \DDelta(X) \\
\dTo^{f} & & \dDashto_{\DDelta(f)} \\
Y & \rTo_{\varphi_Y}^\cocart & \DDelta(Y)
\end{diagram}
with $\DDelta(f)$ over $1$ preserves finite limits. 
\end{Lem}
\begin{proof}
Clearly, the functor $\DDelta$ preserves the terminal object. It remains 
to show that it preserves also pullbacks. Let
\begin{diagram}[small]
U \SEpbk & \rTo^{g_2} & X_2\\
\dTo^{g_1} & & \dTo_{f_2} \\
X_1 & \rTo_{f_1} & Y\\
\end{diagram}
be a pullback in $\X$. 
Then by Lemma~\ref{Moenslem1} the arrow $\theta$ is cocartesian in 

\begin{diagram}[small]
U \SEpbk  &   & \rTo^{g_2}    & &                X_2  \\
      &  \rdDashto_{\theta}     &      &          & \dTo_{\varphi_2}\\
 \dTo^{g_1}   &      &  V  \SEpbk  & \rTo^{\beta_2}   &       Y_2    \\
    &                   & \dTo^{\beta_1}  &      &  \dTo_{\alpha_2}  \\
X_1 & \rTo_{\varphi_1}  & Y_1   & \rTo_{\alpha_1} &        Y          \\
\end{diagram} 
where $f_i = \alpha_i \circ \varphi_i$ with $\alpha_i$ vertical and 
$\varphi_i$ cocartesian. From this we get that the square
\begin{diagram}[small]
\DDelta(U)  & \rTo^{\DDelta(g_2)} & \DDelta(X_2)\\
\dTo^{\DDelta(g_1)} & & \dTo_{\DDelta(f_2)} \\
\DDelta(X_1) & \rTo_{\DDelta(f_1)} & \DDelta(Y)\\
\end{diagram}
is a pullback, too, as it is obtained by applying the pullback preserving
functor $\coprod_{P(Y)}$ to 
\begin{diagram}[small]
V  \SEpbk  & \rTo^{\beta_2}   &       Y_2    \\
\dTo^{\beta_1}  &      &  \dTo_{\alpha_2}  \\
Y_1   & \rTo_{\alpha_1} &        Y          \\
\end{diagram}
which is a pullback in the fiber over $P(Y)$.
\end{proof}

\bigskip
Now we are ready to prove Moens's Theorem.

\begin{Thm}\label{MoensThm}
Let $\B$ have finite limits and $P : \X \to \B$ be fibration of categories 
with finite limits and stable disjoint internal sums. Then $P$ is equivalent
to $P_\Delta$ where $\Delta$ is the finite limit preserving
functor $\DDelta \circ 1$.

More explicitely, the fibered equivalence $E : P \to P_\Delta$ is
given by sending $f : X \to Y$ in $\X$ over $u : I \to J$ to
\begin{diagram}[small]
\DDelta(X) & \rTo^{\DDelta(f)} & \DDelta(Y) \\
\dTo^{\DDelta(\alpha)} & E(f) & \dTo_{\DDelta(\beta)} \\
\DDelta(1_I) & \rTo_{\Delta(u)} & \DDelta(1_J) \\
\end{diagram} 
where $\alpha$ and $\beta$ are terminal projections in their fibers. 
\end{Thm}
\begin{proof}
As $\Delta(u) = \DDelta(1_u)$ the map $E(f)$ arises as the  
image under $\DDelta$ of the square
\begin{diagram}[small]
X  & \rTo^{f} & Y \\
\dTo^{\alpha} &  & \dTo_{\beta} \\
1_I & \rTo_{1_u} & 1_J \\
\end{diagram}
which is a pullback if $f$ is cartesian. As by Lemma~\ref{Moenslem2}
the functor $\DDelta$ preserves pullbacks it follows that $E$ is cartesian.
Thus, the fibered functor $E$ is a fibered equivalence as by 
Corollary~\ref{corMoens1} all fibers of $E$ are (ordinary) equivalences. 
\end{proof}

\medskip
Thus, for categories $\B$ with finite limits we have established a 
1--1--correspondence up to equivalence between fibrations of the form 
$P_F = \partial_1 : \C{\downarrow}F \to \B$ for some finite limit preserving 
$F : \B \to \C$ where $\C$ has finite limits and fibrations over $\B$ of
categories with finite limits and extensive internal sums.\footnote{
This may explain why Lawvere's notion of extensive sums is so important.
Notice, however, that Lawvere's original definition only applied to 
ordinary categories $\C$ with small coproducts in the ordinary sense.
That our notion of Lawvere extensivity is slightly more general can be seen
from the discussion at the end of section 17 where we give an example (due to
Peter Johnstone) of a fibration over $\Set$ of categories with finite limits
and Lawvere extensive small sums which, however, is not of the form
$\Fam(\C)$ for some ordinary category $\C$.}

\medskip
With little effort we get the following generalization of Moens's Theorem.

\begin{Thm}\label{GenMoensThm}
Let $\B$ be a category with finite limits. If $\C$ is a category with finite
limits and $F : \B\to\C$ preserves terminal objects then $P_F$ is a fibration
of finite limit categories and a cofibration where cocartesian arrows are stable
under pullbacks along vertical arrows and one of the following equivalent 
conditions holds
\begin{enumerate}
\item[\emph{(1)}] if $\varphi$ and $\varphi\circ\psi$ are cocartesian 
                  then $\psi$ is cocartesian
\item[\emph{(2)}] if $\alpha$ is vertical and both $\varphi$ and 
                  $\varphi\circ\alpha$ are cocartesian 
                  then $\alpha$ is an isomorphism
\item[\emph{(3)}] a commuting square
\begin{diagram}[small]
X & \rTo^\varphi_\cocart & U \\
\dTo^{\alpha} & & \dTo_{\beta} \\
Y & \rTo_\psi^\cocart & V
\end{diagram}
is a pullback whenever $\psi$ and $\varphi$ are cocartesian and
$\alpha$ and $\beta$ are vertical.
\end{enumerate}
Bifibrations $P$ satisfying these properties are equivalent to $P_{\Delta_P}$
where $\Delta_P$ is the functor $\Delta : \B \to \X_1$ sending $I$ to 
$\Delta(I) = \coprod_I 1_I$ and $u : J \to I$ to the unique vertical arrow
$\Delta(u)$ rendering the diagram
\begin{diagram}[small]
1_J & \rTo^{\varphi_J} & \Delta(J) \\
\dTo^{1_u} & & \dTo_{\Delta(u)} \\
1_I & \rTo_{\varphi_I}^\cocart & \Delta(I)
\end{diagram}
commutive.
This correspondence is an equivalence since $\Delta_{P_F}$ is isomorphic
to $F$.
\end{Thm}
\begin{proof}
We have already seen that these conditions are necessary and the equivalence of 
(1)--(3) follows from Lemma~\ref{MoensLem}.

For the reverse direction first observe that the assumptions on $P$ imply 
that every commuting square
\begin{diagram}[small]
X & \rTo^\varphi & U \\
\dTo^{\alpha} & & \dTo_{\beta} \\
1_I & \rTo_{\varphi_I}^\cocart & \Delta(I)
\end{diagram}
with $\alpha$ and $\beta$ vertical is a pullback iff $\varphi$ is cocartesian. 
Thus pullback along the cocartesian arrow 
$\varphi_I : 1_I \to \coprod_I 1_I = \Delta(I)$ induces an equivalence 
between $X_I$ and $\X_1 / \coprod_I 1_I$. This extends to 
an equivalence between $P$ and $P_\Delta$ since
\begin{diagram}[small]
\X_J & \lTo^{\varphi_J^*}_\simeq & \X_1 / \Delta(J) \\
\uTo^{u^*} & & \uTo_{\Delta(u)^*} \\
\X_I & \lTo_{\varphi_I^*}^\simeq & \X_1 / \Delta(I) \\
\end{diagram}
commutes up to isomorphism for all $u : J \to I$ in $\B$. 
\end{proof}

\medskip
As apparent from the proof fibrations $P : \X \to \B$ over a finite limit 
category $\B$ are equivalent to $P_F$ for some terminal object preserving
functor $F$ to a finite limit category if and only if $P$ is a bifibration 
such that $\X$ has and $P$ preserves finite limits and every commuting square
of the form
\begin{diagram}[small]
X & \rTo^\varphi & U \\
\dTo^{\alpha} & & \dTo_{\beta} \\
1_I & \rTo_{\varphi_I}^\cocart & \Delta(I)
\end{diagram}
with $\alpha$ and $\beta$ vertical is a pullback iff $\varphi$ is cocartesian.

\bigskip
Finally we discuss how the fact that finite limit preserving functors are 
closed under composition is reflected on the level of their fibrations 
associated via glueing. Suppose that $F : \B\to\C$ and $G:\C\to\D$ are finite 
limit preserving functors between categories with finite limits. 
Then $P_{GF} \cong 1^*F^*\FFam(P_G)$ as indicated in 
\begin{diagram}
\D{\downarrow}G{\circ}F \SEpbk & \rInto & \cdot \SEpbk & \rTo & \cdot \SEpbk & 
\rTo & \D{\downarrow}G \SEpbk & \rTo & \D^\two\\
\dTo_{P_{GF}} & & \dTo_{F^*\FFam(P_G)} && \dTo_{\FFam(P_G)} & 
& \dTo_{P_G} & & \dTo_{\partial_1} \\
\B & \rInto_1 & \C{\downarrow}F\SEpbk &\rTo_{\partial_1^*F} & \C^\two & \rTo_{\partial_0}&\C & \rTo_G & \D\\
& \rdEqual & \dTo_{P_F}  & & \dTo_{\partial_1} & & \\
& & \B & \rTo_F & \C & & 
\end{diagram}
because $\partial_0 \circ \partial_1^*F \circ 1 = \partial_0 \circ 1 = F$. 
The fibration $F^*\FFam(P_G)$ is $P_G$ shifted from $\C$ to $\C{\downarrow}F$
via change of base along 
$\DDelta = \partial_0 = \partial_0 \circ \partial_1^*F : \C{\downarrow}F \to \C$. 
The fibration $P_{GF}$ appears as a(n in general proper) subfibration 
of the composite fibration $P_F \circ F^*\FFam(P_G)$.

A fibration $Q : \Y \to \C{\downarrow}F$ is isomorphic to one of the form 
$F^*\FFam(P_G)$ iff $Q$ is a fibration of categories with finite limits and 
stable disjoint internal sums such that $\Delta : \C{\downarrow}F \to \Y_1$
is isomorphic to a functor of the form $G \circ \partial_0$, i.e.\ iff $\Delta$ 
inverts cocartesian arrows of $\C{\downarrow}F$. This latter condition is 
equivalent to the requirement that $1_\varphi$ is cocartesian w.r.t.\ $Q$ 
whenever $\varphi$ is cocartesian w.r.t.\ $P_F$.\footnote{As $\Delta(\varphi)$ 
is an isomorphism iff $1_\varphi$ is cocartesian. 
This can be seen from the diagram 
\begin{diagram}[small]
1_X & \rTo^{\varphi_X}_\cocart  & \Delta(X) \\
\dTo^{1_\varphi} & & \dTo_{\Delta(\varphi)} \\
1_Y & \rTo_{\varphi_Y}^\cocart  & \Delta(Y) \\
\end{diagram}
where $\varphi_X$ and $\varphi_Y$ are cocartesian over the terminal 
projections of $X$ and $Y$, respectively, and $\Delta(\varphi)$ is vertical.
If $1_\varphi$ is cocartesian then $\Delta(\varphi)$ is an isomorphism as it
is vertical and cocartesian. On the other hand if $\Delta(\varphi)$ is an 
isomorphism then $\Delta(\varphi) \circ \varphi_X$ is cocartesian, too, and 
thus by Lemma~\ref{MoensLem}(2) it follows that $1_\varphi$ is cocartesian.} 
This fails e.g.\ for $Q \equiv P_{\Id_{\C{\downarrow}F}}$ if not all 
cocartesian arrows of $\C{\downarrow}F$ are isomorphisms, i.e.\ $\B$ is 
not equivalent to the trivial category $\mathbf{1}$.

\newpage

\section{Geometric Morphisms as Fibrations}

Geometric morphism are adjunctions $F \dashv U : \C \to \B$ where $F$ 
preserves finite limits. Though introduced originally for toposes the notion 
of geometric morphism makes sense already if $\B$ and $\C$ have finite limits.

First we will characterise for functors $F$ between categories with finite 
limits the property that $F$ has a right adjoint in terms of a purely 
fibrational property of its associated fibration $P_F = F^*P_\C$, 
namely that of having \emph{small global sections}. 

First we observe that the requirement $P \dashv 1 \dashv G$ is equivalent to 
$P$ having small global sections since $1 \dashv G$ says that 
for every $X \in P(I)$ there is an $\varepsilon_X : 1_{GX} \to X$ 
such that for every $\sigma : 1_J \to X$ over
$u : J \to I$ there is a unique $v : J \to GX$ with
\begin{diagram}
     &                   &  1_I  \\
     &  \ruTo^{1_u}      &  \uTo_{1_{P(\varepsilon_X)}} \\
1_J  &  \rDashto^{\qquad 1_v}   &  1_{GX} \\
     &  \rdTo_{\sigma}   &  \dTo_{\varepsilon_X} \\
     &                   &  X \\
\end{diagram}
i.e.\ that $\Hom_I(1_I,X)$ is representable.
If $P$ is a fibration of cartesian closed categories (or even a fibered topos)
then $P$ has small global sections iff $P$ is locally small. 
 
\begin{Thm}\label{gmthm1}
Let $F : \B \to \C$ be a functor between categories with finite limits. 
Then $F$ has a right adjoint $U$ iff the fibration
$P_F$ has small global sections, i.e.\ $P_F \dashv 1 \dashv G$.
\end{Thm}
\begin{proof}
Suppose that $F$ has a right adjoint $U$. We show that $1 \dashv G$ by 
exhibiting its counit $\widetilde{\varepsilon}_a$ for an arbitrary object
$a : A \to FI$ in $\C{\downarrow}F$.
For this purpose consider the pullback
\begin{diagram}[small]
C \SEpbk & \rTo^{q} & UA \\
\dTo^{p} & & \dTo_{Ua} \\
I & \rTo_{\eta_I} & UFI \\
\end{diagram}
where $\eta_I$ is the unit of $F \dashv U$ at $I \in \B$.
Then for the transpose $\widehat{q} = \varepsilon_A \circ Fq : FC \to A$ 
of $q$ we have
\begin{diagram}[small]
FC & \rTo^{\widehat{q}} & A \\
   & \rdTo(1,2)_{Fp} \ldTo(1,2)_{a} & \\
   &  FI &  \\
\end{diagram}
We show that $(p,\widehat{q}) : 1_C \to a$ is the desired counit 
$\widetilde{\varepsilon}_a$ of $1 \dashv G$ at $a$. Suppose that 
$(u,s) : 1_J \to a$ in $\C{\downarrow}F$, i.e.\ $u : J \to I$ and $s : FJ \to A$ 
with $a \circ s = Fu$ as shown in the diagram
\begin{diagram}[small]
FJ & \rTo^s & A \\
\dEqual & & \dTo_a \\
FJ & \rTo_{Fu} & FI \\
J & \rTo_u & I\\
\end{diagram}
We have to show that there is a unique $v : J \to C$ with $p \circ v = u$ 
and $\widehat{q} \circ Fv = s$ as shown in the diagram
\begin{diagram}[small]
FJ    &   &      & &                \\
      &  \rdTo_{Fv}    \rdTo(4,2)^{s} &      &          & \\
 \dEqual   &      &  FC  & \rTo_{\widehat{q}}   &       A   \\
    &               & \dEqual  &  \widetilde{\varepsilon}_a &  \dTo_{a}  \\
FJ & \rTo_{Fv}  & FC   & \rTo_{Fp} &        FI          \\
\end{diagram} 
But $\widehat{q} \circ Fv = s$ iff $q \circ v = Us \circ \eta_J$ 
due to $F \dashv U$. Thus $v$ satisfies the above requirements iff
$p \circ v = u$ and  $q \circ v = Us \circ \eta_J$, 
i.e.\ iff $v$ is the mediating arrow in the diagram
\begin{diagram}[small]
J  &           & & & \\
  & \rdTo~{v} \rdTo(2,4)_{u} \rdTo(4,2)^{Us \circ \eta_J} & & & \\
  &            &   \SEpbk & \rTo_{q} &   UA \\
  &            &  \dTo_{p}     &      &  \dTo_{Ua} \\
  &            &  I        & \rTo_{\eta_I} &   UFI  
\end{diagram}
from which there follows uniqueness and existence of $v$ with the 
desired properties. Thus $\widetilde{\varepsilon}_a$ actually is the counit
for $1 \dashv G$ at $a$.

\medskip
For the reverse direction assume that $P_F \dashv 1 \dashv G$. 
Thus, for all $X$ over $1$ we have 
$\B(-,GX) \cong \C{\downarrow}F(1_{(-)},X) \cong \C/F1(F_{/1}(-),X)$, i.e.\ 
$F_{/1} : \B \cong \B/1 \to \C/F1$ has a right adjoint (given by the 
restriction of $G$ to $\C/F1$). Since $\Sigma_{F1} \dashv (F1)^* : \C \to
\C/F1$ and $F = \Sigma_{F1} \circ F_{/1} : \B \cong \B/1 \to \C$ the functor
$F$ has a right adjoint. 

A slightly more abstract proof of the backwards direction goes by observing
that the inclusion $I : \C{\downarrow}F1 \hookrightarrow \C{\downarrow}F$ 
has a left adjoint $R$ sending $a : A \to FI$ to $F!_I \circ a : A \to F1$ and 
a morphism $(u,f)$ from $b : B \to FJ$ to $a : A \to FI$ to $f : R(b) \to R(a)$
since $(u,f)$ from $a : A \to FI$ to $c : C \to F1$ is in 1-1-correspondence
with $f : R(a) \to c$ (because necessarily $u =\; !_I$). Obviously, we have
that $F_{/1} = R \circ 1$ and thus $F_{/1}$ has right adjoint $G \circ I$.
Since $F = \Sigma_{F1} \circ F_{/1}$ and $\Sigma_{F1} \dashv (F1)^*$ it follows
that $F$ has right adjoint $G \circ I \circ (F1)^*$.
\end{proof}

\medskip
Notice that the above proof goes through if $\C$ has just pullbacks and 
$\Sigma_{F1}$ has a right adjoint $(F1)^*$, i.e.\ $F1 \times X$ exists 
for all objects $X$ in $\C$. 

Thus, we have the following lemma which has a structure analogous 
to the one of Lemma~\ref{cbis}.

\begin{Lem}\label{cbsgs}
Suppose $\B$ has finite limits and $\C$ has pullbacks and all products of
the form $F1 \times X$. Then for a functor $F : \B \to \C$
the following conditions are equivalent
\begin{enumerate}
\item[\emph{(1)}] $F$ has a right adjoint
\item[\emph{(2)}] $F^* : \Fib(\C) \to \Fib(\B)$ preserves the property
                  of having small global sections
\item[\emph{(3)}] $P_F = F^*P_\C = \partial_1 : \C{\downarrow}F \to \B$ 
                  has small global sections.
\end{enumerate} 
\end{Lem}
\begin{proof} The proof of (1) $\Rightarrow$ (2) is a special case of
the proof of Lemma~\ref{lspres}. Since $P_\C$ has small global sections
(3) follows from (2). Finally, claim (1) follows from (3) 
by Theorem~\ref{gmthm1} and the subsequent remark on its strengthening. 
\end{proof}

\bigskip
From Lemma~\ref{cbis} and Lemma~\ref{cbsgs} it follows that for a
functor $F : \B \to \C$ between categories with finite limits the fibration
$P_F = F^*P_\C$ has internal sums and small global sections iff $F$ 
preserves pullbacks and has a right adjoint.\footnote{This was already observed
by J.~B\'enabou in \cite{montreal}.}

\emph{Thus, for categories $\B$ with finite limits we get a 
1--1--correspondence (up to equivalence) between geometric morphisms to $\B$ 
(i.e.\ adjunctions $F \dashv U : \C \to \B$ where $\C$ has finite limits and 
$F$ preserves them) and fibrations over $\B$ of categories with finite limits,
stable disjoint sums and small global sections. Such fibrations are 
called {\bf geometric}.}

In Appendix~\ref{jibthm} we prove M.~Jibladze's theorem \cite{Jib} that in 
fibered toposes with internal sums these are automatically stable and disjoint.
As a consequence \emph{geometric morphisms from toposes to a topos $\Se$ are 
(up to equivalence) in 1--1--correspondence with toposes fibered over $\Se$ 
that are  cocomplete and locally small.}

\medskip
In the rest of this section we show that in a fibered sense every geometric
morphism is of the form $\Delta \dashv \Gamma$. 

First we observe that there is a fibered version of the functor
$\Delta = \DDelta \circ 1$ considered in the previous section
\begin{Def}\label{Deltafibdef}
Let $\B$ be a category with finite limits and $P : \X \to \B$ be a fibration
of categories with finite limits and stable disjoint internal sums. Then there
is a fibered functor $\Delta_P : P_\B \to P$ sending the morphism
\begin{diagram}[small]
I_1         &   \rTo^{v}   &  I_2             \\
\dTo^{u_1}  &              & \dTo_{u_2} \\
J_1         &   \rTo_{w}   &  J_2             \\
\end{diagram}
in $P_\B$ to the arrow $\Delta_P(w,v)$ in $\X$ over $w$ making the
following diagram commute

\goodbreak 
\begin{diagram}[small]
1_{I_1}        &   \rTo^{1_v}        &  1_{I_2}    \\
\dTo^{\varphi_{u_1}}  &              & \dTo_{\varphi_{u_2}} \\
\Delta_P(u_1)         &   \rTo_{\Delta_P(w,v)}   &  \Delta_P(u_2)   \\
\end{diagram}
where $\varphi_{u_i}$ is cocartesian over $u_i$ for $i=1,2$.
\end{Def}

Notice that $\Delta_P$ actually is cartesian as if the first square is a 
pullback then $\Delta_P(w,v)$ is cartesian by BCC for internal sums as 
$1_v$ is cartesian and the  $\varphi_{u_i}$ are cocartesian.

Now $P$ having small global sections turns out as equivalent to $\Delta_P$
having a fibered right adjoint $\Gamma_P$.

\begin{Thm}\label{gmthm2}
Let $\B$ be a category with finite limits and $P : \X \to \B$ be a fibration
of categories with finite limits and stable disjoint internal sums. Then $P$
has small global sections iff $\Delta_P$ has a fibered right adjoint $\Gamma_P$.
\end{Thm}
\begin{proof}
For the implication from left to right assume that $P \dashv 1 \dashv G$. 
For $X \in \X$ let $\widetilde{\varepsilon}_X$ be the unique vertical arrow
making the diagram 
\begin{diagram}[small]
1_{GX} & \rTo^{\varphi}_\cocart  & \Delta_P(P(\varepsilon_X))\\
     &  \rdTo_{\varepsilon_X}   &  \dTo_{\widetilde{\varepsilon}_X} \\
     &                   &  X \\
\end{diagram}
commute where $\varepsilon_X$ is the counit of $1 \dashv G$ at $X$. 
Then for $u : I \to J$ and $f : \Delta_P(u) \to X$ there is a unique 
morphism $(w,v) : u \to P(\varepsilon_X)$ with
$\widetilde{\varepsilon}_X \circ \Delta_P(w,v) = f$
as can be seen from the following diagram
\begin{diagram}
1_I & \rTo^{1_v} & 1_{G(X)} & & \\
\dTo^{\varphi_u}  & &  \dTo^{\varphi} & \rdTo(2,4)^{\varepsilon_X} & \\
\Delta_P(u) & \rTo^{\Delta_P(w,v)}  &  \Delta_P(P(\varepsilon_X)) & & \\
& \rdTo(4,2)_{f} & & \rdTo~{\widetilde{\varepsilon}_X} & \\
& & & & X \\
\end{diagram}
using the universal property of $\varepsilon_X$ and that necessarily $w = P(f)$. 
Thus, for $f : \Delta_P(u) \to X$ its lower
transpose $\check{f}$ is given by $(P(f),v) : u \to P(\varepsilon_X)$ 
where $v : I \to G(X)$ is the unique arrow with 
$\varepsilon_X \circ 1_v = f \circ \varphi_u$. 
The induced right adjoint $\Gamma_P$ sends a morphism 
$h : Y \to X$ in $\X$ to the morphism
\begin{diagram}
G(Y)        &   \rTo^{G(h)}        &  G(X)                \\     
\dTo^{P(\varepsilon_Y)}  &   \Gamma_P(h)   &  \dTo_{P(\varepsilon_X)}    \\
P(Y)        &   \rTo_{P(h)}        &  P(X)                \\          
\end{diagram}
in $\B^\two$ because $G(h)$ is the unique morphism $v$ with 
$\varepsilon_X \circ 1_v = h \circ \varepsilon_Y = 
 h \circ \widetilde{\varepsilon}_Y \circ \varphi_{P(\varepsilon_Y)}$ and,
therefore, $(P(h),G(h))$ is the lower transpose 
of $h \circ \widetilde{\varepsilon}_Y$ as required.
The unit $\widetilde{\eta}_u : u \to \Gamma_P(\Delta_P(u)) =
P(\varepsilon_{\Delta_P(u)})$ of $\Delta_P \dashv \Gamma_P$ at $u : I \to J$
is given by $\widetilde{\eta}_u$ making the following diagram commute
\begin{diagram}
I & & 1_I & &  \\
\dDashto^{\widetilde{\eta}_u} & & \dTo^{1_{\widetilde{\eta}_u}} & 
\rdTo^{\varphi_u}_\cocart & \\
G \Delta_P(u) & & 1_{G \Delta_P(u)} & \rTo_{\varepsilon_{\Delta_P(u)}} & 
\Delta_P(u)\\
\end{diagram}
because $(\id_{P(\Delta_P(u))}, \widetilde{\eta}_u)$ is the 
lower transpose of $\id_{\Delta_P(u)}$. As $P_\B\circ\Gamma_P = P$ and 
the components of $\widetilde{\eta}$ and $\widetilde{\varepsilon}$ are 
vertical it follows\footnote{This is an instance of a general fact about
fibered adjunctions whose formulation and (easy) verification we leave as an 
exercise to the reader.} that $\Gamma_P$ is cartesian and thus 
$\Delta_P \dashv \Gamma_P$ is a fibered adjunction.

\smallskip
For the implication from right to left suppose that $\Delta_P$ has a fibered
right adjoint $\Gamma_P$. We write $\widetilde{\varepsilon}$ for the counit 
of this adjunction. For $X \in \X$ we define $\varepsilon_X$ as 
$\widetilde{\varepsilon}_X \circ \varphi$
\begin{diagram}[small]
1_{GX} & \rTo^{\varphi}_{\mbox{cocart}} & \Delta_P \Gamma_P X\\
     &  \rdTo_{\varepsilon_X}   &  \dTo_{\widetilde{\varepsilon}_X} \\
     &                   &  X \\
\end{diagram}
where $\varphi$ is cocartesian over $P(\Gamma_P(X)) : G(X) \to P(X)$.
To verify the desired universal property of $\varepsilon_X$ assume that
$\sigma : 1_I \to X$ is a morphism over $u : I \to P(X)$. 
Let $\sigma = f \circ \varphi_u$ with $f$ vertical and $\varphi_u$ 
cocartesian. Then the existence of a unique arrow $v : I \to G(X)$
with $\varepsilon_X \circ 1_v = \sigma$ follows from considering the diagram
\begin{diagram}[small]
1_I & \rTo^{1_v} & 1_{G(X)} & & \\
\dTo^{\varphi_u}  & &  \dTo^{\varphi} & \rdTo(2,4)^{\varepsilon_X} & \\
\Delta_P(u) & \rTo^{\Delta_P(\id_{P(X)},v)}  &  \Delta_P \Gamma_P X & & \\
& \rdTo(4,2)_{f} & & \rdTo~{\widetilde{\varepsilon}_X} & \\
& & & & X \\
\end{diagram}
using the universal property of $\widetilde{\varepsilon}_X$.
Thus, $P$ has small global sections.
\end{proof}

\bigskip
The following explicitation of $\Delta_{P_F} \dashv \Gamma_{P_F}$ 
for finite limit preserving $F$ will be helpful later on.

\begin{Thm}\label{gmthm3}
For the geometric fibration $P = P_F$ induced by a geometric morphism 
$F \dashv U : \C \to \B$ the fibered adjunction $\Delta_P \dashv \Gamma_P$ 
can be described more concretely as follows.\\
The left adjoint $\Delta_P$ acts by application of $F$ to arrows and squares
in $\B$. The fiber of $\Gamma_P$ over $I \in \B$ is given by 
$\eta_I^* \circ U_{/I}$. The unit $\widetilde{\eta}_u$ for $u : I \to J$ 
is given by
\begin{diagram}[small]
I  &            & & & \\
   & \rdTo~{\widetilde{\eta}_u} \rdTo(2,4)_{u} \rdTo(4,2)^{\eta_I} & & & \\
   &            &  K \SEpbk & \rTo_{q} &   UFI  \\
   &            &  \dTo_{p}     &      &  \dTo_{UFu} \\
   &            &  J        & \rTo_{\eta_J} &   UFJ \\
\end{diagram} 
and for $a : A \to FI$ the counit $\widetilde{\varepsilon}_a$ is given by 
$\varepsilon_A \circ Fq : Fp \to a$ where
\begin{diagram}[small]
K \SEpbk  & \rTo^{q} &   UA       \\
\dTo^{p}  &          &  \dTo_{Ua} \\
I        & \rTo_{\eta_I} &   UFI  \\
\end{diagram} 
\end{Thm}
\begin{proof}
Straightforward exercise when using the description of $\varepsilon$ from the 
proof of Theorem~\ref{gmthm1} and the descriptions of $\widetilde{\eta}$ and 
$\widetilde{\varepsilon}$ from the proof of Theorem~\ref{gmthm2}.
\end{proof}

\newpage

\section{Fibrational Characterisation of Boundedness}

Recall (e.g.\ from \cite{Joh}) that a geometric morphism 
$F \dashv U : \E \to \Se$ between elementary toposes is called \emph{bounded} 
iff there is an object $S \in \E$ such that for every $X \in \E$ there is an 
object $I \in \Se$ such that $X$ appears as a subquotient of $S{\times}FI$
\begin{diagram}[small]
C  & \rEmbed & S{\times}FI \\
\dOnto & & \\
X & & \\
\end{diagram}
i.e.\ $X$ appears as quotient of some subobject $C$ of $S{\times}FI$. 
Such an $S$ is called a \emph{bound} for the geometric morphism $F \dashv U$.
The importance of bounded geometric morphisms lies in the fact that they 
correspond to Grothendieck toposes over $\Se$ (as shown e.g.\ in \cite{Joh}). 

In this section we will show that a geometric morphism 
$F \dashv U : \E \to \Se$ is bounded iff for its corresponding geometric 
fibration $P_F$  there exists a generating family.

\begin{Lem}\label{ccmon}
Let $\B$ have finite limits and $P : \X \to \B$ be a fibration of categories
with finite limits with stable disjoint internal sums. Then a cocartesian
arrow $\varphi : X \to Y$ is monic w.r.t.\ to vertical arrows, i.e.\ 
vertical arrows $\alpha_1, \alpha_2 : Z \to X$ are equal whenever
$\varphi \circ \alpha_1 = \varphi \circ \alpha_2$.
\end{Lem}
\begin{proof}
Let $\alpha_1,\alpha_2 : Z \to X$ be vertical arrows with 
$\varphi \circ \alpha_1 = \varphi \circ \alpha_2$. Then there is a unique
morphism $\alpha$ with $\pi_i \circ \alpha = \alpha_i$ for $i=1,2$. 
Consider the pullback
\begin{diagram}[small]
\cdot \SEpbk & \rTo^{\beta} & X \\
\dTo^{\psi} & & \dTo_{\delta_\varphi} \\
Z & \rTo_{\alpha} & \cdot  \\
\end{diagram}
where $\delta_\varphi$ is the fiberwise diagonal. Notice that both $\alpha$
and $\delta_\varphi$ are above the same mono in $\B$. Thus, the map $\psi$ 
lies above an isomorphism  in the base (as $P$ preserves pullbacks) and, 
moreover, it is cocartesian as it appears as pullback of the cocartesian 
arrow $\delta_\varphi$. Thus, the arrow $\psi$ is an isomorphism and we have
$\alpha = \delta_\varphi \circ \beta \circ \psi^{-1}$ from which it follows
that $\alpha_i = \beta \circ \psi^{-1}$ for $i{=}1,2$. 
Thus, we have $\alpha_1 = \alpha_2$ as desired.

Alternatively, one may argue somewhat simpler as follows. For $i{=}1,2$ we have 
$\alpha_i \circ \psi = \pi_i \circ \alpha \circ \psi = 
 \pi_i \circ \delta_\varphi \circ \beta = \beta$. Accordingly, we have 
$\alpha_1 \circ \psi = \alpha_2 \circ \psi$ from which it follows
that $\alpha_1 = \alpha_2$ since $\psi$ is cocartesian and the $\alpha_i$ are 
vertical.
\end{proof}

\bigskip
For formulating the next lemma we have to recall the notion of collectively
epic morphism as introduced in Theorem~\ref{genfamthm}. If
$P : \X\to\B$ is a fibration then a morphism $f : X \to Y$ in $\X$ is
called \emph{collectively epic} iff for all vertical arrows 
$\alpha_1,\alpha_2 : Y \to Z$ from $\alpha_1 \circ f = \alpha_2 \circ f$ 
it follows that $\alpha_1 = \alpha_2$. Notice that for a collectively epic
morphism $f : X \to Y$ for maps $g_1,g_2 : Y \to Z$ with $P(g_1) = P(g_2)$
from $g_1 f = g_2 f$ it follows that $g_1 = g_2$ because 
if $g_i = \varphi \alpha_i$ with $\varphi$ cartesian and $\alpha_i$ vertical 
then $\alpha_1 f = \alpha_2 f$ and thus $\alpha_1 = \alpha_2$ from which it
follows that $g_1 = \varphi \alpha_1 = \varphi \alpha_2 = g_2$. 

If $P$ is $\Fam(\C)$ for an ordinary category
$\C$ then an arrow $f : X \to Y$ in the total category of $\Fam(\C)$ 
over $u : I \to J$ is collectively epic iff for all $j \in J$ the family
$(f_i : X_i \to Y_j)_{i{\in}u^{-1}(j)}$ is collectively epic in the usual
sense of ordinary category theory. Thus, it would be more precise to say
``family of collectively epic families'' but as this formulation is too
lengthy we prefer the somewhat inaccurate formulation ``collectively epic''.

Notice that for a bifibration $P : \X \to \B$ a morphism $f : X \to Y$ in
$\X$ is collectively epic iff for a cocartesian/vertical factorisation 
$f = \alpha \circ \varphi$ the vertical arrow $\alpha$ is epic in its fiber.

\begin{Lem}\label{bgmlem1}
Let $\B$ be a category with finite limits and $P : \X \to \B$ a geometric 
fibration which is locally small and well--powered. Moreover, suppose that 
collectively epic arrows in $\X$ are stable under pullbacks.

Then for $P$ there exists a generating family iff for $P$ there exists a
\emph{separator}, i.e.\ an object $S \in P(1_ \B)$ such that for every 
object $X \in P(1_\B)$ there exist morphisms $\varphi : Y \to S$, 
$m : Z \to Y$ and $\psi : Z \to X$ with $\varphi$ cartesian, $m$ a vertical 
mono and $\psi$ collectively epic.
\end{Lem}
\begin{proof}
Let $P : \X \to \B$ be a fibration satisfying the conditions above.

Suppose that $G \in P(I)$ is a generating family for $P$. Let 
$\psi_0: G \to S$ be a cocartesian arrow over $!_I : I \to 1$.
Let $\psi_0 = \varphi_0 \circ \eta$ with $\varphi_0$ cartesian and 
$\eta$ vertical. Notice that $\eta$ is monic as by Lemma~\ref{ccmon} the 
cocartesian $\psi_0$ is monic w.r.t.\ vertical arrows.
We show that $S$ is a separator for $P$. Let $X \in P(1_\B)$. 
As $G$ is a generating family for $P$ and $\B$ has binary products by 
Theorem~\ref{genfamthm} there are morphisms $\theta : Z \to G$ and 
$\psi : Z \to X$ with $\theta$ cartesian and $\psi$ collectively epic. 
Then consider the diagram

\begin{diagram}[small]
X & \lTo^{\psi} & Z \SEpbk & \rTo^{\theta}  & G & &  \\
  &             & \dEmbed^{m} &    & \dEmbed^{\eta} & \rdTo^{\psi_0} & \\
  &             &  Y       & \rTo_{\theta'} & I^*S & \rTo_{\varphi_0} & S \\
\end{diagram}
where $\theta'$ is cartesian over $P(\theta)$ and $m$ is vertical. Thus, the 
middle square is a pullback and $m$ is a vertical mono. Furthermore, 
$\varphi := \varphi_0 \circ \theta'$ is cartesian. Thus, we have constructed
morphisms $\varphi : Y \to S$, $m : Z \to Y$ and $\psi : Z \to X$ with 
$\varphi$ cartesian, $m$ a vertical mono and $\psi$ collectively epic 
as required.

\medskip
Suppose that $S \in P(1_\B)$ is a separator for $P$. By well--poweredness 
of $P$ there exists a vertical mono $m_S : G \mono \sigma_S^*S$ classifying 
families of subobjects of $S$. We show that $G$ is a generating family for $P$.

Suppose $X \in P(I)$. Let $\theta_0 : X \to X_0$ be a cocartesian arrow 
over $!_I : I \to 1$ . As $S$ is a separator there exist morphisms 
$\varphi_0 : Y_0 \to S$, $m_0 : Z_0 \to Y_0$ and $\psi_0 : Z_0 \to X_0$ with 
$\varphi_0$ cartesian, $m_0$ a vertical mono and $\psi_0$ collectively epic.
Consider the pullback
\begin{diagram}[small]
Z \SEpbk     & \rTo^{\theta} & Z_0            \\
\dTo^{\psi}  &                & \dTo_{\psi_0}  \\
X            & \rTo_{\theta_0}  & X_0            \\
\end{diagram}
where $\psi$ is collectively epic and $\theta$ is cocartesian since these
classes of arrows are stable under pullbacks. Consider further the diagram
\begin{diagram}
Z &            & & & \\
  & \rdTo~{\eta} \rdTo(2,4)_{m} \rdTo(4,2)^{\theta} & & & \\
  &            &  \cdot \SEpbk  & \rTo_{\varphi'}           &   Z_0 \\
  &            &  \dEmbed_{m'}     &                  &  \dEmbed_{m_0} \\
  &            &  Y        & \rTo_{\varphi_1} &   Y_0 \\  
\end{diagram}
where $\varphi_1$ and $\varphi'$ are cartesian over $P(\theta)$ and $m'$
and $\eta$ are vertical. The inner square is a pullback and thus $m'$ is 
monic as it appears as pullback of the monic arrow $m_0$. The arrow $\eta$ 
is a vertical mono as by Lemma~\ref{ccmon} $\theta$ is monic w.r.t.\
vertical arrows. Thus $m = m' \circ \eta$ is a vertical mono, too. 
Moreover, $\varphi_0 \circ \varphi_1 : Y \to S$ is cartesian. Thus, the 
mono $m : Z \to Y$ is a family of subobjects of $S$ and, accordingly, we have
\begin{diagram}[small]
Z \SEpbk & \rTo^{\varphi} & G \\
\dEmbed^{m} & & \dEmbed_{m_S} \\
Y  & \rTo_{\widetilde{\varphi}} & \sigma_S^*S \\
\end{diagram}
for some cartesian arrows $\varphi$ and $\widetilde{\varphi}$. Thus, we have
morphisms $\varphi : Z \to G$ and $\psi : Z \to X$ with $\varphi$ cartesian
and $\psi$ collectively epic. 

Thus, by Theorem~\ref{genfamthm} it follows 
that $G$ is a generating family for $P$.
\end{proof}

\bigskip
Suppose $F : \B \to \C$ is a finite limit preserving functor between 
categories with finite limits. One easily checks that an arrow
\begin{diagram}[small]
B & \rTo^e & A \\
\dTo^b & f & \dTo_a \\
FJ & \rTo_{Fu} & FI \\
J & \rTo_u & I
\end{diagram} 
in $\C{\downarrow}F$ is collectively epic (w.r.t.\ the fibration
$P_F = \partial_1 : \C{\downarrow}F \to \B$) iff the map $e$ is epic 
in $\C$. Apparently, this condition is sufficient. On the other hand if $f$
is collectively epic then $e$ is epic in $\C$ which can be seen as follows:
suppose $g_1,g_2 : A \to C$ with $g_1 e = g_2 e$ then the maps
\begin{diagram}[small]
A & \rTo^{g_i} & C \\
\dTo^a & \alpha_i & \dTo \\
FI & \rTo_{F!_I} & F1 \\
I & \rTo_{!_I} & 1
\end{diagram}
are both above $I \to 1$ and satisfy $\alpha_1 f = \alpha_2 f$ from which
it follows -- since $f$ is collectively epic -- that $\alpha_1 = \alpha_2$ and
thus $g_1 = g_2$. 

Thus, if in $\C$ epimorphisms are stable under pullbacks along arbitrary
morphisms then in $\C{\downarrow}F$ collectively epic maps are stable under
pullbacks along arbitrary morphisms.

\begin{Thm}\label{bgmthm}
A geometric morphism $F \dashv U : \E \to \Se$ between toposes is bounded iff 
for the corresponding geometric fibration $P_F$ there exists a 
generating family.
\end{Thm}
\begin{proof}
Let $F \dashv U : \E \to \Se$ be a geometric morphism between toposes. 
Then the corresponding geometric fibration $P_F$ is locally small and
well-powered. 

As for $P_F$ reindexing preserves the topos structure and 
in toposes epis are stable under pullbacks vertical epis are stable under 
pullbacks. Thus, collectively epic arrows are stable under pullbacks as 
both vertical epis and cocartesian arrows are stable under pullbacks. 
Alternatively, this follows from the observations immediately preceding 
the current theorem and pullback stability of epimorphisms in toposes.

Thus, since the assumptions of Lemma~\ref{bgmlem1} are satisfied for $P_F$ 
there exists a generating family for $P_F$ iff there exists a separator
for $P_F$ which, obviously, is equivalent to the requirement that the
geometric morphism $F \dashv U$ is bounded.
\end{proof}

\bigskip
From inspection of the proof of Lemma~\ref{bgmlem1} it follows\footnote{In 
more concrete terms for the fibration $P_F = F^*P_\E$ this can be seen as
follows. Suppose $a : A \to F(I)$ is a map in $\E$. As $S$ is a bound 
there exists $J \in \Se$ and $e : C \epi A$ with $n : C \mono F(J) \times S$.
Consider the diagram
\begin{diagram}
A & \lOnto^{\;\;\;e} & C       &         & \\
  &           & \dEmbed^m  &    \rdEmbed^n & \\
\dTo^a & & F(I{\times}J){\times}S \SEpbk & \rTo^{F(\pi'){\times}S} 
                                                  & F(J){\times}S \\
  & & \dTo^\pi & & \dTo_\pi \\
F(I) & \lTo_{\;\;\;F(\pi)} & F(I{\times}J) & \rTo_{F(\pi')} & F(J) \\
\end{diagram}
(where $F(\pi)$ and $F(\pi')$ form a product cone 
because $F$ preserves finite limits and $\pi$ and $\pi'$ form a product
cone) and notice that $\pi \circ m$ appears as pullback of $g_S$ along
$F(\rho)$ where $\rho : I{\times}J \to U\Pow(S)$ is the unique map 
classifying $m$, i.e.\ 
$((\varepsilon_{\Pow(S)}{\circ}F(\rho)){\times}S)^*{\ni_S} \;\cong m$.}
in particular that if $S\in\E$ is a bound for a geometric morphism 
$F\dashv U:\E\to\Se$ between toposes 
then $g_S = \pi \circ m_S : G_S \to F(U\Pow(S))$

\begin{diagram}[small]
G_S \SEpbk & \rTo & \ni_S \\
\dEmbed^{m_S} & & \dEmbed_{\ni_S} \\
FU\Pow(S){\times}S \SEpbk & \rTo^{\varepsilon_{\Pow(S)}{\times}S} &
\Pow(S){\times}S\\
\dTo^{\pi} & & \dTo_{\pi} \\
FU\Pow(S) & \rTo_{\varepsilon_{\Pow(S)}} & \Pow(S)
\end{diagram}
is a generating family for $P_F$. This condition, however, also implies 
that $S$ is a bound for $F \dashv U$ since if $g_S = \pi \circ m_S$ is a 
generating family for $P_F$ then for every $A \in \E$ there is a map 
$u : I \to U\Pow(S)$ in $\Se$ and an epi $e : u^*G_S \epi A$ such that
\begin{diagram}[small]
A & \lOnto^e & u^*G_S \SEpbk & \rTo & G_S \SEpbk & \rTo & \ni_S \\
\dTo & & \dEmbed^m & & \dEmbed^{m_S} & & \dEmbed_{\ni_S} \\
& & FI{\times}S \SEpbk & \rTo^{Fu{\times}S} & FU\Pow(S){\times}S 
\SEpbk & \rTo^{\varepsilon_{\Pow(S)}{\times}S} &
\Pow(S){\times}S\\
& & \dTo^{\pi} & & \dTo^{\pi} & & \dTo_{\pi}\\
F1 & \lTo_{F!_I} & FI & \rTo_{Fu} & FU\Pow(S) & \rTo_{\varepsilon_{\Pow(S)}} 
& \Pow(S)
\end{diagram}
from which it follows that $A$ appears as quotient of a subobject of some
$FI{\times}S$. 

Thus $S$ is a bound for a geometric morphism $F \dashv U : \E\to\Se$ between 
toposes iff $g_S = \pi \circ m_S : G_S \to FU\Pow(S)$ is a generating
family for $P_F$. In case $\Se$ is $\Set$ this amounts to the usual 
requirement that the family of subobjects of $S$ is generating for the 
topos $\E$.

\smallskip
One can characterize boundedness of geometric morphisms in terms 
of preservation properties as follows.

\begin{Thm}\label{bgmthm'}
A geometric morphism $F \dashv U : \E \to \Se$ between toposes is bounded
if and only if change of base along $F$ preserves existence of small 
generating families for geometric fibrations of toposes.

Thus, a terminal object preserving functor $F : \Se\to\E$ between toposes
is the inverse image part of a bounded geometric morphism iff
$F^*P_\E$ is is a geometric fibration with a small generating family iff
change of base along $F$ preserves the property of being a geometric
fibration of toposes with a small generating family.
\end{Thm}
\begin{proof}
Suppose $F^*$ preserves existence of small generating families for 
geometric fibrations of toposes. Then since 
$P_\E : \E^\two \to \E$ has a small generating family
so does $F^*P_\E = P_F$ and thus by Theorem~\ref{bgmthm} the 
geometric morphism $F \dashv U$ is bounded.

On the other hand if a geometric morphism $F \dashv U$ is bounded
and $P$ is a geometric fibration of toposes over $\E$ with a small
generating family then by Theorem~\ref{bgmthm} $P$ is equivalent to
$G^*P_\F = P_G$ for some bounded geometric morphism $G \dashv V : \F \to \E$
and thus $F^*P$ is equivalent to $F^*G^*P_\F \simeq (GF)^*P_\F = P_{GF}$
which has a small generating family by Theorem~\ref{bgmthm} 
since $GF \dashv UV$ is a bounded geometric morphism (as by \cite{Joh} 
bounded geometric morphisms are closed under composition).

For the second claim first recall that by Lemma~\ref{cbis} and
Lemma~\ref{cbsgs} a terminal object preserving functor $F : \Se\to\E$
between toposes is the inverse image part of a geometric morphism iff
$F^*P_\E$ is a geometric fibration of toposes iff change of base
along $F$ preserves geometric fibrations of toposes.

Thus, if $F$ is the inverse image part of a bounded geometric morphism
then $F^*P_\E$ is a geometric fibration with a small generating family and
change of base along $F$ preserves geometric fibrations with a small
generating family.

Suppose $F : \Se\to\E$ preserves terminal objects. Thus, if $F^*P_\E$ is
a geometric fibration with a small generating family then $F$ is the inverse
image part of a geometric morphism which by Theorem~\ref{bgmthm} is also
bounded.
If $F^*$ preserves geometric fibrations with a small generating family
then $F^*P_\E$ is a geometric fibration with a small generating family and
thus $F$ is the inverse image part of a bounded geometric morphism. 
\end{proof}

\medskip
Thus we may observe that for (bounded) geometric morphisms $F \dashv U$ 
change of base along $F$ for geometric fibrations of toposes (with a small 
generating  family) corresponds to postcomposition with $F \dashv U$ for 
(bounded) geometric morphisms.\footnote{A consequence of this observation
is that change of base along inverse image parts of geometric morphisms
for geometric fibrations of toposes reflects the property of having a small 
generating family since as observed in \cite{Joh} for geometric morphisms 
$f$ and $g$ from $fg$ bounded it follows that $g$ is bounded.}

By Theorem~\ref{bgmthm'} change of base along inverse image parts of unbounded 
geometric morphisms does not preserve existence of small generating families.
From \cite{Joh} we recall the following example of an unbounded geometric 
morphism. Let $\E$ be the full subcategory of $\widehat{\Ze} = \Set^{\Ze^\op}$
on those objects $A$ such that $\forall a{\in}A. A(n)(a) = a$ for some 
$n\in\N$, i.e.\ there is a finite bound on the size of all orbits of the
action $A$. One easily sees that $\E$ is a topos and 
$\Delta \dashv \Gamma : \E \to \Set$ is a geometric morphism which, however, 
is not bounded (as otherwise there were an \emph{a priori} bound on the
size of all orbits of objects of $\widehat{\Ze}$). Notice, however, that 
$\E$ admits a (countable) generating family in the sense of ordinary category 
theory, namely the family $(\Ze_n)_{n\in\N}$ (of all finite orbits up to 
isomorphism), whose sum, however, does not exist in $\E$. 

Johnstone's example also demonstrates that toposes $\E$ over $\Set$ need not 
be cocomplete in the sense of ordinary category theory, i.e.\ do not have all 
small sums, although the associated fibration $P_\Delta$ certainly has
internal sums.\footnote{Thus, the fibrations $P_\Delta = \Delta^*P_\E$
and $\Fam(\E)$ over $\Se$ are not equivalent because $P_\Delta$ has internal
sums whereas $\Fam(\E)$ doesn't!

Consider also the following somewhat weaker counterexample. Let $\mathcal{A}$ 
be a \emph{partial combinatory algebra}, $\mathbf{RT}[\mathcal{A}]$ the 
realizability topos over $\mathcal{A}$ (see e.g.\ \cite{realizbook})
and $\Gamma \dashv \nabla : \Set \to \mathbf{RT}[\mathcal{A}]$ the geometric 
morphism where $\Gamma = \mathbf{RT}[\mathcal{A}](1,-)$ is the global elements 
functor. 
Then $P_\nabla = \nabla^*P_{\mathbf{RT}[\mathcal{A}]}$ is a fibration
with stable and disjoint internal sums over $\Set$ although for 
{\bf nontrivial} $\mathcal{A}$ in the realizability topos 
$\mathbf{RT}[\mathcal{A}]$ the sum $\coprod_{|\mathcal{A}|} 1$ does not exist 
for cardinality reasons. 

Moreover, for nontrivial $\mathcal{A}$ internal sums w.r.t.\ the fibration 
$P_\nabla$ in general do not coincide with the corresponding external sums 
(if they exists): consider e.g.\ $\coprod_2 1$ w.r.t.\ $P_\nabla$, i.e.\ 
$\nabla(2)$, which is not isomorphic to $1+1$ in $\mathbf{RT}(\mathcal{A})$. 
Thus $\nabla(1+1) \not\cong \nabla(1) + \nabla(1)$ from which it follows
that $\nabla$ does not have a right adjoint. Accordingly, the fibration 
$P_\nabla$ over $\Set$ does not have small global elements.}
Apparently there is a difference between \emph{internal} and 
\emph{external} families of objects in $\E$ where a family $(X_i)_{i \in I}$ 
in $\E$ is internal if there is a map $f : Y \to \Delta(I)$ in $\E$ with 
$X_i \cong \inj_i^*f$ for all $i \in I$. Of course, every internal family
gives rise to an external one whereas e.g.\ $(\Ze_n)_{n \in \N}$ is an 
external family in the topos $\E$ which is not internal. 
It is an easy exercise to show that a family 
$(X_i)_{i \in I}$ in a topos $\E$ over $\Set$ is internal if and only if 
the family $(X_i)_{i \in I}$ is \emph{bounded} in the sense that there exists 
an object $X \in \E$ 
such that all $X_i$ appear as subobjects of $X$.\footnote{In the example
$P_\nabla = \partial_1 : 
\mathbf{RT}[\mathcal{A}] {\downarrow} \nabla\to\Set$ every $X \to \nabla(I)$
in $\mathbf{RT}[\mathcal{A}]$ may be understood as a family of $I$-indexed
subobjects of $X$ but the ensuing cartesian functor (over $\Set$) from 
$P_\nabla$ to $\Fam(\mathbf{RT}[\mathcal{A}])$ is far from being an 
equivalence. 

Firstly, it does not reflect isos (in each fiber) since $\id_{\nabla(2)}$ and 
$\eta_2 : 2 \to \nabla\Gamma(2) \cong \nabla(2)$ are not isomorphic in the 
slice over $\nabla(2)$ but both give rise to $(1)_{i \in 2}$ 
in $\Fam(\mathbf{RT}[\mathcal{A}])(2)$. Thus, different internal families 
(over $2$ already) may give rise to the same external family.

Secondly, there are external $\N$-indexed families $(X_n)_{n\in\N}$ in 
$\mathbf{RT}[\mathcal{A}]$ which do not arise from a morphism 
$X \to \nabla(\N)$ because any such family would have to be isomorphic
to a family $(X^\prime_n)_{n\in\N}$ for which symmetry and transitivity
are realized by $e_1, e_2 \in \mathcal{A}$ independently from $n \in \N$.
It is left as an exercise to the reader to give a concrete counterexample.}

Notice that due to Giraud's Theorem (see \cite{Joh}) toposes bounded 
over $\Set$ are precisely the Grothendieck toposes and, therefore, do have 
all small sums. Actually, one may see this more directly as follows. 
Suppose $S$ is a bound for the geometric morphism 
$\Delta \dashv \Gamma : \E \to \Set$. Then $\E$ has all small copowers
$\coprod_{i \in I} X \cong \Delta(I) \times X$. Suppose $(X_i)_{i \in I}$ is
a family in $\E$. Then for every $i \in I$ there is a set $J_i$ such that
$X_i$ is a subquotient of $\Delta(J_i) \times S$. Thus, all $X_i$ are 
subobjects of $\Pow(\Delta(J){\times}S)$ via some mono $m_i$ 
where $J = \bigcup_{i \in I} J_i$ 
(since $\Delta(J_i) \times S$ is a subobject of $\Delta(J) \times S$).
Let $\chi_i$ classify the subobject $m_i$ for $i \in I$ and $\chi :
\coprod_{i \in I} \Pow(\Delta(J){\times}S) \to \Omega$ be the source
tupling of the $\chi_i$. Then the sum $\coprod_{i \in I} X_i$ appears 
as the subobject of the copower
$\coprod_{i \in I} \Pow(\Delta(J){\times}S) \cong 
 \Delta(I){\times}\Pow(\Delta(J){\times}S)$ 
classified by $\chi$ in $\E$.

But there exist toposes over $\Set$ which, in the sense of ordinary 
category theory, are cocomplete but do not admit a small generating family. 
A typical such example (due to Peter Freyd) is the topos $\Gcal$ whose objects 
are pairs $(A,f)$ where $A$ is a set and $f$ is a family of bijections of $A$ 
indexed over the class of all sets such that the class 
$\supp(A,f) = \{ s \mid f_s \not= \id_A \}$ is a set and 
whose morphisms from $(A,f)$ to $(B,g)$ are the maps $h : A \to B$ 
with $h(f_s(a)) = g_s(h(a))$ for all $a \in A$ and all sets $s$. 
The construction of this category can be rephrased as follows. 
Let $\Gbb$ be the free group generated by the class of all sets. 
Then $\Gcal$ is isomorphic to the full subcategory of $\widehat{\Gbb}$ 
on those objects $A$ where $\{s \mid A(s) \not= \id_{A(*)}\}$ is a set. 
The proof that $\Gcal$ is a topos is analogous to the proof that for every 
group $G$ the presheaf category $\Set^{G^\op}$ is a boolean topos. Moreover 
$\Gcal$ has all small
limits and colimits (which are constructed pointwise). Suppose 
$(G_i,g^{(i)})_{i \in I}$ were a small generating family for $\Gcal$. 
Let $J = \bigcup_{i \in I} \supp(G_i,g^{(i)})$ and $s_0$ be a set with 
$s_0 \not\in J$. Now let $(A,f)$ be the object of $\Gcal$ where $A = \{0,1\}$ 
and $f_s \not= \id_A$ only for $s = s_0$. There cannot exist a morphism 
$h : (G_i,g^{(i)}) \to (A,f)$ unless $G_i$ is empty as otherwise there is 
a $z \in G_i$ for which we have $h(z) = h(g^{(i)}_{s_0})(z)) = f_{s_0}(h(z))$
Obviously $(A,f)$ has two different endomorphisms which, however, cannot be 
distinguished by morphisms of the form $h : (G_i,g^{(i)}) \to (A,f)$. Thus,
there cannot exist a small generating family for the cocomplete boolean topos
$\Gcal$.

One easily shows that for a cocomplete topos $\E$ the functor 
$\Delta : \Set \to \E$ preserves finite limits. 
Thus, for a locally small topos $\E$ it holds that
$$ \E \mbox{ bounded over } \Set \Longrightarrow \E \mbox{ cocomplete }
   \Longrightarrow \E \mbox{ over } \Set $$
and the above counterexamples show that none of these implications can be
reversed in general.\footnote{The category $\Set^{\mathbf{Ord}^\op}$ of $\Set$-valued presheaves over the large category $\mathbf{Ord}$ of ordinals is an example of a cocomplete topos which, however, is not locally small since there are class many subterminals and thus $\Omega$ has class many global elements.} 
Freyd's counterexample shows that the first implication cannot be reversed in 
general and Johnstone's counterexample shows that the second implication 
cannot be reversed in general. 

If $\E$ is a topos bounded over $\Set$ then for $\E$ there exists a generating
family in the sense of ordinary category theory. However, as Johnstone's 
counterexample shows the reverse implication does not hold in general for 
toposes over $\Set$. Freyd's counterexample shows there are toposes $\E$ 
over $\Set$ such that there does not even exist a generating family for $\E$
in the sense of ordinary category theory and that such toposes may even be
cocomplete.

Notice that toposes $\E$ cocomplete in the sense of ordinary category theory 
are bounded over $\Set$ iff there exists a generating family for $\E$ in the 
sense of ordinary category theory. The reason is that if $(G_i)_{i \in I}$
is a generating family for $\E$ in the sense of ordinary category theory then 
$\coprod_{i \in I} !_{G_i} : \coprod_{i \in I} G_i \to \coprod_{i \in I} 1_\E 
= \Delta(I)$ is a generating family for the fibration
$\Delta^*P_\E = P_\Delta$. Thus, a topos $\E$ is bounded over $\Set$ iff
$\E$ is cocomplete and there exists a generating family for $\E$ in the sense
of ordinary category theory. However, this characterisation does not generalise
to arbitrary base toposes $\Se$. Formally, the fibrational characterisation
of bounded toposes over $\Se$ as cocomplete locally small fibered toposes 
over $\Se$ with a generating family looks similar but as we have seen above 
cocomplete in the sense of fibered categories is weaker than cocomplete in the
sense of ordinary category theory and generating family in the sense of fibered 
categories is stronger than in the sense of ordinary category theory.

Finally we observe that a topos over $\Set$ which in the sense of ordinary
category theory is neither cocomplete nor has a small generating family can
be obtained by combining the ideas of Freyd's and Johnstone's counterexamples,
namely the full subcategory of Freyd's counterexample $\Gcal$ on those objects
$(A,f)$ for which there exists an $n \in \N$ such that $(f_s)^n = \id_A$ for
all sets $s$.

\newpage

\section{Properties of Geometric Morphisms}

In this section we will characterise some of the most common properties of 
geometric morphisms $F \dashv U$ in terms of simple fibrational properties
of the corresponding geometric fibration $P_F$ for which we often simply 
write $P$. Moreover, the fibered adjunction $\Delta_P \dashv \Gamma_P$ 
induced by $P$ is  often referred to as $\Delta \dashv \Gamma$
and the corresponding unit and counit are denoted by $\widetilde{\eta}$ 
and $\widetilde{\varepsilon}$, respectively.

\subsection{Injective Geometric Morphisms}

\begin{Thm}\label{gmprop1}
Let $F \dashv U : \C \to \B$ be a geometric morphism and $P$ be the induced
geometric fibration $P_F$. Then the following conditions are equivalent.
\begin{enumerate}
\item[\rm (1)] The geometric morphism $F \dashv U$ is \emph{injective}, 
i.e.\ $U$ is full and faithful.
\item[\rm (2)] The counit $\widetilde{\varepsilon}$ of $\Delta \dashv \Gamma$
is a natural isomorphism.  
\item[\rm (3)] For the counit $\varepsilon$ of $1 \dashv G : \C{\downarrow}F \to \B$ 
it holds that $\varepsilon_X$ is cocartesian for all objects $X \in \C{\downarrow}F$.
\end{enumerate}
\end{Thm}
\begin{proof} Conditions (2) and (3) are equivalent as by Theorem~\ref{gmthm2}
we have $\varepsilon_X = \widetilde{\varepsilon}_X \circ \varphi$ with
$\varphi : 1_{GX} \to \Delta\Gamma X$ cocartesian over $P(\varepsilon_X)$.

Condition (2) says that all $\Gamma_I$ are full and faithful. In particular, 
we have that $U\cong\Gamma_1$ is full and faithful. Thus (2) implies (1).

It remains to show that (1) entails (2). Condition (1) says that the counit 
$\varepsilon$ of $F \dashv U$ is a natural isomorphism. 
But then for every $a : A \to FI$ in $\C{\downarrow}F$ we have
\begin{diagram}
FK \SEpbk  & \rTo^{Fq} &   FUA  \SEpbk & \rTo^{\varepsilon_A}_{\cong} & A \\
\dTo^{Fp}  &           &  \dTo_{FUa}   &                    & \dTo_{a }\\
FI        & \rTo_{F\eta_I} &   FUFI & \rTo_{\varepsilon_{FI}}^{\cong} & FI \\
\end{diagram}
from which it follows by  Theorem~\ref{gmthm3} that the map 
$\widetilde{\varepsilon}_a = \varepsilon_A \circ Fq$ 
is an isomorphism as it appears as pullback of the identity
$\id_{FI} = \varepsilon_{FI} \circ F(\eta_I)$.
\end{proof}

\bigskip
The full subfibration of $P_\B$ as given by $\Gamma$ may be characterized
as the class of all morphisms in $\B$ for which the naturality square for
the unit of $F \dashv U$ is a pullback. 

In a paper from 2021 Joyal~et.al.\ have characterized such classes of maps 
as the right part of so called ``lex stable factorization systems'', i.e.\ 
factorization systems $(\mathcal{L},\mathcal{R})$ on $\B$ such that 
$\mathcal{L}$ is stable under pullbacks along arbitrary morphisms in $\B$ 
and $\mathcal{R}$ consists of those maps in $\B$ for which the naturality 
square for the unit of the ensuing adjunction is a pullback.

\subsection{Surjective Geometric Morphisms}

\begin{Thm}\label{gmprop2}
Let $F \dashv U : \C \to \B$ be a geometric morphism and $P$ be its induced
geometric fibration $P_F$. Then the following conditions are 
equivalent.\footnote{This holds without assuming that $F$ has a right 
adjoint. It suffices that $F$ preserves finite limits.}
\begin{enumerate}
\item[\rm (1)]
The geometric morphism $F \dashv U$ is \emph{surjective}, 
i.e.\ $F$ reflects isomorphisms.
\item[\rm (2)] 
A morphism $u$ in $\B$ is an isomorphism whenever $1_u$ is cocartesian.
\end{enumerate}
\end{Thm}
\begin{proof}
Obviously, the functor $F$ reflects isomorphisms iff all $\Delta_I$ reflect
isomorphisms.

For a morphism $u : w \to v$ in $\B/I$ (i.e.\ $w = v \circ u$) we have
\begin{diagram}
1_K          & \rTo^{\varphi_w}     &   \Delta_I(w) \\
\dTo^{1_u}   &                      &   \dTo_{\Delta_I(u)}  \\
1_J        & \rTo_{\varphi_v}       &   \Delta_I(v)           \\
\end{diagram}
where $\varphi_w$ and $\varphi_v$ are cocartesian over $w : K \to I$ and
$v : J \to I$, respectively, and $\Delta(u)$ is vertical over $I$.  
As internal sums in $P_F$ are stable and disjoint it follows 
from Lemma~\ref{MoensLem} that $\Delta_I(u)$ is an isomorphism iff $1_u$
is cocartesian. Thus, the functor $\Delta_I$ reflects isomorphisms iff $u$
is an isomorphism whenver $1_u$ is cocartesian.

Thus, the functor $F$ reflects isomorphisms iff it holds for all maps $u$ 
in $\B$ that $u$ is an isomorphism whenever $1_u$ is cocartesian. 
\end{proof}

\bigskip
A geometric morphisms $F \dashv U$ between toposes is known to be surjective 
iff $F$ is faithful. One easily sees that a finite limit preserving functor 
$F : \B \to \C$ between categories with finite limits is faithful iff for
the associated fibration $P_F$ it holds for $u,v : J \to I$ that $u = v$ 
whenever $\varphi_I \circ 1_u = \varphi_I \circ 1_v$ where
$\varphi_I : 1_I \to \coprod_I 1_I$ is cocartesian over $I \to 1$.
But, of course, in this general case $F$ being faithful does not imply that
$F$ reflects isos, e.g.\ if $\B$ is posetal then $F$ is always faithful but
in general does not reflect isos. However, if $F$ reflects isos then it is 
also faithful since $F$ preserves equalizers.

\subsection{Connected Geometric Morphisms}

\begin{Thm}\label{gmprop3}
Let $F \dashv U : \C \to \B$ be a geometric morphism and $P$ be its induced
geometric fibration $P_F$. Then the following conditions are equivalent.
\begin{enumerate}
\item[\rm (1)]
The geometric morphism $F \dashv U$ is \emph{connected}, 
i.e.\ $F$ is full and faithful.
\item[\rm (2)] 
The right adjoint $G$ of $1 : \B \to \C{\downarrow}F$ 
sends cocartesian arrows to isomorphisms.
\item[\rm (3)] The fibered functor $\Gamma$ is cocartesian, i.e.\ preserves
cocartesian arrows.
\end{enumerate}
\end{Thm}
\begin{proof}
Obviously, the functor $F$ is full and faithful iff all $\Delta_I$ are full 
and faithful, i.e.\ all $\widetilde{\eta}_u$ are isomorphisms. 
Let $u : I \to J$ be a morphism in $\B$. Then we have
\begin{diagram}
1_{G(1_I)}                & \rEqual^{\varepsilon_{1_I}}  &  1_I \\
\dTo^{1_{\widetilde{\eta}_u}} &                   & \dTo_{\varphi_u} \\
1_{G(\Delta(u))}  & \rTo_{\varepsilon_{\Delta(u)}}  &  \Delta(u) \\       
\end{diagram}
where $\varphi_u : 1_I \to \Delta(u)$ is cocartesian over $u$ and, therefore, 
we have $\widetilde{\eta}_u = G(\varphi_u)$.
Thus, the functor $F$ is full and faithful iff $G(\varphi)$ is an isomorphism
for cocartesian $\varphi$ whose source is terminal in its fiber. But then
$G$ sends all cocartesian arrows to isomorphisms which can be seen as follows.
Suppose $\varphi : X \to Y$ is cocartesian over $u : I \to J$. 
Let  $\varphi_u : 1_I \to \Delta(u)$ be cocartesian over $u$. 
Then by Lemma~\ref{MoensLem} the commuting square

\begin{diagram}
X  \SEpbk & \rTo^{\varphi}  &  Y\\
\dTo^{\alpha} &             & \dTo_{\beta} \\
1_I  & \rTo_{\varphi_u}  &  \Delta(u) \\       
\end{diagram}
with $\alpha$ and $\beta$ vertical over $I$ and $J$, respectively, 
is a pullback. As $G$ is a right adjoint it preserves pullbacks and, therefore,
\begin{diagram}
G(X)  \SEpbk & \rTo^{G(\varphi)}  &  G(Y) \\
\dTo^{G(\alpha)} &             & \dTo_{G(\beta)} \\
G(1_I)  & \rTo_{G(\varphi_u)}  &  G(\Delta(u)) \\       
\end{diagram}
is a pullback, too, from which it follows that $G(\varphi)$ is an isomorphism 
as $G(\varphi_u)$ is an isomorphism by assumption.
Thus, we have shown the equivalence of conditions (1) and (2). 

The equivalence of conditions (2) and (3) can be seen as follows.
From (inspection of) the proof of Theorem~\ref{gmthm1} we know that for 
$\varphi : X \to Y$ its image under $\Gamma$ is given by
\begin{diagram}
G(X)        &   \rTo^{G(\varphi)}        &  G(Y)                \\     
\dTo^{P(\varepsilon_X)}  &  \Gamma(\varphi)  &  \dTo_{P(\varepsilon_Y)} \\
P(X)        &   \rTo_{P(\varphi)}        &  P(Y)              \\          
\end{diagram}
Thus $\Gamma(\varphi)$ is cocartesian iff $G(\varphi)$ is an isomorphism.
Accordingly, the functor $G$ sends all cocartesian arrows to isomorphisms 
iff $\Gamma$ preserves cocartesianness of arrows.
\end{proof}

\bigskip
Notice that condition (2) of Theorem~\ref{gmprop3} is equivalent to the 
requirement that $G$ inverts just cocartesian arrows over terminal projections
which can be seen as follows. Suppose $\varphi : X \to Y$ is cocartesian 
over $u : I \to J$. Let $\psi : Y \to \coprod_J Y$ be a cocartesian arrows 
over $!_J : J \to 1$. Then $\psi \circ \varphi$ is cocartesian over 
$!_I : I \to 1$. As $G(\psi \circ \varphi) = G(\psi) \circ G(\varphi)$
and by assumption $G(\psi \circ \varphi)$ and $G(\psi)$ are isomorphisms 
it follows immediately that $G(\varphi)$ is an isomorphism, too. Moreover,
one easily sees that $G$ inverts cocartesian arrows over terminal projections 
if and only if $G$ inverts cocartesian arrows above terminal projections whose
source is terminal in its fiber. Of course, this condition is necessary. For
the reverse direction suppose that $\varphi : X \to Y$ is cocartesian over
$!_I : I \to 1$. Then since $P$ is a geometric fibration we have

\begin{diagram}
X \SEpbk & \rTo^{\varphi}_\cocart & Y \\
\dTo^{\alpha} & & \dTo_{\beta} \\
1_I & \rTo_{\varphi_I}^\cocart & \Delta(I)
\end{diagram}
where $\varphi_I$ is cocartesian over $I \to 1$ and $\alpha$ and $\beta$ are
the unique vertical arrows making the diagram commute. As $G$ is a right 
adjoint it preserves pullbacks and, therefore, we have
\begin{diagram}
G(X) \SEpbk & \rTo^{G(\varphi)} & G(Y) \\
\dTo^{G(\alpha)} & & \dTo_{G(\beta)} \\
G(1_I) & \rTo_{G(\varphi_I)} & G(\Delta(I))
\end{diagram}
from which it follows that $G(\varphi)$ is an isomorphism as $G(\varphi_I)$
is an isomorphism by assumption. Thus, a geometric fibration $P$ is connected
if and only if $G$ inverts all cocartesian arrows over terminal projections
which start from a fiberwise terminal object, i.e.\ if for all 
$\sigma : 1_J \to \Delta(I)$ there exists a unique $u : J \to I$ with
$\sigma = \varphi_I \circ 1_u$ as gets immediate from the following diagram
\begin{diagram}
1_{G1_I} & \rEqual^{\;\varepsilon_{1_I}}  & 1_I \\
\dTo_\cong^{1_{G\varphi_I}} & & \dTo_{\varphi_I} \\
1_{G \Delta(I)} & \rTo_{\;\;\varepsilon_{\Delta(I)}} & \Delta(I)\\
\end{diagram}
with $\varphi_I : 1_I \to \Delta(I)$ cocartesian over $I \to 1$.
Analogously, faithfulness of $F$ is equivalent to the requirement 
that $u = v$ whenever $\varphi_I \circ 1_u = \varphi_I \circ 1_v$ 
providing an alternative characterisation of surjectivity for geometric 
morphisms between toposes (as $F \dashv U$ is surjective iff $F$ is faithful).

Obviously, condition (1) of Theorem~\ref{gmprop3} is equivalent to the 
requirement that $\eta : \Id_\B \to UF$ is a natural isomorphism.
For the particular case of a geometric morphism $\Delta \dashv \Gamma : 
\E \to \Set$ where $\E$ is a topos one easily sees that 
$\eta_I : I \to \Gamma\Delta I$ (sending $i \in I$ to the injection 
$\inj_i : 1 \to \coprod_{i \in I} 1$) is a bijection for all sets $I$  iff
the terminal object of $\E$ is \emph{indecomposable} in the sense that for all 
subterminals $U$ and $V$ with $U{+}V \cong 1_\E$ either $U$ or $V$ is
isomorphic to $0_\E$. 

\subsection{Hyperconnected Geometric Morphisms}

For various characterizations of hyperconnected geometric morphisms 
between toposes see A.4.6 of \cite{Ele}. We concentrate here on those
which can be reformulated as palatable properties of the associated geometric
fibrations.

\begin{Thm}\label{hypconn1}
Let $F \dashv U : \C \to \B$ be a geometric morphism and $P$ be the induced
geometric fibration $P_F$. Then the following conditions are equivalent.
\begin{enumerate}
\item[\rm (1)]
The geometric morphism $F \dashv U$ is \emph{hyperconnected}, i.e.\
connected and all components of the counit of the adjunction are monic. 
\item[\rm (2)] 
The fibered functor $\Gamma$ preserves cocartesian arrows and all counits
$\widetilde{\varepsilon}_X : \Delta\Gamma X \to X$ are vertical monos.
\end{enumerate}
\end{Thm}

For geometric morphisms $\Delta \dashv \Gamma : \E \to \Set$ the counit
$\varepsilon_X : \Delta\Gamma X \to X$ at $X \in \E$ is monic
iff distinct global elements of $X$ are disjoint. 
The above Theorem~\ref{hypconn1} generalizes this condition to arbitrary bases.

But there is a different characterization of hyperconnected geometric
morphisms between elementary toposes whose fibrational analogue might 
appear as more intuitive.

\begin{Thm}\label{hypconn2}
Let $F \dashv U : \C \to \B$ be a geometric morphism and $P$ be the induced
geometric fibration $P_F$. Then the following conditions are equivalent.
\begin{enumerate}
\item[\rm (1)] 
The geometric morphism $F \dashv U$ is hyperconnected, i.e.\ 
for all $I \in \B$ the functor $F_{/I} : \B/I \to \C/FI$ restricts 
to an equivalence between $\Sub_\B(I)$ and $\Sub_\C(FI)$.
\item[\rm (2)] 
The fibered functor $\Delta : P_\B \to P_F$ restricts to a fibered 
equivalence between $\Sub_\B$ and $F^*\Sub_\C$ considered as full subfibrations 
of $P_\B$ and $P_F$, respectively.
\end{enumerate}
\end{Thm}

Thus, a geometric morphism is hyperconnected iff the corresponding  
fibered adjunction $\Delta \dashv \Gamma$ between $P_\B$ and $P_F$ 
restricts to a fibered equivalence between their subterminal parts.

Of course, the conditions (2) of theorems \ref{hypconn1} and \ref{hypconn2}, 
respectively, are in general not equivalent for geometric fibrations 
where $\B$ is not a topos and $P_F$ is not a fibration of toposes
(i.e.\ $\C$ is not a topos).

\subsection{Local Geometric Morphisms}

\begin{Thm}\label{gmprop4}
Let $F \dashv U : \C \to \B$ be a geometric morphism and $P$ be the induced
geometric fibration $P_F$. Then the following conditions are equivalent.
\begin{enumerate}
\item[\rm (1)]
The geometric morphism $F \dashv U$ is \emph{local}, 
i.e.\ $F$ is full and faithful and $U$ has a right adjoint.
\item[\rm (2)] 
The fibered functor $\Gamma$ has a fibered right adjoint $\nabla$.
\end{enumerate}
\end{Thm}
\begin{proof}
First we show that (2) implies (1). If $\Gamma$ has a fibered right adjoint 
$\nabla$ then $\Gamma$ preserves cocartesian arrows as it is a fibered left
adjoint. Thus, by the previous Theorem~\ref{gmprop3} it follows that $F$
is full and faithful. As $\Gamma_1 \dashv \nabla_1$ and $U \cong \Gamma_1$
it follows that $U$ has a right adjoint.

Now we show that (1) implies (2). If $F$ is full and faithful then the
unit $\eta : \Id_\B \to UF$ is an isomorphism. Therefore, the fibered functor
$\Gamma$ acts on objects and morphisms simply by applying the functor $U$
and then postcomposing with the inverse of $\eta$, 
i.e.\ $\Gamma(a) = \eta_I^{-1} \circ U(a)$ 
for $a : A \to FI$ in $\C{\downarrow}F$. Then $\Gamma$ has a fibered right 
adjoint $\nabla$ with $\nabla(v) = \phi_J^*R(v)$ for $v : K \to J$ 
in $\B{\downarrow}\B$ where $\phi_J : FJ \to RJ$ is the transpose (w.r.t.\
$U \dashv R$) of $\eta^{-1}_J : UFJ \to J$ as follows from the natural
1-1-correspondence between
\begin{diagram}
UA & & \rTo^w & & K & & A & & \rTo^{\check{w}} & & RK \\
\dTo^{Ua} & & & & \dTo_v & \qquad\mbox{and}\qquad & \dTo^{a} & & & & 
\dTo_{Rv} \\
UFI & \rTo_{\eta^{-1}_I} & I & \rTo_u & J & & FI & \rTo_{Fu} & FJ & 
\rTo_{\phi_J}& RJ \\
\end{diagram}
exploiting the fact that the transpose of
$u \circ \eta^{-1}_I = \eta^{-1}_J \circ UFu$ is $\phi_J \circ Fu$.
\end{proof}

\subsection{Locally Connected or Molecular Geometric Morphisms}

are geometric morphisms whose inverse image part has a fibered left adjoint
which requirement can be reformulated more elementarily as in the following
\begin{Thm}\label{gmprop5}
Let $F \dashv U : \C \to \B$ be a geometric morphism and $P$ be the induced
geometric fibration $P_F$. Then the following conditions are equivalent.
\begin{enumerate}
\item[\rm (1)]
The geometric morphism $F \dashv U$ is \emph{locally connected}, 
i.e.\ $F$ has a left adjoint $L$ such that
\begin{diagram}
B \SEpbk & \rTo^{f} & A & & &  & LB \SEpbk & \rTo^{Lf} & LA \\
\dTo^{b} & & \dTo_{a} &&\mbox{implies}&& \dTo^{\widehat{b}} &&\dTo_{\widehat{a}}\\
FJ& \rTo_{Fu} & FI &    & & & J & \rTo_{u} & I \\
\end{diagram}
where $\widehat{a}$ and $\widehat{b}$ are the upper transposes of $a$ and $b$,
respectively.
\item[\rm (2)] 
The fibered functor $\Delta$ has a fibered left adjoint $\Pi$.
\end{enumerate}
\end{Thm}
\begin{proof}
If $L \dashv F$ then $\Delta$ has an ordinary left adjoint $\Pi_L$ sending
\begin{diagram}
B         & \rTo^{f} & A & &&  & LB  & \rTo^{Lf} & LA \\
\dTo^{b} & & \dTo_{a} &&\mbox{to}&& \dTo^{\widehat{b}} &&\dTo_{\widehat{a}}\\
FJ & \rTo_{Fu} & FI &  & & &  J & \rTo_{u} & I \\
\end{diagram}
and satisfying $P_{\B} \circ \Pi_L = P_F$. Obviously, this functor
$\Pi_L$ is cartesian iff $L$ satisfies the requirement of condition (1).
Thus, condition (1) entails condition (2).

On the other hand if $\Delta$ has a fibered left adjoint $\Pi$ then 
$F \cong \Delta_1$ has an ordinary left adjoint $L \cong \Pi_1$ and as 
$\Pi \cong \Pi_L$ in the 2-category $\Cat{\downarrow}\B$ with vertical natural 
transformations as 2-cells (because both functors are left adjoints to $\Delta$
in this 2-category) it follows that $\Pi_L$ is also cartesian and, therefore, 
the functor $L$ satisfies the requirement of condition (1). 
Thus, condition (2) entails condition (1).
\end{proof}

\medskip
Fibrations satisfying the equivalent conditions of the previous theorem
were originally called ``molecular'' since when $\B$ is $\Set$ one easily 
sees that $L$ associates with $A$ its set $LA$ of connected components the 
family of which is represented by $\eta_A : A \to FLA$ and one may think 
of the connected components of an object as the ``molecules'' it is made of. 

Moreover, it follows from the fibered version of the Special Adjoint Functor 
Theorem\footnote{which applies as $P_\Se$ has a small generating family and, 
therefore, also a small cogenerating family (as shown by R.~Par\'e and 
D.~Schumacher)} that a geometric morphism $F \dashv U : \E \to \Se$ between 
toposes is locally connected if and only if $F$ preserves the locally 
cartesian closed structure as $\Delta : P_\Se \to P_F$ preserves
internal limits iff $F$ preserves (finite limits and) dependent products.
Obviously, the latter condition is equivalent to the requirement that
$F_{/I} : \Se/I \to \E/FI$ preserves exponentials for all $I \in \Se$.

Thus, summarizing we observe that for a geometric morphism 
$F \dashv U : \E\to\Se$ the following are equivalent
\begin{enumerate}
\item[(1)] $F \dashv U$ is locally connected
\item[(2)] $F$ preserves dependent products 
\item[(3)] $F_{/I} : \Se/I \to \E/FI$ preserves exponentials for all $I\in\Se$
\end{enumerate}
as formulated in Proposition~C.3.3.1 of \cite{Ele}

\subsubsection{Connected Locally Connected Geometric Morphims} 

Next we characterize those locally connected geometric morphisms which are 
moreover connected.
\begin{Lem}\label{clclm}
A locally connected geometric morphism $F \dashv U : \C \to \B$ is connected
iff the left adjoint $L$ of $F$ preserves terminal objects.
\end{Lem}
\begin{proof}
The forward direction is immediate since from $LF \cong \Id_\B$ and 
preservation of terminal objects by $F$ it follows that $L 1_\C \cong
L F 1_\B \cong 1_\B$.

For the backwards direction suppose that $L$ preserves terminal objects.
Consider the square
\begin{diagram}[small]
FI  & \rTo^{!_{FI}} & 1_\C \\
\dEqual & & \dTo \\
FI & \rTo_{F !_I} & F 1_\B
\end{diagram}
which is a pullback since the downward arrows in it are isomorphisms.
Thus, since $F \dashv U$ is assumed as locally connected the square
\begin{diagram}[small]
LFI  & \rTo^{L !_{FI}} & L 1_\C \\
\dTo^{\varepsilon_I} & & \dTo \\
I & \rTo_{!_I} & 1_\B
\end{diagram}
is a pullback, too, from which it follows that $\varepsilon_I$ is an
isomorphism since $L 1_\C \to 1_\B$ is an isomorphism due to the assumption 
that $L$ preserves terminal objects.
\end{proof}

\medskip
Next we show that for connected geometric morphisms being locally connected 
can be expressed in terms of existence and the requirement of simple 
preservation properties of a further left adjoint.
\begin{Thm}\label{clcpresprop}
A connected geometric morphism $F \dashv U : \C\to\B$ is locally connected
iff $F$ has a left adjoint $L$ which sends pullbacks of cospans with one
leg in the image of $F$ to pullbacks in $\B$.
\end{Thm}
\begin{proof}
Suppose $F \dashv U : \C\to\B$ is connected and $L \dashv F$. Then 
for $u : J \to I$ in $\B$ and pullbacks
\begin{diagram}
B \SEpbk & \rTo^f & A \\
\dTo^b & & \dTo_a \\
FJ & \rTo_{Fu} & FI
\end{diagram}
in $\C$ we have
\begin{diagram}
LB \SEpbk& \rTo^{Lf} & LA & & LB \SEpbk& \rTo^{Lf} & LA \\
\dTo^{Lb} & & \dTo_{La} & \qquad\quad\mbox{iff}\qquad\quad & 
\dTo^{\widehat{b}} & & \dTo_{\widehat{a}} \\
LFJ & \rTo_{LFu} & LFI & & J & \rTo_u & I\\
\end{diagram} 
since $\widehat{a} = \varepsilon_I \circ La$ and 
$\widehat{b} = \varepsilon_J \circ Lb$ and both $\varepsilon_I$ and
$\varepsilon_J$ are isomorphisms.
\end{proof}

\subsection{Atomic Geometric Morphisms}

A geometric morphism $F \dashv U : \E \to \Se$ between toposes
is called \emph{atomic} iff $F : \Se \to \E$ is logical. Since logical functors
preserve dependent products atomic geometric morphisms between toposes are in 
particuar also locally connected. 
Atomic geometric morphisms can be characterised as those locally connected 
geometric morphisms $F \dashv U : \E\to\Se$ where all monomorphisms $m$ 
in $\E$ are $\Se$-definable, i.e.\ satisfy
\begin{diagram}
X \SEpbk & \rTo^{\eta_X} & F L X \\
\dEmbed^{m} & & \dTo_{F L m} \\
Y  & \rTo_{\eta_Y} & F L Y \\
\end{diagram}
where $\eta$ is the unit of $L \dashv F$. This can be seen as follows. Recall
that for a locally connected geometric morphism $F \dashv U : \E\to\Se$ 
the monomorphism $F(\top_\Se)$ classifies $\Se$-definable monomorphisms.
Now if $F$ is logical then $F(\top_\Se)$ is a subobject classifier in $\E$
and, therefore, all monomorphisms in $\E$ are $\Se$-definable. On the other
hand, if all monomorphisms in $\E$ are $\Se$-definable then $F(\top_\Se)$
is a subobject classifier (as it classifies all $\Se$-definable monomorphisms)
and thus $F$ is logical.

When the base topos $\Se$ is $\Set$ the geometric morphism from $\E$ to $\Set$
is atomic iff it is molecular and, moreover, the subobjects of $A$ correspond
to subsets of $LA$, i.e.\ $\Sub_\E(A)$ is isomorphic to $\Pow(LA)$ which is
an atomic lattice for which reason the connected components of $A$ may be
considered as the ``atoms'' from which $A$ is built of. Atomic presheaf toposes 
over $\Set$ are up to equivalence of the form $\Set^{\Gbb^\op}$ for some small 
groupoid $\Gbb$  whereas the Sierpi\'nski topos $\Set^{\two^\op}$ is locally 
connected but not atomic over $\Set$ since its terminal object is a molecule 
but not an atom since it contains a proper subobject which is not initial.

\newpage
\appendix

\section{M.~Jibladze's Theorem on Fibered Toposes}\label{jibthm}

Let $\B$ be a category with finite limits. A \emph{topos fibered over $\B$}
is a fibration $P : \X \to \B$ all whose fibers are toposes and all whose
reindexing functors are logical. In \cite{Jib} M.~Jibladze has shown that 
if a fibered topos has internal sums then these sums are necessarily universal 
(i.e.\ pullback stable) and disjoint. 
Thus, by Moens's Theorem it follows that $P \simeq P_\Delta$ 
where $\Delta : \B \to \E = \X_1 : I \mapsto \coprod_I 1_I$.

As a preparation we need the following results about logical functors 
$F : \E \to \F$ between toposes. If $L \dashv F$ then (by A.2.4.8 of \cite{Ele}) 
functor $L$ preserves monos and (by A.2.3.8 of \cite{Ele}) 
the following are quivalent
\begin{enumerate}
\item[(1)] $L_{/1} : \F_{/1} \to \E_{/L1}$ is an equivalence\footnote{the right
adjoint of $L_{/1}$ is given by $\eta_1^* \circ F_{/ L1}$}
\item[(2)] $L$ is faithful
\item[(3)] $L$ preserves equalizers
\item[(4)] $L$ preserves pullbacks.
\end{enumerate}
One easily shows\footnote{Suppose $\Sigma_I : \E_{/I} \to \E$ is full 
and faithful. Then the unit $\eta$ of $\Sigma_I \dashv I^*$ is an isomorphism.
For $a : A \to I$ we have $\eta_a = \langle a , \id_A \rangle : a \to I^*\Sigma_Ia$
as depicted in 
\begin{diagram}[small]
A  &            & & & \\
&\rdTo~{\eta_a} \rdTo(2,4)_a \rdEqual(4,2) &&&\\
  &            &  I{\times}A \SEpbk & \rTo^{\pi_2} & A \\
  &            &  \dTo_{\pi_1}     &      &  \dTo \\
  &            &  I        & \rTo &   1 \\
\end{diagram}
Since $\eta_a$ is an isomorphism the projection $\pi_2 : I \times A \to A$ 
is an isomorphism, too. Thus, for $a^\prime : A \to I$ we have 
$\pi_2 \circ \langle a , \id_A \rangle = \id_A = 
 \pi_2 \circ \langle a^\prime , \id_A \rangle$
from which it follows that $a = a^\prime$ since $\pi_2$ is an isomorphism. 
Thus $I$ is subterminal. But then the counit $\varepsilon_A$ at $A$ is given by
\begin{diagram}[small]
I^*A \SEpbk & \rEmbed^{\varepsilon_A} & A \\
\dTo & & \dTo \\
I & \rEmbed & 1
\end{diagram}
and thus monic.}
that if $L$ is full and faithful then $L1$ is subterminal 
from which it follows that all components of the counit $\varepsilon$ are monos. 

\medskip
We are now ready to prove Jibladze's Theorem on Fibered Toposes.

\begin{proof}
First observe that for a mono $m : J \mono I$ we have $m^*\coprod_m \cong 
\Id_{P(J)}$ as follows from the Beck--Chevalley condition for internal sums 
at the pullback square
\begin{diagram}[small]
J \SEpbk & \rEqual & J \\
\dEqual & & \dEmbed_m \\
J & \rEmbed_m & I
\end{diagram}
Since $m^*$ is logical and $\coprod_m$ is full and faithful all components
of the counit of $\coprod_m \dashv m^*$ are monic.

Next we show that for all $u : J \to I$ in $\B$ and $X \in P(J)$ the map
$\eta_X : X \to u^*\coprod_u X$ is monic (where $\eta$ is the counit of
$\coprod_u \dashv u^*$). Recall that $\eta_X$ is the unique vertical map
such that
\begin{diagram}
X & \rTo_{\cocart}^{\varphi_u(X)} & \coprod_u X  \\
\dTo^{\eta_X} & \ruTo^{\cart}_{\theta_u(X)} & \\
u^*\coprod_u X
\end{diagram}
Let $k_0,k_1 : K \to J$ be a kernel pair of $u$ in $\B$ and $d_u : J \to K$
with $k_0 d_u = \id_J = k_1 d_u$. Consider the diagram
\begin{diagram}
d_u^*k_1^*X & \rTo_\cart^\psi & k_1^*X \SEpbk & \rTo_\cart^\theta & X \\ 
\dEqual & & \dTo_\varphi & & \dTo_{\varphi_u(X)} \\
X & \rTo_{\eta_X} & u^*\coprod_u X & \rTo_{\theta_u(X)}^\cart & \coprod_u X
\end{diagram}
with $\theta \circ \psi = \id_X$. Notice that $\varphi$ is cocartesian by the 
Beck--Chevalley condition for internal sums.
Next consider the diagram
\begin{diagram}
d_u^*k_1^*X & & \\
\dTo^{\widetilde{\varphi_1}} & \rdTo^\psi & \\
\coprod_{d_u}d_u^*k_1^*X & \rEmbed_{\varepsilon_{k_1^*X}} & k_1^*X \\
\dTo^{\widetilde{\varphi_2}} & & \dTo_\varphi \\
\coprod_{k_0}\coprod_{d_u}d_u^*k_1^*X = X &  \rEmbed_{\eta_X} &  u^*\coprod_u X \\
\end{diagram}
where $\varphi$ and $\widetilde{\varphi_2}$ are cocartesian over $k_0$ and
$\widetilde{\varphi_1}$ is cocartesian over $d_u$. Since $d_u$ is monic
the map $\varepsilon_{k_1^*X}$ is monic. Since $\coprod_{k_0}$ is left 
adjoint to the logical functor $k_0^*$ it preserves monos from which 
it follows that $\eta_X$ is monic.

Now since all components of the counit $\eta$ of $\coprod_u \dashv u^*$ are monic
it follows that $\coprod_u$ is faithful. Since $u^*$ is logical it follows
that ${\coprod_u}_{/ 1_J}$ is an equivalence. Recall that its right adjoint
is given by $\eta_{1_J}^* \circ {u^*}_{/ \coprod_u 1_J}$, i.e.\ pullback 
along the cocartesian arrow $\varphi_u : 1_J \to \coprod_u 1_J$. 

That the counit of the adjunction ${\coprod_u}_{/ 1_J} \dashv \varphi_u^*$ 
is an isomorphism means that for vertical $\alpha$ in the pullback
\begin{diagram}
\varphi_u^* X \SEpbk & \rTo^{\alpha^*\varphi_u} & X \\
\dTo^{\varphi_u^*\alpha} & & \dTo_\alpha \\
1_J & \rTo^\cocart_{\varphi_u} & \coprod_u 1_J
\end{diagram}
the top arrow $\alpha^*\varphi_u$ is cocartesian. This is sufficient 
for showing that cocartesian arrows are stable under pullbacks along 
vertical arrows, i.e.\ that internal sums are universal (since by the 
Beck--Chevalley condition cocartesian arrows are stable under pullbacks 
along cartesian arrows anyway).

That the unit of the adjunction ${\coprod_u}_{/ 1_J} \dashv \varphi_u^*$ 
is an isomorphism means that
\begin{diagram}
X & \rTo_\cocart & \coprod_u X \\
\dTo^{!_X} & & \dTo_{\coprod_u !_X} \\
1_J & \rTo^\cocart_{\varphi_u} & \coprod_u 1_J
\end{diagram}
is a pullback. From this it follows that all diagrams of the form
\begin{diagram}
X & \rTo_\cocart & \coprod_u X \\
\dTo^{\alpha} & & \dTo_{\coprod_u \alpha} \\
Y & \rTo^\cocart & \coprod_u Y
\end{diagram}
are pullbacks. But (from the proof of Moens's Theorem) this is known 
to imply disjointness of internal sums.\footnote{One can see this more easily
as follows. Since ${\coprod_u}_{/ 1_J}$ is an equivalence it follows that
$\coprod_u$ reflects isomorphisms which is known (from the proof of Moens's
Theorem) to entail that internal sums are disjoint provided they are 
universal.} 
\end{proof}

\bigskip
Though claimed otherwise in \cite{Ele} Jibladze's Theorem was not 
proved in Moens's Th\'ese \cite{Moe} from 1982. Johnstone claims that
Moens proved in some other way that for a fibered topos internal sums are
universal and disjoint. But this is not the case because he considered
fibered variants of Giraud's Theorem where internal sums are \emph{assumed} 
as universal and disjoint. The only known way of showing that for a fibered 
topos internal sums are universal and disjoint is via Jibladze's Theorem.

However, in Jibladze's original formulation he did not prove universality
and disjointness for internal sums in a fibered topos. For him it was
sufficient to show that all $\coprod_u$ are faithful because from this
it follows that the adjunctions $\coprod_u \dashv u^*$ are (equivalent to
ones) of the form $\Sigma_A \dashv A^*$ for some $A$ in $\X_1$ and this
is sufficient for showing that $P \simeq P_\Delta$.

\newpage
\section{Descent and Stacks}\label{DecSta}

Let $\mathfrak{C}$ be a 2-category, $x$ a 0-cell in $\mathfrak{C}$ and
$f : z \to y$ a 1-cell in $\mathfrak{C}$ then $f$ is called a
\emph{descent} map w.r.t.\ $x$ iff the functor
$ \mathfrak{C}(f,x) : \mathfrak{C}(y,x) \to \mathfrak{C}(z,x) $
is an equivalence of (ordinary) categories.
In particular this definition applies to 2-categories $\Fib(\B)$
for arbitrary ordinary categories $\B$.
\begin{Def}\label{DescentMap}
  For $P \in \Fib(\B)$ a \emph{descent map w.r.t.\ $P$} is a cartesian
  functor $F : Q^\prime \to Q$ over $\B$ such that
  \[ \Fib(\B)(F,P) :  \Fib(\B)(Q,P) \to \Fib(\B)(Q^\prime,P) \]
  is an equivalence of (ordinary) categories. 
\end{Def}
If $u : J \to I$ is a morphism in $\B$ we write $K_u$ for the
posetal groupoid fibered over $\B$ whose morphisms over $K$ are
pairs $(v_1,v_2)$ of morphisms from $K$ to $J$ with $uv_1 = uv_2$
for which reindexing along $w : L \to K$ is given by $(v_1w , v_2w)$.
We write $Q_u$ for the cartesian functor from $K_u$ to $\underline{I}$
sending $(v_1,v_2)$ to $uv_1 = uv_2$.
This cartesian functor $Q_u : K_u \to \underline{I}$
factors through the discrete subfibration
$i_u : S_u \hookrightarrow \underline{I}$
consisting of all maps of the form $uv$ in $\B$ via a (unique) cartesian
functor $E_u : K_u \to S_u$. Obviously $E_u : K_u \to S_u$ is an
equivalence in the 2-category $\Fib(\B)$ for which reason
$i_u$ is a descent map w.r.t.\ $P$, i.e.\
$\Fib(\B)(i_u,P) : \Fib(\B)(\underline{I},P) \to \Fib(\B)(S_u,P)$ is
an equivalence of ordinary categories, iff
$\Fib(\B)(Q_u,P) : \Fib(\B)(\underline{I},P) \to \Fib(\B)(K_u,P)$
is an equivalence of ordinary categories.

Traditionally, one writes $\mathsf{Des}_u(P)$ for $\Fib(\B)(K_u,P)$ and
calls it the category of ``descent data for $P$ w.r.t.\ $u$'' and says
that ``$u$ is is a descent map w.r.t.\ $P$'' iff $\mathsf{Des}_u(P)$ is 
equivalent to $P(I) \simeq  \Fib(\B)(\underline{I},P)$ 
via $\Fib(\B)(Q_u,P)$.\footnote{The celebrated B\'enabou-Roubaud Theorem
from 1970 characterizes descent maps for fibrations $P$ with internal sums
over a base category $\B$ with pullbacks as those maps $u : J \to I$ in $\B$
for which $u^* : P(I) \to P(J)$ is monadic.

Its proof is based on a lemma saying that $\mathsf{Des}_u(P)$ is equivalent 
to the category of algebras for the monad induced by the adjunction 
$\coprod_u \dashv u^*$.} 
Thus $u$ is a descent map w.r.t.\ $P$ in this traditional sense iff 
$i_u$ is a descent map w.r.t.\ $P$ in the sense of Def.~\ref{DescentMap}.

Using the notion of descent map one easily defines what is a
$\mathfrak{J}$-stack for a Grothendieck topology $\mathfrak{J}$ on $\B$.
\begin{Def}\label{J-Stack}
 A \emph{$\mathfrak{J}$-stack} is a fibration $P \in \Fib(\B)$ such that
 for every $S \in \mathfrak{J}(I)$ the inclusion
 $i_S : S \hookrightarrow \underline{I}$ is a descent map w.r.t.\ $P$. 
\end{Def}
Obviously, a discrete fibration over $\B$ is a $\mathfrak{J}$-stack
iff the corresponding presheaf over $\B$ is a $\mathfrak{J}$-sheaf.

\newpage
\section{Precohesive Geometric Morphisms}\label{precohesive}

In his work on ``axiomatic cohesion'' Lawvere calls a topos $\E$ over $\Set$
\emph{precohesive} iff $\E$ is 2-valued, i.e.\ $\Gamma$ preserves subobject 
classifiers, $\Delta : \Set\to\E$ has a left adjoint $\Pi$ and 
$\Gamma : \E\to\Set$ has a right adjoint $\nabla$. 

Let $\Se$ be an arbitrary base topos. By A.4.6 of \cite{Ele}
a geometric morphism $F \dashv U : \E\to\Se$ is hyperconnected iff $U$ 
preserves subobject classifiers. Thus, it appears as natural to call
a geometric morphism $F \dashv U : \E\to\Se$ \emph{precohesive} iff it is 
hyperconnected, locally connected and local. 

Thus, by Theorem~\ref{clcpresprop} a geometric morphism $F \dashv U : \E\to\Se$ 
is precohesive iff $F$ has a left adjoint $L$ and $U$ has a right adjoint $R$ 
such that $U$ preserves subobject classifiers and $L$ sends pullbacks of 
cospans where one of the legs is in the image of $F$ to pullbacks. 

As shown in \cite{JohPLC} for precohesive geometric morphisms $F \dashv U : 
\E\to\Se$ the left adjoint $L$ to $F$ sends pullbacks of cospans whose 
common codomain is in the image of $F$ to pullbacks, i.e.\ the fibered 
left adjoint $\Pi$ to $\Delta$ preserves binary products in each fiber. 
Thus, a geometric morphism $F \dashv U : \E\to\Se$ is precohesive iff it 
is hyperconnected, local and $F$ has a left adjoint $L$ sending pullbacks 
of cospans with common codomain in the image of $F$ to pullbacks, i.e.\ 
$L_{/FI} : \E/FI \to \Se/LFI \simeq \Se/I$ preserves binary products for 
all $I$ in $\Se$. 

One knows (see e.g.\ Proposition~C.3.3.1 of \cite{Ele}) that for a geometric 
morphism $F \dashv U : \E\to\Se$ the following are equivalent
\begin{enumerate}
\item[(1)] $F \dashv U$ is locally connected
\item[(2)] $F$ preserves dependent products 
\item[(3)] $F_{/I} : \Se/I \to \E/FI$ preserves exponentials for all $I\in\Se$.
\end{enumerate}
Accordingly, a hyperconnected and local geometric morphism 
$F \dashv U : \E\to\Se$ is precohesive iff $F$ preserves dependent products 
iff all slices of $F$ preserve exponentials.

Already in their 1980 paper \emph{Molecular Toposes} introducing locally connected
geometric morphisms Barr and Par\'e proved that for a geometric morphism 
$F \dashv U : \E\to\Se$ its inverse image part $F$ preserves ordinary exponentials
iff $F$ has a left adjoint enriched over $\Se$. As shown in \cite{JohPLC} such a 
further $\Se$-enriched left adjoint preserves also finite products whenever the 
geometric morphism is also local and hyperconnected. Using this result Hemelaer 
and Rogers in their 2020 APCS paper have come up with an example of a local and 
hyperconnected geometric morphism whose inverse image part does not preserve ordinary
exponentials although it has a left adjoint.

Lawvere and Menni in their 2015 TAC paper use the terminology 
\emph{stably precohesive} for what we have called \emph{precohesive}. 
They call a geometric morphism \emph{precohesive} iff it is hyperconnected, 
local and its inverse image part preserves exponentials, i.e.\ has a finite 
product preserving left adjoint.\footnote{The reason why they are interested in 
this presumably weaker notion is that it corresponds to a string of adjoints 
$L \dashv F  \dashv U  \dashv  R : \Se\to\E$
such that $F$ (and thus also $R$) are full and faithful, $L$ preserves finite 
products and the so called ``Nullstellensatz'' holds claiming that for every 
$X \in \E$ the unique map $\theta_X : UX \to LX$ with 
\begin{diagram}
FUX & \rTo^{\;\;\varepsilon_X} & X \\
& \rdTo_{F\theta_X} & \dTo_{\eta_X} \\
 & & FLX
\end{diagram}
is epic. Intuitively, the ``points-to-pieces transform'' $\theta_X$ sends every point 
to the piece in which it lies. Thus, the ``Nullstellensatz'' claims that 
``every piece of $X$ contains a point''.}
Accordingly, in their terminology a geometric morphism is called
stably precohesive iff all its slices are precohesive in their sense. 
It is an open question raised by Lawvere and Menni in \emph{loc.cit.}\
whether all precohesive geometric morphisms are stably precohesive.

We expect the answer to their question to be negative although we have not been 
able to come up with a counterexample so far. One reason is that generally in 
toposes a predicate on an object need not be universally valid even if it holds 
for all its global elements. Moreover, in 2020 R.~Garner and T.~Streicher have 
come up with an example of a (bounded) local geometric morphism which is not 
locally connected though its inverse image part has a finite product preserving
left adjoint.
Thus, this also provides an example of a geometric morphism whose inverse image part 
has a left adjoint which is enriched but not fibered over the base topos, 
i.e.\ whose inverse image part preserves ordinary but not dependent function spaces.
Alas, the geometric morphism in Garner's counterexample is not hyperconnected.
But we do not see how the additional assumption of hyperconnectedness allows one 
to derive preservation of dependent function spaces from preservation of ordinary 
function spaces. 

However, as recently shown by Menni for boolean toposes $\Se$ a geometric morphism 
$F \dashv U : \E\to\Se$ is locally connected whenever $F$ is full and faithful and 
has a left adjoint preserving finite products. 

\end{document}